%% file: Testing_Manifold_Hypothesis_Dec3-013.tex
\documentclass{amsart}[12pt]

\usepackage{amsmath, amssymb, amsthm, graphicx, algorithm, algorithmic, verbatim, pdfsync, color}
\usepackage{euler} \usepackage[T1]{fontenc}
\usepackage[margin=2.5cm]{geometry}

\include{macros}

\begin{document}

\title{Testing the manifold hypothesis}

\author{Charles Fefferman}
\address{Department of Mathematics, Princeton University}
\email{cf@math.princeton.edu}

\author{Sanjoy Mitter}
\address{Laboratory for Information and Decision Systems, MIT}%
\email{mitter@mit.edu}

\author{Hariharan Narayanan}
\address{Department of Statistics and Department of Mathematics, University of Washington}
\email{harin@uw.edu}

\maketitle

\begin{abstract}
The hypothesis that high dimensional data tend to lie in the vicinity of a low dimensional manifold is the basis of manifold learning.
The goal of this paper is to  develop an algorithm (with accompanying complexity guarantees) for testing the existence of a manifold that fits a probability distribution supported  in a separable Hilbert space, only using i.i.d samples from that distribution. More precisely, our setting is the following.
Suppose that data are drawn independently at random from a probability distribution $\PP$ supported on the
unit ball of a separable Hilbert space $\HH$.    Let $\G(d,  V, \tau)$ be  the set of submanifolds of the unit ball of $\HH$ whose volume is
at most $V$
 and reach (which is the supremum of all $r$ such that any point at a distance less than $r$ has a unique nearest point on the manifold) is at least $\tau$.  Let $\LL(\MM, \PP)$ denote
 mean-squared distance of a random point from the probability distribution $\PP$ to $\MM$.
We obtain an algorithm that tests the manifold hypothesis in the following sense.

 The algorithm takes i.i.d random samples from $\PP$ as input, and determines which of the following two is true (at least one must be):
 \ben
 \item There exists $\MM \in \G(d,  CV, \frac{\tau}{C})$ such that $\LL(\MM, \PP) \leq  C \eps.$
 \item There exists no $\MM \in \G(d,  V/C, C\tau)$ such that $\LL(\MM, \PP) \leq   \frac{\eps}{C}.$
 \een
The answer is correct with probability at least $1-\delta$.
\end{abstract}

\tableofcontents
\newpage
\section{Introduction}

We are increasingly confronted with very high dimensional data from speech, images,  and genomes and other sources.
A collection of methodologies for analyzing high dimensional data based on the hypothesis that data tend to  lie near a low dimensional manifold is now called "Manifold Learning".
(see Figure~\ref{fig:torus}) We refer to the underlying  hypothesis as the "manifold hypothesis."  Manifold Learning has been an area of intense activity over the past two decades.  We refer the interested reader to a limited set of papers associated with this field; see \cite{misha, Carlsson, Wasserman, Dasgupta2, donoho,  hastie, Kamb93, kegl, NarNiy, NSW, Marina, LLE, PrincipalManifolds, ISOMAP, MaximumVariance} and the references therein.

 The goal of this paper is to  develop an algorithm that tests the manifold hypothesis.

Examples of low-dimensional manifolds embedded in high-dimensional spaces include: image vectors representing 3D objects under different illumination conditions, and camera views and phonemes in speech signals. The low-dimensional structure typically arises due to constraints arising from physical laws.  A recent empirical study~\cite{Carlsson} of a large number of 3 $\times$ 3 images represented as points in $\mathbb{R}^9$ revealed that they approximately lie on a two-dimensional manifold knows as the Klein bottle.

One of the characteristics of high-dimensional data of the type mentioned earlier is that the number of dimensions is comparable, or larger than, the number of samples.  This has the consequence that the sample complexity of function approximation can grow exponentially.  On the positive side, the data exhibits the phenomenon of ``concentration of measure''~\cite{bb1,bb2} and asymptotic analysis of statistical techniques is possible.  
 Standard dimensional reduction techniques such as Principal Component Analysis and Factor Analysis, work well when the data lies near a linear subspace of high-dimensional space.  They do not work well when the data lies near a nonlinear manifold embedded in the high-dimensional space.

Recently, there has been considerable interest in fitting low-dimensional nonlinear manifolds from sampled data points in high-dimensional spaces.  These problems have been viewed as optimization problems generalizing the projection theorem in Hilbert Space.   One line of research starts with  principal curves/surfaces~\cite{hastie} and topology preserving networks~\cite{ee}.  The main ideas is that information about the global structure of a manifold can be obtained by analyzing the ``interactions'' between overlapping local linear structures.  The so-called Local Linear Embedding method (local PCA) constructs a local geometric structure that is invariant to translation and rotation in the neighborhood of each data point~\cite{ff}.

In another line of investigation~\cite{gg}, pairwise geodesic distances of data points with respect to the underlying manifold are estimated and multi-dimensional scaling is used to project the data points on a low-dimensional space which best preserves the estimated geodesics.  The tangent space in the neighborhood of a data point can be used to represent the local geometry and then these local tangent spaces can be aligned to construct the global coordinate system of the nonlinear manifold~\cite{hh}.

A comprehensive review of Manifold Learning can be found in the recent book \cite{MaFu}.  In this paper, we take a ``worst case'' viewpoint of the Manifold Learning problem.  Let ${\mathcal H}$ be a separable Hilbert space, and let $P$ be a probability measure supported on the unit ball $B_{\mathcal H}$ of ${\mathcal H}$.  Let $| \, \cdot \, |$ denote the Hilbert space norm of ${\mathcal H}$ and for any $x,y \in {\mathcal H}$ let $d(x,y) = |x-y|$.  For any $x \in B_{\mathcal H}$ and any ${\mathcal M} \subset B_{\mathcal H}$, a closed subset, let
$d (x, {\mathcal M}) = \inf_{y \in {\mathcal M}} |x-y|$ and ${\mathcal L} ({\mathcal M},\PP) = \int d(x, {\mathcal M})^2 d\PP(x)$.  We assume that i.i.d data is generated from sampling $\PP$, which is fixed but unknown.  This is a worst-case view in the sense that no prior information about the data generating mechanism is assumed to be available or used for the subsequent development.  This is the viewpoint of modern Statistical Learning Theory~\cite{ii}.

In order to state the problem more precisely, we need to describe the class of manifolds within which we will search for the existence of a manifold which satisfies the manifold hypothesis.

Let ${\mathcal M}$ be a submanifold of $\mathcal H$.  The reach $\tau >0$ of ${\mathcal M}$ is the largest number such that for any $0 < r < \tau$, any point at a distance $r$ of ${\mathcal M}$ has a unique nearest point on ${\mathcal M}$.

Let ${\mathcal G}_e = {\mathcal G}_e (d, V, \tau)$ be the family of $d$-dimensional $\mathcal{C}^2-$submanifolds of the unit ball in $\mathcal H$ with volume $\le V$ and reach $\ge \tau$.

Let $\PP$ be an unknown probability distribution supported in the unit ball of a separable (possibly infinite-dimensional) Hilbert space and let $(x_1, x_2, \ldots)$ be i.i.d random samples sampled from $\PP$.

The test for the Manifold Hypothesis answers the following affirmatively:  Given error $\varepsilon$, dimension $d$, volume $V$, reach $\tau$ and confidence $1-\delta$, is there an algorithm that takes a number of samples depending on these parameters and with probability $1- \delta$ distinguishes between the following two cases (as least one must hold):\\ (a) Whether  there is a
$${\mathcal M} \in {\mathcal G}_e = {\mathcal G}_e (d, CV, \tau/C)$$
such that
$$
\int d (M,x)^2 dP(x) < C\varepsilon \ .
$$

(b) Whether there is no manifold $${\mathcal M} \in {\mathcal G}_e (d, V/C, C\tau)$$
such that
$$
\int d (M,x)^2 dP(x) < \varepsilon/C \ .$$
Here $d(M,x)$ is the distance from a random point $x$ to the manifold ${\mathcal M}$, $C$ is a constant depending only on $d$.

The basic statistical question is:

What is the number of samples needed for testing the hypothesis that data lie near a low-dimensional manifold?

The desired result is that the sample complexity of the task depends only on the ``intrinsic'' dimension, volume and reach, but not the ``ambient'' dimension.

We approach this by considering the Empirical Risk Minimization problem.

Let
$$
{\mathcal L} (M,P) = \int d(x,M)^2 dP (X) \ ,
$$
and define the Empirical Loss
$$
L_{\rm emp} (M) = \frac{1}{s} \sum^s_{i=1} d(x_i, M)^2
$$
where $(x_1, \ldots, x_s)$ are the data points.  The sample complexity is defined to be the smallest $s$ such that there exists a rule ${\mathcal A}$ which assigns to given $(x_1, \ldots, x_s)$  a manifold  ${\mathcal M}_{\mathcal A}$ with the property that if $x_1, \dots, x_s$ are generated i.i.d from $\PP$, then
$$
\mathbb{P} \left[ {\mathcal L} (\MM_A, \PP) - \inf_{\MM\in {\mathcal G}_e} {\mathcal L} (M,\PP) > \varepsilon \right] < \delta.
$$
We need to determine how large $s$ needs to be so that
$$
\mathbb{P} \left[  \sup_{{\mathcal G}_e} \Big| \frac{1}{s} \sum^s_{i=1} d(x_i, \MM)^2 - {\mathcal L} (\MM,\PP) \Big| < \varepsilon\right] >1- \delta.
$$
The answer to this question is given by Theorem 1 in the paper.

The proof of the theorem proceeds by approximating manifolds using point clouds and then using uniform bounds for $k-$means (Lemma~\ref{lem:key} of the paper).


The uniform bounds for $k-$means are proven by getting an upper bound on the Fat Shattering Dimension of a certain function class and then using an integral related to Dudley's entropy integral. The bound on the Fat Shattering Dimension is obtained using a random projection and the Sauer-Shelah Lemma. The use of random projections in this context appears  in  Chapter 4, \cite{MaFu} and \cite{NarMit}, however due to the absence of chaining,  the bounds derived there are weaker.


The Algorithmic question can be stated as follows:

Given $N$ points $x_1, \ldots, x_N$ in the unit ball in $\mathbb{R}^n$, distinguish between the following two cases (at least one must be true):\\
 (a) Whether there is a manifold ${\mathcal M} \in {\mathcal G}_e = {\mathcal G}_e (d, CV, C^{-1}\tau)$ such that
$$
\frac{1}{N} \sum^N_{i=1} d(x_i, M)^2 \le C \varepsilon
$$
where $C$ is some constant depending only on d.\\
(b) Whether there is no manifold ${\mathcal M} \in {\mathcal G}_e = {\mathcal G}_e (d, V/C, C\tau)$ such that
$$
\frac{1}{N} \sum^N_{i=1} d(x_i, M)^2 \le  \varepsilon/C
$$
where $C$ is some constant depending only on d.\\

The key step to solving this problem is to translate the question of optimizing the squared-loss over a family of manifolds to that of optimizing over sections of a disc bundle. The former involves an optimization over a non-parameterized infinite dimensional space, while the latter involves an optimization over a parameterized (albeit infinite dimensional) set.

 We introduce the notion of a cylinder packet in order to define a disc bundle. A cylinder packet  is  a finite collection of cylinders satisfying certain alignment constraints. An ideal cylinder packet corresponding to a $d-$manifold $\MM$ of reach $\tau$ (see Definition~\ref{def:reach}) in $\R^n$ is obtained by taking a net (see Definition~\ref{def:net}) of the manifold and for every point $p$ in the net, throwing in a cylinder  centered at $p$ isometric to $2\bar \tau (B_d \times B_{n-d})$ whose $d-$dimensional central cross-section is tangent to $\MM$. Here $\bar \tau = c\tau$ for some appropriate constant $c$ depending only on $d$, $B_d$ and $B_{n-d}$ are $d-$dimensional and $(n-d)-$dimensional balls respectively.

 For every cylinder $\cyl_i$ in the packet, we define a function $f_i$ that is the squared distance to
 the $d-$ dimensional central cross section of $\cyl_i$. These functions are put together using a partition of unity defined on $\cup_i \cyl_i$. The resulting function $f$ is an ``approximate-squared-distance-function" (see Definition~\ref{def:12}).
 The base manifold is the set of points $x$ at which the gradient $\nabla f$ is orthogonal to every eigenvector corresponding to values in $[c, C]$ of the Hessian $\hess f (x)$. Here $c$ and $C$ are constants depending only on the dimension $d$ of the manifold. The fiber of the disc bundle at a point $x$ on the base manifold is defined to be the $(n-d)-$dimensional Euclidean ball centered at $x$ contained in the span of the aforementioned eigenvectors of the Hessian. The base manifold and its fibers together define the disc bundle.

 The optimization over sections of the disc bundle proceeds as follows. We fix a cylinder $\cyl_i$ of the cylinder packet. We optimize the squared loss over local sections corresponding to jets whose $C^2-$ norm is bounded above by $\frac{c_1}{\bar \tau}$, where $c_1$ is a controlled constant. The corresponding graphs  are each contained inside $\cyl_i$. The optimization over local sections is performed by minimizing squared loss over a space of $C^2-$jets (see Definition~\ref{def:crnorm}) constrained by inequalities developed in \cite{Feff_jets2}. The resulting local sections corresponding to various $i$ are then patched together  using the disc bundle and a partition of unity supported on the base manifold. The last step is performed implicitly, since we do not actually need to produce a manifold, but only need to certify the existence or non-existence a manifold possessing certain properties. The results of this paper together with those of  \cite{Feff_jets2} lead to an algorithm fitting a manifold to the data as well; the main additional is to construct local sections from jets, rather than settling for the existence of good local sections as we do here.

 Such optimizations are performed over a large ensemble of cylinder packets. Indeed the the size of this ensemble is the chief contribution in the complexity bound.

\subsection{Definitions}
\begin{definition}[reach]\lab{def:reach}
Let $\M$ be a subset of $\HH$. The
reach of $\MM$  is the largest number $\tau$ to have the property that any point at a distance  $r <
\tau$ from $\M$ has a unique nearest point in $\M$.
\end{definition}

\begin{definition}[Tangent Space]
Let $\HH$ be a separable Hilbert space.
For a closed $A \subseteq \HH$, and $a \in A$, let the ``tangent space" $Tan^0(a, A)$ denote the set of all vectors $v$ such that for all $\eps > 0$, there exists $b \in A$ such that $0 < |a - b| < \eps$ and $\big|v/{|v|} - \frac{b-a}{|b-a|}\big| < \eps$. For a set $X\subseteq \HH$ and a point $a\in \HH$  let $\dist(a, X)$ denote the Euclidean distance of the nearest point in $X$ to $a$.
 Let $Tan(a, A)$ denote the set of all $x$ such that $ x -a \in Tan^0(a, A)$.
\end{definition}

The following result of Federer (Theorem 4.18, \cite{federer_paper}), gives an alternate characterization of the reach.
 \begin{proposition}\label{thm:federer} Let $A$ be a closed subset of $\RR^n$. Then,
 \beq \reach(A)^{-1} = \sup\left\{2|b-a|^{-2}\dist(b, Tan(a, A))\big| \, a, b \in A\right\}.\eeq \end{proposition}

 \begin{definition}[$C^r-$submanifold]\lab{def:man}
We say that a closed subset $\MM$ of $\HH$ is a $d-$dimensional $\C^r-$submanifold of $\HH$ if the following is true.
 For every point $p \in \MM$ there exists a chart $(U \subseteq \HH, \phi : U \ra \HH)$, where $U$ is an open subset of $\HH$ containing $p$ such that $\phi$ possesses $k$ continuous derivatives and $\phi(\MM \cap U)$ is the intersection of a $d$-dimensional affine subspace with $\phi(U).$
Let $B_{\HH}$ be the unit ball in $\HH$.
 Let $\G = \G(d, V, \tau)$ be the family of  boundaryless $C^r-$submanifolds of $B_\HH$  having dimension $d$, volume less or equal to $ V$, reach greater or equal to $\tau$. We assume that $\tau < 1$ and $r = 2$.
 \end{definition}

 Let $\HH$ be a separable Hilbert space and $\PP$ be a probability distribution supported on its unit ball $B_\HH$.  Let $|\cdot|$ denote the Hilbert space norm on $\HH$. For $x, y \in \HH$, let $\dist(x, y) := |x-y|$.
 For any $x \in B_\HH$ and any $\M \subseteq B_{\HH}$, let
 $\dist(x, \MM) := \inf_{y\in \MM} |x - y|,$ and
 $$\LL(\M, \PP) := \int \dist(x, \M)^2 d\PP(x).$$

Let $\mathcal{B}$ be a black-box function which when given two vectors $v, w \in \HH$ outputs the inner product $\mathcal{B}(u, v) = <v, w>$ .
We develop an algorithm which for given $\de, \eps \in (0, 1)$, $V > 0$, integer $d$ and $\tau > 0$ does the following. 

We obtain an algorithm that tests the manifold hypothesis in the following sense.

 The algorithm takes i.i.d random samples from $\PP$ as input, and determines which of the following two is true (at least one must be):
 \ben
 \item There exists $\MM \in \G(d,  CV, \frac{\tau}{C})$ such that $\LL(\MM, \PP) \leq  C \eps.$
 \item There exists no $\MM \in \G(d,  V/C, C\tau)$ such that $\LL(\MM, \PP) \leq   \frac{\eps}{C}.$
 \een
The answer is correct with probability at least $1-\delta$.

The number of data points required is of the order of
\beqs n:= \frac{N_p \ln^4 \left(\frac{N_p}{\eps}\right) + \ln \de^{-1}}{\eps^2} \eeqs
where \beqs N_p := V\left(\frac{1}{\tau^d} + \frac{1}{\eps^{d/2}\tau^{d/2}}\right),\eeqs
and the number of arithmetic operations is
\beqs \exp\left(C\left(\frac{V}{\tau^d}\right)n \ln \tau^{-1}\right). \eeqs
The number of calls made to $\mathcal{B}$ is $O(n^2)$.


\begin{figure}\label{fig:torus}
\begin{center}
\includegraphics[height=2.0in]{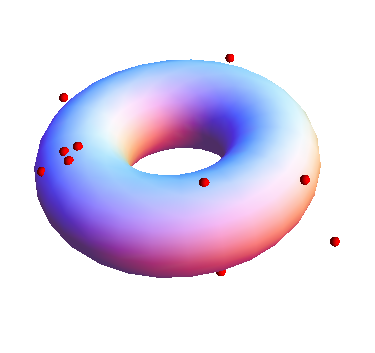}
\caption{Data lying in the vicinity of a two dimensional torus.}
\end{center}
\end{figure}

\subsection{A note on controlled constants}
In this section, and the following sections, we will make frequent use of constants $c, C, C_1, C_2, \oc_1, \dots, \oc_{11}$ and $c_{12}$ etc. These constants are "controlled constants" in the sense that their value is entirely determined by the dimension $d$ unless explicitly specified otherwise (as for example in Lemma~\ref{lem:charlie}). Also, the value of a constant can depend on the values of constants defined before it, but not those defined after it. This convention clearly eliminates the possibility of loops.

\section{Sample complexity of manifold fitting}

In this section, we show that if instead of estimating a least-square optimal manifold using the probability measure, we randomly sample sufficiently many points and then find the least square fit manifold to this data, we obtain an almost optimal manifold.

\begin{definition}[Sample Complexity]
 Given error  parameters $\eps, \de$, a space $X$ and a set of functions (henceforth function class)  $\F$ of functions $f:X \ra \RR$,  we define the sample complexity $s = s(\eps, \de, \F)$  to be the least number such that the following is true. There exists a function $\A:X^s \ra \FF$ such that, for any probability distribution $\PP$ supported on $X$,  if $(x_1, \dots, x_s) \in X^s$ is sequence of i.i.d draws from $\PP$, then  $f_{out} := \A((x_1, \dots, x_s))$  satisfies
$$\p\left[\E_{x \dashv \PP}  f_{out}(x) < (\inf_{f \in \F} \E_{x \dashv \PP} f) + \eps\right] > 1 - \de.$$
\end{definition}
 We state below, a sample complexity bound when mean-squared error is minimized over $\G(d, V, \tau)$.
\begin{theorem}\lab{thm:ext_manifold}
For $r > 0$, let $$U_\G(1/r) = CV\left(\frac{1}{\tau^d} + \frac{1}{(\tau r)^{d/2}}\right).$$
Let
$$s_\G(\eps, \de) :=   C\left(\frac{U_\G(1/\eps)}{\eps^2} \left(\log^4 \left(\frac{U_\G(1/\eps)}{\eps}\right)\right) + \frac{1}{\eps^2} \log \frac{1}{\de}\right).$$ Suppose $s \geq s_\G(\eps, \de)$
and $x = \{x_1, \dots, x_s\}$ be a set of i.i.d points from $\PP$ and $\PP_X$ is the uniform probability measure over $X$. Let $\MM_{erm}$ denote a manifold in $\G(d, V, \tau)$ that approximately minimizes the quantity  $$\sum_{i = 1}^s \dist(x_i, \M)^2 $$ in that
\beqs  \LL(\M_{erm}(x), \PP_X) -  \inf_{\M \in \G(d, V, \tau)} \LL(\M, \PP_X) < \frac{\eps}{2}. \eeqs Then,
 \beqs  \p\left[ \LL(\M_{erm}(x), \PP) -  \inf_{\M \in \G(d, V, \tau)} \LL(\M, \PP) < {\eps}\right] > 1 - \de.\eeqs
\end{theorem}


Let $\MM \in \G(d, V, \tau)$. For $x \in \MM$ denote  the orthogonal projection from $\HH$
 to the affine subspace $Tan(x, \MM)$ by $\Pi_{x}$. We will need the following claim to  prove Theorem~\ref{thm:ext_manifold}.

\begin{claim} \lab{cl:g1sept}
Suppose that $\MM \in \G(d, V, \tau)$. Let \beqs U:= \{y\big||y-\Pi_xy| \leq \tau/C\} \cap  \{y\big||x-\Pi_xy| \leq \tau/C\},\eeqs for a sufficiently large controlled constant $C$.
There exists a $C^{1, 1}$ function $F_{x, U}$ from $\Pi_x( U)$ to $\Pi_x^{-1}(\Pi_x(0))$ such that
\beqs \MM \cap U = \{ y + F_{x, U}(y) \big | y \in \Pi_x(U)\}\eeqs such that the Lipschitz constant of the gradient of $F_{x, U}$ is bounded above by $C$.
\end{claim}

\section{Proof of Claim~\ref{cl:g1sept}}

\subsection{Constants:}
$D$ is a fixed integer. Constants $c, C, C'$ etc depend only on $D$. These symbols may denote different constants in different occurrences, but $D$ always stays fixed.
\subsection{$D-$planes:}
$\HH$ denotes a fixed Hilbert space, possibly infinite-dimensional, but in any case of dimension $> D$. A $D-$plane is a $D-$dimensional vector subspace of $\HH$. We write $\Pi$ to denote a $D-$plane and we write $DPL$ to denote the space of all $D-$planes. If $\Pi, \Pi' \in DPL$, then we write $dist(\Pi, \Pi')$ to denote the infimum of $\|T - I\|$ over all orthogonal linear transformations $T:\HH \ra \HH$ that carry $\Pi$ to $\Pi'$. Here, the norm $\|A\|$ of a linear map $A:\HH \ra \HH$ is defined as $$
\sup_{v \in \HH\setminus \{0\}} \frac{\|Av\|_\HH}{\|v\|_\HH}.$$ One checks easily that
$(DPL, dist)$ is a metric space. We write $\Pi^\perp$ to denote the orthocomplement of $\Pi$ in $\HH$.
\subsection{Patches:}
Suppose $B_{\Pi}(0, r)$ is the ball of radius $r$ about the origin in a $D-$plane $\Pi$, and suppose

$$\Psi:B_{\Pi}(0, r) \ra \Pi^\perp$$ is a
$C^{1, 1}-$map, with $\Psi(0) = 0$. Then we call
$$\Gamma = \{x + \Psi(x): x \in B_\Pi(0, r)\} \subset \HH$$ \underline{a patch of radius $r$  over $\Pi$ centered at $0$.} We define
$$\|\Gamma\|_{\dot{C}^{1, 1}(B_{\Pi}(0, r))} := \sup_{\text{distinct}\, x, y\in B_{\Pi}(0, r)}\frac{\|\nabla \Psi(x) - \nabla \Psi(y)\|}{\|x - y\|};$$
Here, $$\nabla \Psi(x): \Pi \ra \Pi^\perp$$ is a linear map, and for linear maps
 $A:\Pi \ra \Pi^\perp$, we define $\|A\|$ as
$$\sup_{v \in \Pi \setminus \{0\}} \frac{\|Av\|}{\|v\|}.$$
If also $$\nabla \Psi(0) = 0 $$ then we call $\Gamma$ \underline{a patch of radius $r$ tangent to $\Pi$ at its center $0$.} If $\Gamma_0$ is a patch of radius $r$ over $\Pi$ centered at $0$ and if $z \in \HH$, then we call the translate $\Gamma = \Gamma_0 + z \subset \HH$ \underline{a patch of radius $r$ over $\Pi$, centered at $z$.} If $\Gamma_0$ is tangent to $\Pi$ at its center $0$, then we say that $\Gamma$ is tangent to $\Pi$ at its center $z$.

The following is an easy consequence of the implicit function theorem in fixed dimension ($D$ or $2D$).

\begin{lemma}\lab{lem:smiley}
Let $\Gamma_1$ be a patch of radius $r_1$ over $\Pi_1$ centered at $z_1$ and tangent to $\Pi_1$ at $z_1$. Let $z_2$ belong to $\Gamma_1$ and suppose $\|z_2 - z_1\| < c_0 r_1.$ Assume $$\|\Gamma_1\|_{\dot{C}^{1, 1}(B_{\Pi}(z_1, r_1))} \leq \frac{c_0}{r_1}.$$ Let $\Pi_2 \in DPL$ with $dist(\Pi_2, \Pi_1) < c_0.$ Then there exists a patch $\Gamma_2$ of radius $c_1r_1$ over $\Pi_2$ centered at $z_2$ with
$$\|\Gamma_2\|_{\dot{C}^{1, 1}(B_{\Pi}(0, c_1r_1))} \leq \frac{200 c_0}{r_1},$$ and $$\Gamma_2 \cap B_{\HH}\left(z_2, \frac{c_1 r_1}{2}\right) = \Gamma_1 \cap B_{\HH}\left(z_2, \frac{c_1 r_1}{2}\right).$$
\end{lemma}
Here $c_0$ and $c_1$ are small constants depending only on $D$, and by rescaling, we may assume without loss of generality that $r_1 = 1$ when we prove Lemma~\ref{lem:smiley}.

\begin{figure}\label{fig:claim}
\begin{center}
\includegraphics[height=2.4in]{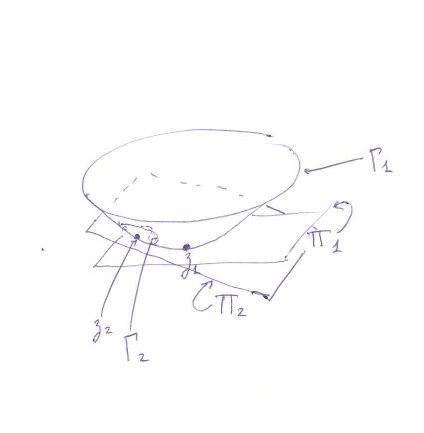}
\caption{}
\end{center}
\end{figure}

The meaning of Lemma~\ref{lem:smiley} is that if $\Gamma$ is the graph of a map
$$\Psi:B_{\Pi_1}(0, 1) \ra \Pi_1^\perp$$ with $\Psi(0)= 0$ and $\nabla \Psi(0) = 0$ and the $C^{1, 1}-$norm of $\Psi$ is small then at any point $z_2 \in \Gamma$ close to $0$, and for any $D-$plane $\Pi_2$ close to $\Pi_1$, we may regard $\Gamma$ near $z_2$ as the graph $\Gamma_2$ of a map
$$\tilde{\Psi}: B_{\Pi_2}(0, c) \ra \Pi_2^\perp;$$ here $\Gamma_2$ is centered at $z_2$ and the $C^{1, 1}-$norm of $\tilde{\psi}$ is not much bigger than that of $\Psi$.
\subsection{Imbedded manifolds:}
Let $\MM \subset \HH$ be a "compact imbedded $D-$manifold" (for short, just a "manifold") if the following hold:
\begin{itemize}
\item $\MM$ is compact.
\item There exists an $r_1 > r_2 > 0$ such that for every $z \in \MM$, there exists $T_z\MM \in DPL$ such that $\MM \cap B_\HH(z, r_2) = \Gamma \cap B_\HH(z, r_2)$ for some patch $\Gamma$ over $T_z(\MM)$ of radius $r_1$, centered at $z$ and tangent to $T_z(\MM)$ at $z$. We call $T_z(\MM)$ the \underline{tangent space} to $\MM$ at $z$.
\end{itemize}

We say that $\MM$ has \underline{infinitesimal reach} $ \leq \rho $ if for every $\rho' < \rho$, there is a choice of $r_1 > r_2 > 0$  such that for every $z \in \MM$ there is a patch $\Gamma$ over $T_z(\MM)$ of radius $r_1$, centered at $z$ and tangent to $T_z(\MM)$ at $z$ which has $C^{1, 1}-$norm at most $\frac{1}{\rho'}$.

\subsection{Growing a Patch}
\begin{lemma}["Growing Patch"]\lab{lem:GP}
Let $\MM$ be a manifold and let $r_1, r_2$ be as in the definition of a manifold. Suppose $\MM$ has infinitesimal reach $\geq 1$. Let $\Gamma \subset \MM$ be a patch of radius $r$ centered at $0$, over $T_0\MM$.
Suppose $r$ is less than a small enough constant $\hat{c}$ determined by $D$. Then there exists a patch $\Gamma^+$ of radius $ r + c r_2$ over $T_0\MM$, centered at $0$ such that $\Gamma \subset \Gamma^+ \subset \MM.$
\end{lemma}
\begin{corollary}\lab{cor:GP}
Let $\MM$ be a manifold with infinitesimal reach $\geq 1$ and suppose $0 \in \MM$. Then there exists a patch $\Gamma$ of radius $\hat{c}$ over $T_0\MM$ such that $\Gamma \subset \MM$.
\end{corollary}
\underline{Lemma~\ref{lem:GP} implies Corollary~\ref{cor:GP}.} Indeed, we can start with a tiny patch $\Gamma$ (centered at $0$) over $T_0\MM$, with $\Gamma \subset \MM$. Such $\Gamma$ exists because $\MM$ is a manifold. By repeatedly applying the Lemma, we can repeatedly increase the radius of our patch by a fixed amount $cr_2$; we can continue doing so until we arrive at a patch of radius $\geq \hat{c}$.
\begin{proof}[Proof of Lemma~\ref{lem:GP}] Without loss of generality, we can take $\HH = \RR^D \oplus \HH'$ for a Hilbert space $\HH'$; and we may assume that $$T_0\MM = \R^D\times \{0\} \subset \R^D\oplus\HH'.$$ Our patch $\Gamma$ is then a graph $$\Gamma = \{(x, \Psi(x)):x \in B_{\RR^D}(0, r)\} \subseteq \R^D \oplus \HH'$$
for a $C^{1, 1}$ map $$\Psi:B_{\RR^D}(0, r) \ra \HH',$$ with
$\Psi(0) = 0,$ $\nabla \Psi(0) = 0,$ and $$\|\Psi\|_{\dot{C}^{1, 1}(B_{\RR^D}(0, r))}\leq C_0.$$ For $y \in B_{\RR^D}(0, r)$, we therefore have $|\nabla \psi(y)|\leq C_0$. If $r$ is less than a small enough $\hat{c}$ then Lemma~\ref{lem:smiley} together with the fact that $\MM$ agrees with a patch of radius $r_1$ in $B_{\RR^D\oplus\HH'}((y, \Psi(y)), r_2)$ (because $\MM$ is a manifold) tells us that there exists a $C^{1, 1}$ map $$\Psi_y:B_{\RR^D}(y, c'r_2) \ra \HH'$$ such that
$$\MM\cap B_{\RR^D\oplus\HH'}((y, \Psi(y)), c''r_2) = \{(z, \Psi_y(z)):z \in B_{\RR^D}(y, c'r_2)\}\cap B_{\RR^D\oplus \HH'}((y, \Psi(y)), c''r_2).$$ Also, we have a priori bounds on $\|\nabla_z \Psi_y(z)\|$ and on $\|\Psi_y\|_{\dot{C}^{1, 1}}.$ It follows that whenever $y_1, y_2 \in B_{\RR^D}(0, r)$ and
$z \in B_{\RR^D}(y_1, c'''r_2) \cap B_{\RR^D}(y_2, c'''r_2)$, we have $\Psi_{y_1}(z) = \Psi_{y_2}(z).$

This allows us to define a global $C^{1, 1}$ function
$$\Psi^+:B_{\RR^D}(0, r + c'''r_2) \ra \HH';$$ the graph of $\Psi^+$ is simply the union of the graphs of
$$\Psi_y|_{B_{\RR^D}(y, c'''r_2)}$$ as $y$ varies over $B_{\RR^D}(0, r).$ Since the graph of each $\Psi_y|_{B_{\RR^D}(y, c'''r_2)}$ is contained in $\MM$, it follows that the graph of $\Psi^+$ is contained in $\MM$. Also, by definition, $\Psi^+$ agrees on $B_{\RR^D}(y, c'''r_2)$ with a $C^{1, 1}$ function, for each $y \in B_{\RR^D}(0, r)$. It follows that
$$\|\Psi^+\|_{\dot{C}^{1, 1}(0, rc'''r_2)} \leq C.$$ Also, for each $y \in B_{\RR^D}(0, r)$, the point $(y, \Psi(y))$ belongs to $$\MM \cap B_{\RR^D\oplus\HH'}((y, \Psi(y)), \frac{c'''r_2}{1000}),$$ hence it belongs to the graph of
$\Psi_y|_{B_{\RR^D}(y, c'''r_2)}$ and therefore it belongs to the graph of $\Psi^+$. Thus $ \Gamma^+ = \text{\,graph of \,} \Psi^+$ satisfies $\Gamma \subset \Gamma^+ \subset \MM,$ and $\Gamma^+$ is a patch of radius $r + c'''r_2$ over $T_0\MM$ centered at $0$. That proves the lemma.
\end{proof}

\subsection{Global Reach}
For a real number $\tau > 0$,
A manifold $\MM$ has \underline{reach} $\geq \tau$ if and only if every $x \in \HH$ such that $\dist(x, \MM) < \tau$ has a unique closest point of $\MM$.
By Federer's characterization of the reach in Proposition~\ref{thm:federer}, if the reach is greater than one, the infinitesimal reach is greater than $1$ as well.

\begin{lemma}
Let $\MM$ be a manifold of reach $\geq 1$, with $0 \in \MM$. Then, there exists a patch $\Gamma$ of radius $\hat{c}$ over $T_0\MM$ centered at $0$, such that
$$\Gamma \cap B_\HH(0, \check{c}) = \MM\cap B_{\HH}(0, \check{c}).$$
\end{lemma}
\begin{proof}
There is a patch $\Gamma$ of radius $\hat{c}$ over $T_0\MM$ centered at $0$ such that
$$\Gamma \cap B_\HH(0, c^{\sharp}) \subseteq \MM\cap B_\HH(0, c^{\sharp}).$$ (See Lemma~\ref{lem:GP}.) For any $x \in \Gamma \cap B_\HH(0, c^{\sharp})$, there exists a tiny ball $B_x$ (in $\HH$) centered at $x$ such that $\Gamma \cap B_x = \MM \cap B_x$; that's because $\MM$ is a manifold.

It follows that the distance from

$$\Gamma_{yes}:= \Gamma \cap B_\HH(0, \frac{c^{\sharp}}{2})$$ to $$\Gamma_{no} := \left[\MM\cap B_\HH(0, \frac{c^{\sharp}}{2})\right]\setminus \left[\Gamma\cap B_\HH(0, \frac{c^{\sharp}}{2})\right].$$ is strictly positive.

Suppose $\Gamma_{no}$ intersects $B_H(0, \frac{c^\sharp}{100})$; say $y_{no} \in B_\HH(0, \frac{c^\sharp}{100}) \cap \Gamma_{no}$. Also,  $0 \in B_\HH(0, \frac{c^\sharp}{100})\cap \Gamma_{yes}.$

The continuous function $B_\HH(0, \frac{c^\sharp}{100}) \ni y \mapsto \dist(y, \Gamma_{no}) - \dist(y, \Gamma_{yes})$ is positive at $y=0$ and negative at $y = y_{no}$. Hence at some point,
$$y_{Ham} \in B_\HH(0, \frac{c^\sharp}{100})$$
we have $$\dist(y_{Ham}, \Gamma_{yes}) = \dist(y_{Ham}, \Gamma_{no}).$$ It follows that $y_{Ham}$  has two distinct closest points in $\MM$ and yet $$\dist(y_{Ham}, \MM) \leq \frac{c^\sharp}{100}$$ since $0\in \MM$ and $y_{Ham} \in B_\HH(0, \frac{c^\sharp}{100}).$ That contradicts our assumption that $\MM$ has reach $\geq 1$. Hence our assumption that $\Gamma_{no}$ intersects $B_\HH(0, \frac{c^\sharp}{100})$ must be false. Therefore, by definition of $\Gamma_{no}$ we have $$\MM\cap B_\HH(0, \frac{c^\sharp}{100})\subset \Gamma \cap B_\HH(0, \frac{c^\sharp}{100}).$$ Since also $$\Gamma \cap B_\HH(0, c^\sharp) \subset \MM\cap B_\HH(0, c^\sharp),$$
it follows that
$$\Gamma \cap B_\HH(0, \frac{c^\sharp}{100}) = \MM\cap B_\HH(0, \frac{c^\sharp}{100}),$$ proving the lemma.
\end{proof}

This completes the proof of Claim~\ref{cl:g1sept}.

\section{A bound on the size of an $\eps-$net}

\begin{definition}\lab{def:net}
Let $(X, \dist)$ be a metric space,  and $r>0$. We say that $Y$ is an $r-$net of $X$ if  $Y \subseteq X$ and for every $x \in X$, there  is a point $y \in Y$ such that $\dist(x, y) < r$.
\end{definition}

\begin{corollary}\lab{cor:r-net}
Let \beqs U_\G:\RR^+ \ra \RR\eeqs be given by \beqs U_\G(1/r) = CV\left(\frac{1}{\tau^d} + \frac{1}{(\tau r)^{d/2}}\right).\eeqs
Let $\MM \in \G$, and $\MM$ be equipped with the metric $\dist_\HH$ of the $\HH$. Then, for any $r > 0$, there is an $\sqrt{\tau r}-$net of $\MM$ consisting of no more than $U_\G(1/r)$ points.
\end{corollary}
\begin{proof}
It suffices to prove that for any $r\leq \tau$, there is an $r-$net of $\MM$ consisting of no more than $CV\left(\frac{1}{\tau^d} +  \frac{1}{r^d}\right)$, since if $r > \tau$, a  $\tau-$net is also an $r-$net.
Suppose $Y = \{y_1, y_2, \dots\}$ is constructed by the following greedy procedure. Let $y_1 \in \MM$ be chosen arbitrarily. Suppose $y_1, \dots y_k$ have been chosen. If the set of all $y$ such that $\min_{1\leq i\leq k} |y - y_i|) \geq r$ is non-empty,  let $y_{k+1}$ be an arbitrary member of this set. Else declare the construction of $Y$ to be complete.

 We see that that $Y$ is an $r-$net of $\MM$. Secondly, we see that the  the distance between any two distinct points $y_i, y_j \in Y$ is greater or equal to $r$. Therefore the two balls $\MM \cap B_\HH(y_i, r/2)$ and $\MM \cap B_\HH(y_j, r/2)$ do not intersect.

By Claim~\ref{cl:g1sept} for each $y \in Y$,  there are controlled constants $0< c < 1/2$ and $0 < c'$ such that for any $r \in (0, \tau]$, the volume of $\MM \cap B_\HH(y, cr)$ is greater than $c' r^d$.

Since the volume of $$\{z\in \MM|\dist(z, Y) \leq r/2\}$$ is less or equal to $V$ the cardinality of $Y$  is less or equal to $\frac{V}{c'r^d}$ for all $r\in (0, \tau].$ The corollary follows.
\end{proof}

\subsection{Fitting $k$ affine subspaces of dimension $d$}
A natural generalization of k-means was proposed in
\cite{Bradley00} wherein one fits  $k$ $d-$dimensional planes to data in a manner that
minimizes the average squared distance of a data point to the nearest $d-$dimensional plane. For more recent results on this kind of model, with the average $p^{th}$ powers rather than squares, see \cite{Lerman}.
We can view $k-$means as a $0-$dimensional special case of $k-$planes.

In this section, we derive an upper bound for the generalization error of fitting $k-$planes. Unlike the earlier bounds
for fitting manifolds, the bounds here are linear in the dimension $d$ rather than exponential in it. The dependence on
$k$ is linear up to logarithmic factors, as before. In the section, we assume for notation convenience that the dimension $m$ of the Hilbert space is finite, though the results can be proved for any separable Hilbert space.

Let $\PP$ be a probability distribution supported on $B :=\{x \in \RR^m \big|\, \|x\| \leq 1\}$.
Let $\H := \H_{k, d}$ be the set whose elements are unions of not more than $k$ affine subspaces of dimension $\leq d$, each of which  intersects $B$.
Let $\FF_{k, d}$ be the set of all loss functions $F(x) =  \dist(x, H)^2$ for some $H \in \H$
(where $\dist(x, S) := \inf_{y \in S} \|x - y\|$) .

We wish to obtain a probabilistic upper bound on

\beq \lab{eq1}\sup_{F \in \FF_{k, d}} \Bigg | \frac{\sum_{i=1}^s F(x_i)}{s} - \E_\PP F(x)\Bigg |, \eeq
where $\{x_i\}_1^s$ is the train set and $\E_\PP F(x)$ is the expected value of $F$ with respect to $\PP$.
Due to issues of measurability, (\ref{eq1}) need not be random variable for arbitrary $\FF$. However, in our situation, this is the case because $\FF$ is a family of bounded piecewise quadratic functions, smoothly parameterized by $\HH_b^{\times k}$, which has a countable dense subset, for example, the subset of elements specified using rational data.
We obtain a bound that is independent of $m$, the ambient dimension.

\begin{theorem}
 Let $x_1, \dots, x_s$ be i.i.d samples from $\PP$, a distribution supported on the ball of radius $1$ in $\R^m$.
 If
$$s \geq  C\left(\frac{dk}{\eps^2} \log^4 \left(\frac{dk}{\eps}\right)  + \frac{d}{\eps^2} \log \frac{1}{\de}\right),$$  then
$\p\left[\sup\limits_{F \in \F_{k, d}} \Bigg | \frac{\sum_{i=1}^s F(x_i)}{s} - \E_\PP F(x)\Bigg | < \eps\right] > 1 - \de. $
\end{theorem}
\begin{proof}
Any $F \in \FF_{k, d}$ can be expressed as $F(x) = \min_{1 \leq i \leq k} \dist(x, H_i)^2$ where each $H_i$ is an affine subspace of dimension less or equal to $d$ that intersects the unit ball. In turn, $\min_{1 \leq i \leq k} \dist(x, H_i)^2$ can be expressed as
\beqs \min\limits_{1 \leq i \leq k} \left(\|x - c_i\|^2 - (x - c_i)^{\dag} A_i^{\dag} A_i(x - c_i)\right), \eeqs
where $A_i$ is defined by the condition that for any vector $z$, $\left(z - (A_i z)\right)^{\dag}  $ and $A_i z$ are the components of $z$ parallel and perpendicular to $H_i$, and $c_i$ is the point on $H_i$ that is the nearest to the origin (it could have been any point on $H_i$). Thus \beqs F(x) := \min_i \left(\|x\|^2 - 2 c_i^{\dag}  x + \|c_i\|^2 - x^{\dag}   A_i^{\dag}  A_i x + 2c_i^{\dag}  A_i^{\dag}  A_i x - c_i^{\dag}  A_i^{\dag}  A_i c_i\right). \eeqs

Now, define vector valued maps $\Phi$ and $\Psi$ whose respective domains are the space of $d$ dimensional affine subspaces and $\HH$ respectively.
\beqs \Phi(H_i) := \left(\frac{1}{\sqrt{d+5}}\right)\left(\|c_i\|^2, A_i^{\dag}  A_i, (2A_i^{\dag} A_ic_i - 2c_i)^{\dag} \right)\eeqs
and \beqs \Psi(x) :=  \left(\frac{1}{\sqrt{3}}\right) (1, x x^{\dag} , x^{\dag} ),\eeqs
where $A_i^{\dag}  A_i$ and $xx^{\dag} $ are interpreted as rows of $m^2$ real entries.

Thus, $$\min_i \left(\|x\|^2 - 2 c_i^{\dag}  x + \|c_i\|^2 -  x^{\dag}  A_i^{\dag}  A_ix + 2c_i^{\dag}  A_i^{\dag}  A_i x - c_i^{\dag}  A_i^{\dag}  A_i c_i\right)$$ is equal to   $$\|x\|^2 + \sqrt{3(d+5)} \min_i \Phi(H_i) \cdot \Psi(x).$$
We see that since $\|x\| \leq 1$, $\|\Psi(x)\|\leq 1$. The Frobenius norm  $\|A_i^{\dag}  A_i\|^2$ is equal to $Tr(A_iA_i^{\dag} A_iA_i^{\dag} ),$ which is the rank of $A_i$ since $A_i$ is a projection. Therefore, $$(d+5) \|\Phi(H_i)\|^2 \leq \|c_i\|^4 + \|A_i^{\dag}  A_i\|^2 + \|(2(I - A_i^{\dag} A_i)c_i\|^2,$$ which,  is less or equal to $d+5$.

Uniform bounds for classes of  functions of the form  $\min_i \Phi(H_i) \cdot \Psi(x)$ follow from Lemma~\ref{lem:key}. We infer from Lemma~\ref{lem:key} that
if $$s \geq  C\left(\frac{k}{\eps^2} \log^4 \left(\frac{k}{\eps}\right)  + \frac{1}{\eps^2} \log \frac{1}{\de}\right),$$  then
$\p\left[\sup\limits_{F \in \F_{k, d}} \Bigg | \frac{\sum_{i=1}^s F(x_i)}{s} - \E_\PP F(x)\Bigg | < \sqrt{3(d+5)}\eps\right] > 1 - \de. $
The last statement can be rephrased as follows. If
$$s \geq  C\left(\frac{dk}{\eps^2} \log^4 \left(\frac{dk}{\eps}\right) + \frac{d}{\eps^2} \log \frac{1}{\de}\right),$$  then
$\p\left[\sup\limits_{F \in \F_{k, d}} \Bigg | \frac{\sum_{i=1}^s F(x_i)}{s} - \E_\PP F(x)\Bigg | < \eps\right] > 1 - \de.$
\end{proof}

\section{Tools from empirical processes}
\begin{figure}\label{fig:unif_bound}
\begin{center}
\includegraphics[height=1.8in]{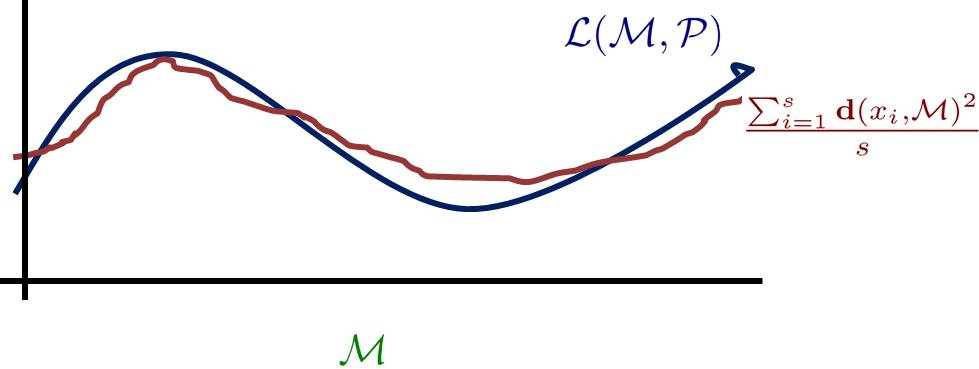}
\caption{A uniform bound (over $\G$) on the difference between the empirical and true loss.}
\end{center}
\end{figure}
In order to prove a uniform bound of the form
\beq \lab{eq2}\p\left[\sup_{F \in \F} \Bigg | \frac{\sum_{i=1}^s F(x_i)}{s} - \E_\PP F(x)\Bigg | < \eps\right] > 1 - \de, \eeq it suffices to bound a measure of the complexity of $\F$ known as the Fat-Shattering dimension  of the function class $\F$. The metric entropy (defined below) of $\F$ can be bounded using the Fat-Shattering dimension, leading to a uniform bound of the form of (\ref{eq2}).

\begin{definition}[metric entropy]
Given a metric space $(Y, \rho)$, we call $Z \subseteq Y$ an $\eta-$net of $Y$ if for every $y \in Y$, there is a $z \in Z$ such that $\rho(y, z) < \eta.$
Given a measure $\PP$ supported on a metric space $X$, and $\FF$ a class of functions from $X$ to $\R$. Let $N(\eta, \F, \LL_2(\PP))$ denote the minimum number of elements that an $\eta-$net of $\F$ could have, with respect to the metric imposed by the Hilbert space $\LL_2(\PP)$, wherein the distance between $f_1:X \ra \R$ and $f_2:X \ra \R$ is $$\|f_1 - f_2\|_{\LL_2(\PP)} = \sqrt{\int (f_1(x) - f_2(x))^2 d\mu}.$$ We call $\ln N(\eta, \F, \LL_2(\PP))$ the metric entropy of $\FF$ at scale $\eta$ with respect to $\LL_2(\PP)$.
\end{definition}

 \begin{definition} [Fat-shattering dimension] \lab{def:fat} Let $\F$ be a set of real valued functions. We say that a set of points $x_1, \dots, x_k$ is $\gamma-$shattered by $\F$ if there is a vector of real numbers $t = (t_1, \dots, t_k)$ such that for all binary vectors $\mathbf{b} = (b_1, \dots, b_k)$  and each $ i \in [s] = \{1, \dots, s\}$, there is a function $f_{\mathbf{b}, t}$ satisfying,
\beq\lab{eq:deffat} f_{\mathbf{b}, t}(x_i) = \left\{ \begin{array}{ll}
    > t_i + \gamma, & \hbox{if $ b_i = 1$;} \\
    < t_i - \gamma, & \hbox{if $ b_i = 0$.}
  \end{array} \right. \eeq
More generally, the supremum taken over $(t_1, \dots, t_k)$ of  the number of binary vectors $\mathbf{b}$ for which there is a function $f_{\mathbf{b}, t} \in \FF$ which satisfies (\ref{eq:deffat}), is called the $\gamma-$shatter coefficient.
For each $\gamma > 0$, the {\it Fat-Shattering dimension} $\fat_\gamma(\F)$ of the set $\F$ is defined to be the size of the largest $\gamma-$shattered set if this is finite; otherwise $\fat_\gamma(\F)$ is declared to be infinite.
\end{definition}

We will also need to use the notion of VC dimension, and some of its properties. These appear below.
\begin{definition}
Let $\La$ be a collection of measurable subsets of $\R^m$. For $x_1, \dots, x_k \in \R^m$, let the number of different sets in $\{\{x_1, \dots, x_k\} \cap H; H \in \La\}$ be denoted the shatter coefficient $N_\La (x_1, \dots, x_k)$. The VC dimension $VC_\La$ of $\La$ is the largest integer $k$ such that there exist $x_1, \dots x_k$ such that $N_\La (x_1, \dots, x_k) = 2^k$.
\end{definition}
The following result concerning the VC dimension of halfspaces is well known (Corollary 13.1, \cite{devroye}).
\begin{theorem}\lab{thm:steele}
Let $\La$ be the class of halfspaces in $\R^g$. Then
$VC_\La = g + 1$.\end{theorem}
We state the Sauer-Shelah Lemma below.
\begin{lemma}[Theorem 13.2, \cite{devroye}]\lab{lem:ssh}
For any $x_1, \dots, x_k \in \R^g$, $N_\La(x_1, \dots, x_k) \leq \sum_{i=0}^{VC_\La} {k \choose i}$.
\end{lemma}
For $VC_\La > 2$, $\sum_{i=0}^{VC_\La} {k \choose i} \leq k^{VC_\La}$.

The  lemma below follows from existing results from  the theory of Empirical Processes in a straightforward manner, but does not seem to have appeared in print before. We have provided a proof in the appendix.
\begin{lemma}\lab{lem:fat_to_gen}
 Let $\mu$ be a measure supported on $X$,   $\F$  be a class of functions $f:X \ra \R$. Let $x_1, \dots, x_s$ be independent random variables drawn
from $\mu$ and $\mu_s$ be the uniform measure on $x:= \{x_1, \dots, x_s\}$.
If $$s \geq \frac{C}{\eps^2}\left(\left(\int_{{c\eps}}^\infty \sqrt{\fat_\gamma(\F)}d\gamma\right)^2 + \log 1/\de\right) ,$$
then,
$$\p\left[\sup_{f \in {\F}} \bigg| \E_{\mu_s} f(x_i) - \E_\mu  f\bigg| \geq \eps\right] \leq 1 - \de.$$
\end{lemma}

 A key component in the proof of the uniform bound in Theorem~\ref{thm:ext_manifold} is an upper bound on the fat-shattering dimension of functions given by the maximum of a set of minima of collections of linear functions on a ball in $\HH$. We will use the Johnson-Lindenstrauss Lemma \cite{JL} in its proof.\\\\
\begin{figure}\label{fig:proj}
\begin{center}
\includegraphics[height=2.4in]{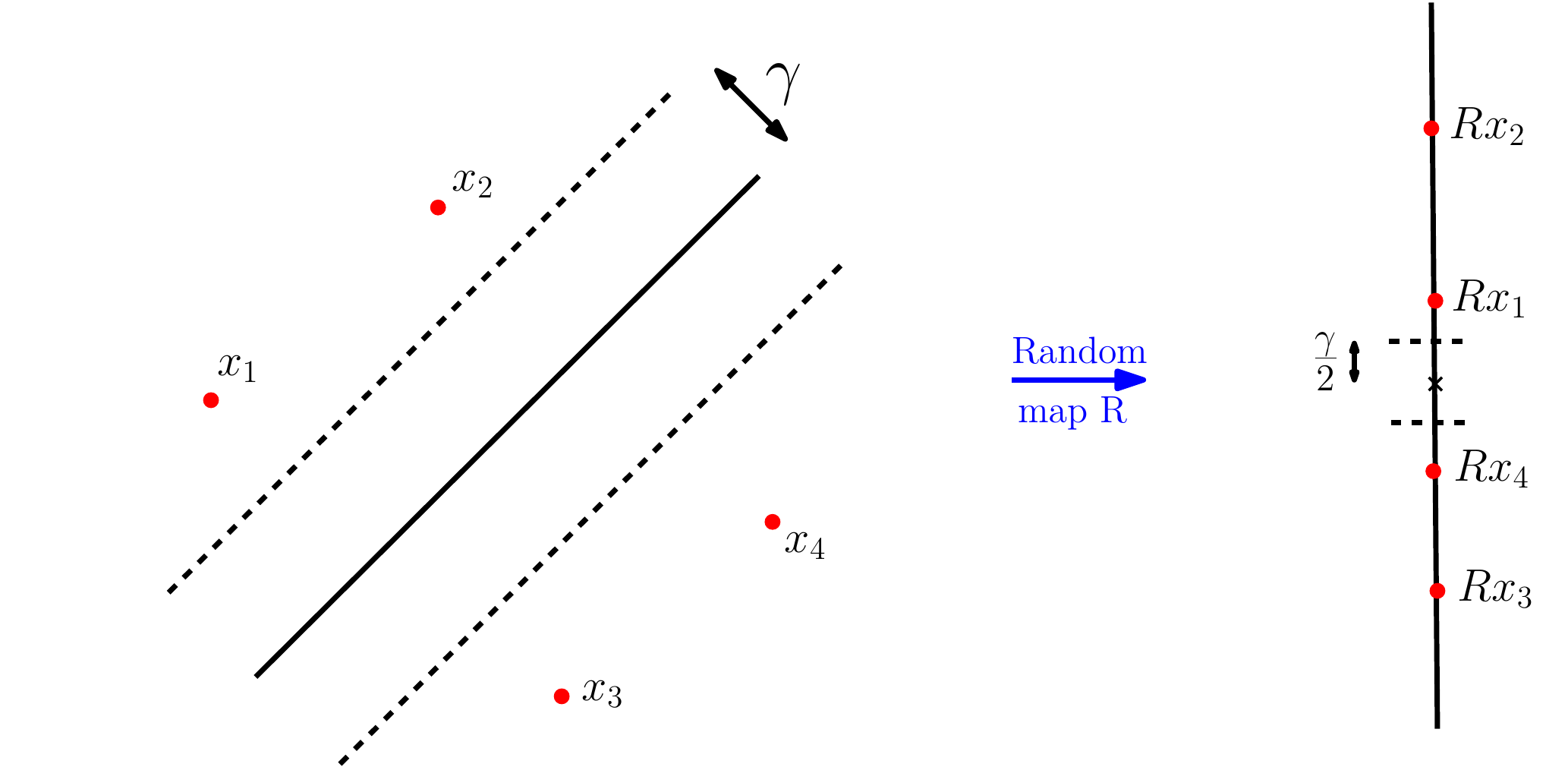}
\caption{Random projections are likely to preserve linear separations.}
\end{center}
\end{figure}
Let $J$ be a finite dimensional vectorspace of dimension greater or equal to $g$. In what follows, by "uniformly random $g-$dimensional subspace in $J$," we mean a random variable taking taking values in the set of $g-$dimensional subspaces of $J$, possessing the following property. Its distribution is invariant under the action of the orthogonal group acting on $J$.

\noindent\underline{Johnson-Lindenstrauss Lemma}:
Let $y_1, \dots, y_{\bar \ell}$ be points in the unit ball in $\R^m$ for some finite $m$. Let $R$ be an orthogonal projection onto a random $g-$dimensional subspace (where $g = C \frac{\log \bar{\ell}}{\gamma^2}$ for some $\gamma > 0$, and an absolute constant $C$).
Then, \beqs \p\left[\sup\limits_{i, j \in \{1, \dots, g\}} \bigg|\left(\frac{m}{g}\right)(R y_i) \cdot (R y_j) -  y_i \cdot y_j \bigg| > \frac{\gamma}{2}\right] < \frac{1}{2}.\eeqs

\begin{lemma}\lab{lem:key} Let $\PP$ be a probability distribution supported on $B_\HH$.
Let $\FF_{k,\ell}$ be the set of all functions $f$ from $B_\HH := \{x \in \HH: \|x\| \leq 1\}$ to $\RR$, such that for some $k \ell$ vectors $v_{11}, \dots, v_{k\ell} \in B$, $$f(x) = \max_j \min_i (v_{ij} \cdot x). $$
\ben
\item $\fat_\gamma(\F_{k, \ell}) \leq \frac{Ck\ell}{\gamma^2}\log^2\frac{Ck\ell}{\gamma^2}.$
\item If $s \geq \frac{C}{\eps^2}\left(k\ell\ln^4(k\ell/\eps^2) + \ln 1/\de\right)$,
then $\p\left[\sup_{f \in \F_{k, \ell}} \big| \E_{\mu_s} f(x_i) - \E_\mu  f\big| \geq \eps\right] \leq 1 - \de.$
\een
\end{lemma}

\begin{proof}
We proceed to obtain an upper bound on the fat shattering dimension  $\fat_\gamma({\FF_{k, \ell}})$.
Let $x_1, \dots, x_s$ be $s$ points such that $$\forall \A \subseteq X := \{x_1, \dots, x_s\},$$ there exists $V := \{ v_{11}, \dots, v_{k\ell}\} \subseteq B$ and $ f \in {\FF_{k, \ell}}$ where $f(x) = \max_j \min_i v_{ij} \cdot x$   such that for some $\mathbf{t} = (t_1, \dots, t_s)$, for all

\beq\lab{eq:11}x_r \in \A, \forall\, j \in [\ell],\, \text{\,there\, exists\,} \, i \in [k] \,\,\,\,\,\,\, v_{ij} \cdot x_r < t_r - \gamma\eeq and \beq\lab{eq:22}\forall x_r \not\in \A, \exists\, j \in [\ell],\, \forall i \in [k] \,\,\,\,\,\,\, v_{ij} \cdot x_r > t_r +  \gamma.\eeq

We will obtain an upper bound on $s$.
Let $g:=C_1\left({\gamma}^{-2} \log(s + k\ell)\right)$ for a sufficiently large universal constant $C_1$.
Consider a particular $\A \in X$ and $f(x) := \max_j \min_i v_{ij} \cdot x$ that satisfies (\ref{eq:11}) and (\ref{eq:22}).

Let $R$ be an orthogonal  projection onto a uniformly random $g-$dimensional subspace of $span(X \cup V)$; we denote the family of all such linear maps $\Re$. Let $RX$ denote the set $\{Rx_1, \dots, Rx_s\}$ and likewise, $RV$ denote the set $\{Rv_{11}, \dots, Rv_{kl}\}$.
Since all vectors in $X \cup V$ belong to the unit ball $B_\HH$, by the Johnson-Lindenstrauss Lemma,   with probability greater than ${1/2}$, the inner product of every pair of vectors in $RX \cup RV$ multiplied by $\frac{m}{g}$ is within $\gamma $ of the inner product of the corresponding vectors in $X \cup V$.

Therefore,  we have the following.
 \begin{obs}\lab{obs:11} With probability at least $\frac{1}{2}$  the following statements are true.
\beq\lab{eq:11R}\forall x_r \in \A, \forall\, j \in [\ell],\, \exists \, i \in [k] \,\,\,\,\,\,\, \left(\frac{m}{g}\right)R v_{ij} \cdot R x_r < t_r \eeq and \beq\lab{eq:22R}\forall x_r \not\in \A, \exists\, j \in [\ell],\, \forall i \in [k] \,\,\,\,\,\,\, \left(\frac{m}{g}\right)R v_{ij} \cdot R x_r > t_r.\eeq
\end{obs}

Let $R \in \Re$ be a projection onto a uniformly random $g-$dimensional subspace in $span(X\cup V)$.  Let $J := \mathrm{span}(RX)$ and let $t^J :J \ra \RR$ be the function given by
\beqs t^J(y) := \left\{
                \begin{array}{ll}
                  t_i, & \hbox{\text{if}\, $y = Rx_i$ for some $i \in [s]$;} \\
                  0, & \hbox{otherwise.}
                \end{array}
              \right.\eeqs
Let $\FF_{J, k, \ell}$ be the concept class consisting of all subsets of $J$ of the form $$\left\{z: \max_j \min_i \left(\begin{array}{c}
                                 w_{ij} \\
                                 1
                               \end{array}\right)
 \cdot \left(\begin{array}{c}
                                 z \\
                                 - t^J(z)
                               \end{array}\right)\leq 0\right\},$$ where $w_{11}, \dots w_{k\ell}$ are arbitrary vectors in $J$.
\begin{claim}\lab{cl:min}
 Let $y_1, \dots, y_s \in J$. Then, the number $W(s, \FF_{J, k, \ell})$  of distinct sets  $\{y_1, \dots, y_s\} \cap \imath, \,\,\imath \in \FF_{J, k, \ell}$ is less or equal to  $s^{O((g+2)k\ell)}$.
\end{claim}
\begin{proof}[Proof of Claim~\ref{cl:min}]
Classical VC theory (Theorem~\ref{thm:steele}) tells us that the VC dimension of Halfspaces in the span of all vectors of the form $(z;- t^J(z))$ is at most $ g+2$. Therefore, by the Sauer-Shelah Lemma (Lemma~\ref{lem:ssh}),
the number $W(s, \FF_{J, 1, 1})$ of distinct sets  $\{y_1, \dots, y_s\} \cap \jmath, \,\, \jmath \in \FF_{J, 1, 1}$ is less or equal to $ \sum_{i=0}^{g+2} {s \choose i} $, which is less or equal to $ s^{g+2}$.
Every set of the form $\{y_1, \dots, y_s\} \cap \imath, \,\,\imath \in  \FF_{J, k, \ell}$ can be expressed as an intersection of a union of sets  of the form $\{y_1, \dots, y_s\} \cap \jmath, \,\,\jmath \in \FF_{J, 1, 1}$, in which the total number of sets participating is $k\ell$.
Therefore, the number $W(s, \FF_{J, k, \ell})$  of distinct sets  $\{y_1, \dots, y_s\} \cap \imath, \,\,\imath \in \FF_{J, 1, 1}$ is less or equal to $ W(s, \FF_{J, 1, 1})^{k\ell}$, which is in turn less or equal to $s^{(g+2)k\ell}$.

\end{proof}

 By  Observation~\ref{obs:11}, for a random $R \in \Re$, the expected number of sets of the form $R X \cap \imath, \,\,\imath \in  \FF_{J, k, \ell}$ is greater or equal to $2^{s-1}$. Therefore, there exists an $R \in \Re$ such that
the number of sets of the form $R X \cap \imath, \,\,\imath \in  \FF_{J, k, \ell}$ is greater or equal to $2^{s-1}$. Fix such an $R$ and set  $J := \mathrm{span}(RX)$.
By Claim~\ref{cl:min},
\beq \lab{eq:2s} 2^{s-1} \leq  s^{k\ell(g+2)}.\eeq
Therefore $s - 1 \leq k\ell(g+2) \log s.$
Assuming without loss of generality that $s \geq k \ell$, and substituting $C_1\left({\gamma}^{-2} \log(s + k\ell)\right)$ for $g$, we see that
\beqs s \leq  O\left(k\ell {\gamma}^{-2} \log^2 s \right),\eeqs and hence
\beqs \frac{ s}{\log^2(s)} \leq O\left( \frac{k \ell}{\gamma^2}\right),\eeqs
implying that \beqs s \leq O\left(\left(\frac{k \ell}{\gamma^2}\right) \log^2 \left(\frac{k \ell}{\gamma}\right)\right). \eeqs
Thus, the fat shattering dimension  $\fat_\gamma({\FF_{k, \ell}})$ is $O\left(\left(\frac{k \ell}{\gamma^2}\right) \log^2 \left(\frac{k \ell}{\gamma}\right)\right).$ We independently know that $\fat_\gamma({\FF_{k, \ell}})$ is $0$ for $\gamma > 2$.

Therefore by  
 Lemma~\ref{lem:fat_to_gen},
 if \beq\lab{eq:integral} s \geq \frac{C}{\eps^2}\left(\left(\int_{{c\eps}}^2 \frac{\sqrt{k\ell \log^2(k\ell/\gamma^2)}}{\gamma}d\gamma\right)^2 + \log 1/\de\right) ,\eeq
then,
$$\p\left[\sup_{f \in {\F}} \bigg| \E_{\mu_s} f(x_i) - \E_\mu  f\bigg| \geq \eps\right] \leq 1 - \de.$$
Let $t = \ln \left(\frac{\sqrt{k \ell}}{\gamma}\right)$. Then the integral in (\ref{eq:integral}) equals
\beqs \sqrt{k\ell} \int_{\ln (Ck\ell/\eps^2)}^{\ln (\sqrt{kl}/2)} {-t} dt < C \sqrt{kl} \left({\ln (Ck\ell/\eps^2) }\right)^{2}, \eeqs
and so if
\beqs s \geq \frac{C}{\eps^2}\left(k\ell\ln^4\left(k\ell/\eps^2\right) + \log 1/\de\right),\eeqs
then
$$\p\left[\sup_{f \in {\F}} \bigg| \E_{\mu_s} f(x_i) - \E_\mu  f\bigg| \geq \eps\right] \leq 1 - \de.$$
\end{proof}

In order to prove Theorem~\ref{thm:ext_manifold}, we  relate the empirical squared loss $s^{-1}\sum_{i=1}^s  \dist(x_i, \M)^2$ and the expected squared loss over  a class of manifolds whose covering numbers at a  scale $\eps$ have a specified upper bound.
Let $U:\RR^+ \ra \Z^+$ be a real-valued function. Let $\TG$ be any
family of subsets  of the unit ball  $B_\HH$ in a Hilbert space $\HH$ such that for all $r>0$ every element $\M \in \TG$ can be covered  using $U(\frac{1}{r})$ open Euclidean balls.

A priori, it is unclear if
\beq\lab{eq:rv?} \sup\limits_{\M \in \TG} \bigg|\frac{\sum_{i=1}^s \dist(x_i, \M)^2}{s} - \E_\PP \dist(x, \M)^2 \bigg|,\eeq is a random variable, since the supremum of a set of random variables is not always a random variable (although if the set is countable this is true).  Let  $\dhaus$ represent Hausdorff distance. For each $n \geq 1$, $\TG_n$ be a countable set of finite subsets of $\HH$, such that for each $\MM \in \TG$, there exists $\MM' \in \TG_n$ such that $\dhaus(\MM, \MM') \leq 1/n,$ and for each $\MM' \in \TG_n$, there is an $\MM'' \in \TG$ such that $\dhaus(\MM'', \MM') \leq 1/n$. For each $n$, such a $\TG_n$ exists because $\HH$ is separable.
Now (\ref{eq:rv?}) is equal to \beqs \lim_{n \ra \infty} \sup\limits_{\M' \in \TG_n} \bigg|\frac{\sum_{i=1}^s \dist(x_i, \M_n)^2}{s} - \E_\PP \dist(x, M_n)^2 \bigg|,\eeqs and for each $n$, the supremum in the limits is over a countable set; thus, for a fixed $n$, the quantity in the limits is a random variable. Since the pointwise limit of a sequence of measurable functions (random variables) is a measurable function (random variable), this proves that \beqs \sup\limits_{\M \in \TG} \bigg|\frac{\sum_{i=1}^s \dist(x_i, \M)^2}{s} - \E_\PP \dist(x, \M)^2 \bigg|,\eeqs is a random variable.
\begin{lemma}\lab{lem:main1}
 Let $\eps$ and $\de$  be  error parameters. Let $U_\G:\RR^+ \ra \R^+$ be a function taking values in the positive reals. Suppose every  $\MM \in \G(d, V, \tau)$ can be covered  by the union of some $U_\G(\frac{1}{r})$ open Euclidean balls of radius $\frac{\sqrt{r\tau}}{16}$, for every $r > 0$.  If
$$s \geq   C\left(\frac{U_\G(1/{\eps})}{\eps^2} \left(\log^4 \left(\frac{U_\G(1/{\eps})}{\eps}\right)\right) + \frac{1}{\eps^2} \log \frac{1}{\de}\right),$$
Then,
\beqs \p\left[\sup\limits_{\MM \in \G(d, V, \tau)} \bigg|\frac{\sum_{i=1}^s \dist(x_i, \MM)^2}{s} - \E_\PP \dist(x, \MM)^2 \bigg| <\eps\right] > 1 - \de.\eeqs
\end{lemma}
\begin{proof}

Given a collection $\cc := \{c_1, \dots, c_k\}$  of points in $\HH$, let \beqs f_\cc(x) := \min_{c_j \in \cc} |x - c_j|^2.\eeqs Let $\F_{k}$ denote the set of all such functions for \beqs\cc = \{c_1, \dots, c_k\} \subseteq B_\HH,\eeqs $B_\HH$ being the unit ball in the Hilbert space.

Consider  $\MM \in \G := \G(d, V, \tau)$. Let $\cc(\MM, \eps) = \{\hat{c}_1, \dots, \hat{c}_{\hat{k}}\}$ be a set of $\hat{k} := U_\G(1/\eps)$ points in $\MM$, such that $\MM$ is contained in the union of Euclidean balls of radius $\sqrt{\tau \eps}/16$ centered at these points. Suppose $x \in B_{\HH}$.
Since $c(\MM, \eps) \subseteq \MM$,  we have $\dist(x, \MM) \leq \dist(x, c(\MM, \eps))$. To obtain a bound in the reverse direction,  let $y \in \MM$ be a point such that $|x-y| = \dist(x, \MM)$, and let $z \in \cc(\MM, \eps)$ be a point such that $|y - z| < \sqrt{\tau \eps}/16$.
 Let $z'$ be the point on $Tan(y, \MM)$ that is closest to $z$.
 By the reach condition, and Proposition~\ref{thm:federer},
 \begin{eqnarray*}
 |z-z'| & = &  \dist(z, Tan(y, \MM))\\
 & \leq & \frac{|y-z|^2}{2\tau}\\
 & \leq & \frac{\eps}{512}.
 \end{eqnarray*}
 Therefore,
 \begin{eqnarray*} 2\langle y-z, x-y\rangle & = &   2\langle y - z' + z'-z, x-y\rangle\\
 & = & 2 \langle z'-z, x-y\rangle\\
 & \leq & 2 |z-z'||x-y|\\
 & \leq & \frac{\eps}{128}. \end{eqnarray*}
 Thus \begin{eqnarray*} \dist(x, \cc(\MM, \eps))^2 & \leq & |x- z|^2\\
 & \leq & |x-y|^2 + 2\langle y-z, x-y\rangle + |y-z|^2\\
 & \leq \dist(x, \MM)^2 + \frac{\eps}{128} + \frac{\eps\tau}{256}.
  \end{eqnarray*}
  Since $\tau < 1$, this shows that
  $$\dist^2(x, \MM) \leq \dist^2(x, \cc(\MM, \eps)) \leq \dist^2(x, \MM) + \frac{\eps}{64}.$$
Therefore,
\beq \lab{eq:one_j} \p\left[\sup\limits_{\MM \in \G} \bigg|\frac{\sum_{i=1}^s \dist(x_i, \MM)^2}{s} - \E_\PP \dist(x, \MM)^2 \bigg| < {\eps}\right] > \p\left[\sup\limits_{f_\cc(x) \in \F_{ \hat{k}}} \bigg|\frac{\sum_{i=1}^s f_\cc(x_i)}{s} - \E_\PP f_\cc(x_i) \bigg| < \frac{\eps}{3}\right].\eeq
Inequality (\ref{eq:one_j}) reduces the problem of deriving uniform bounds over a space of manifolds to a problem of deriving uniform bounds for $k-$means. (For the best previously known bound for $k-$means, see \cite{MaurerPontil}.)

Let \beqs\Phi:\mathbf{x} \mapsto 2^{-1/2}(\mathbf{x}, 1)\eeqs map a point $x \in \HH$ to one in $\HH\oplus\R$, which we equip with the natural Hilbert space structure. For each $i$, let
\beq \lab{eq:tci}\tc_i := 2^{-1/2}(c_i, \frac{\|c_i\|^2}{2}).\eeq The factor of $2^{-1/2}$ (which could have been replaced by a slightly larger constant) is present because we want $\tc_i$ to belong to to the unit ball.
Then,
$$f_\cc(x) = |x|^2 + 4 \min(\langle \Phi(x), \tc_1\rangle, \dots, \langle\Phi(x), \tc_k\rangle).$$

 Let $\F_{\Phi}$ be the set of functions of the form  $4\min_{i= 1}^k \Phi(x)\cdot \tc_i$ where $\tc_i$ is given by (\ref{eq:tci}) and \beqs \cc = \{c_1, \dots, c_k\} \subseteq B_\HH.\eeqs
The metric entropy of the function class obtained by translating  $\F_\Phi$   by adding $|x|^2$ to every function in it is the same as the metric entropy of $\F_\Phi$. Therefore the integral of the square root of the metric entropy of functions in ${\F}_{\cc,k}$ can be bounded above,
%
and  by Lemma~\ref{lem:key}, if
\beqs s\geq  C\left(\frac{k}{\eps^2} \left(\log^4 \left(\frac{k}{\eps}\right)\right) + \frac{1}{\eps^2} \log \frac{1}{\de}\right),\eeqs then
\beqs  \p\left[\sup\limits_{\MM \in \G} \bigg|\frac{\sum_{i=1}^s \dist(x_i, \MM)^2}{s} - \E_\PP \dist(x, \MM)^2 \bigg| < \eps\right] > 1 - \de.\eeqs
\end{proof}

\begin{proof}[Proof of Theorem~\ref{thm:ext_manifold}]
This follows immediately from Corollary~\ref{cor:r-net} and Lemma~\ref{lem:main1}.
\end{proof}

\section{Dimension reduction}\lab{sec:dim_red}
  Suppose that $X= \{x_1, \dots, x_s\}$ is a set of  i.i.d random points drawn from $\PP$, a probability measure supported in the unit ball $B_\HH$ of a separable Hilbert space $\HH$.
 Let $\MM_{erm}(X)$ denote a manifold in $\G(d, V, \tau)$ that (approximately) minimizes  $$\sum_{i = 1}^s \dist(x_i, \M)^2 $$ over all $\MM \in \G(d, V, \tau)$ and  denote by $\PP_X$ the probability distribution on $X$ that assigns a probability of $1/s$ to each point.
More precisely, we know from  Theorem~\ref{thm:ext_manifold} that there is some function  $s_\G(\eps, \de)$  of $\epsilon, \delta, d, V$ and $\tau$ such that if
 $$s \geq   s_\G(\eps, \de)$$
 then,
 \beq \lab{eq:one_i_repeat1} \p\left[ \LL(\M_{erm}(X), \PP_X) -  \inf_{\M \in \G} \LL(\M, \PP) < {\eps}\right] > 1 - \de.\eeq

 \begin{lemma}\lab{lem:dimred}
Suppose $\eps < c\tau$. Let $W$ denote an arbitrary $2s_\G(\eps, \de)$ dimensional linear subspace of $\HH$ containing $X$. Then
 \beq \inf\limits_{\G(d, V, \tau(1 - c))\ni \MM \subseteq W} \LL(\MM, \PP_X) \leq C \eps + \inf_{\MM \in \G(d, V, \tau)} \LL(\MM, \PP_X).\eeq
 \end{lemma}

 \begin{proof}
 Let $\MM_2 \in \G := \G(d, V, \tau)$ achieve
 \beq \LL(\MM_2, \PP_X) \leq \inf\limits_{\MM \subseteq \G} \LL(\MM, \PP_X) + \eps.\eeq
 Let $N_\eps$ denote a set of no more than $s_G(\eps, \de)$ points contained in $\MM_2$ that is an $\eps-$net of $\MM_2$. Thus for every $x \in \MM_2$, there is $y \in N_\eps$ such that $|x - y| < \eps$.
  Let $O$ denote a unitary transformation from $\HH$ to $\HH$ that fixes each point in $X$ and maps every point in $N_\eps$ to some point in $W$. Let $\Pi_W$ denote the map from $\HH$ to $W$ that maps $x$ to the point in $W$ nearest to $x$. Let $\MM_3 := O \MM_2.$ Since $O$ is an isometry that fixes $X$,
 \beq \LL(\MM_3, \PP_X) = \LL(\MM_2, \PP_X) \leq \inf\limits_{\MM \subseteq \G} \LL(\MM, \PP_X) + \eps.\eeq
 Since $\PP_X$ is supported in the unit ball and the Hausdorff distance between $\Pi_W\MM_3$ and $\MM_3$ is at most $\eps$,
 \begin{eqnarray*}
 \big|\LL(\Pi_W\MM_3, \PP_X) - \LL(\MM_3, \PP_X)\big| & \leq & \E_{x \dashv \PP_X} \big|\dist(x, \Pi_W\MM_3)^2 - \dist(x, \Pi_W\MM_3)^2\big|\\
 & \leq & \E_{x \dashv \PP_X} 4 \big|\dist(x, \Pi_W\MM_3) - \dist(x, \Pi_W\MM_3)\big|\\
 & \leq & 4\eps.
 \end{eqnarray*}
 By Lemma~\ref{lem:manifold_approx}, we see that  $\Pi_W\MM_3$ belongs to $\G(d, V, \tau(1-c))$, thus proving the lemma.
 \end{proof}

By Lemma~\ref{lem:dimred}, it suffices to find a manifold $\G(d, V, \tau) \ni \tilde{M}_{erm}(X) \subseteq V$ such that $$\LL(\tilde{M}_{erm}(X), \PP_X) \leq C\eps + \inf\limits_{V\supseteq\MM\in \G(d, V, \tau)} \LL(\MM, \PP_X).$$

\begin{lemma}\lab{lem:manifold_approx}
Let $\MM \in \G(d, V, \tau)$, and let $\Pi$ be a map that projects $\HH$ orthogonally onto a subspace containing the linear span of a ${c\eps}\tau-$net $\bar S$ of $\MM$. Then, the image of $\MM$, is a $d-$dimensional submanifold of $\HH$ and
$$\Pi(\MM) \in \G(d, V, \tau(1 - C\sqrt{\eps})).$$
\end{lemma}
\begin{proof}
The volume of $\Pi(\MM)$ is no more than the volume of $\MM$ because $\Pi$ is a contraction. Since $\MM$ is contained in the unit ball, $\Pi(\MM)$ is contained in the unit ball.

\begin{claim}\lab{cl:g2sept10}For any   ${x}, {y}\in \MM$, \beqs |\Pi(x - y)| \geq (1- {C\sqrt{\eps}}) |x - y|.\eeqs
\end{claim}
\begin{proof}
First suppose that $|{x}- y| < \sqrt{\eps}\tau$.
Choose $\tx \in \bar S$ that satisfies \beqs|\tilde{x} - x| < C_1 \eps\tau.\eeqs
Let $z := x + \frac{(y - x)\sqrt{\eps}\tau}{|y - x|}$.
By linearity and Proposition~\ref{thm:federer}, \begin{eqnarray} \dist(z, Tan(x, \MM)) &  = & \dist(y, Tan(x, \MM))\left(\frac{\sqrt{\eps}\tau}{|y-x|}\right)\\ &   \leq &  \frac{|x-y|^2}{2\tau}\left(\frac{\sqrt{\eps}\tau}{|y-x|}\right)\\
& \leq & \frac{\eps \tau}{2}.\end{eqnarray}
     Therefore, there is a point $\hat{y} \in Tan(x, \MM)$ such that  \beqs \bigg  | {\hat y} - \left(\tx + \frac{(y - x)\sqrt{\eps}\tau}{|y - x|}\right)\bigg | \leq C_2{\eps}\tau. \eeqs
By  Claim~\ref{cl:g1sept}, there is a point $\bar y \in \MM$ such that  \beqs \bigg  | {\bar y} - \hat y\bigg | \leq C_3{\eps}\tau. \eeqs
Let $\ty\in \bar S$ satisfy \beqs | \ty -  \bar y | < {c\eps}\tau.\eeqs Then,
\beqs \bigg  | {\ty } - \left(\tx + \frac{(y - x)\sqrt{\eps}\tau}{|y - x|}\right)\bigg | \leq C_4{\eps}\tau, \eeqs \ie
\beqs \bigg  | \left(\frac{\ty - \tx}{\sqrt{\epsilon}\tau}\right) - \frac{(y - x)}{|y - x|}\bigg | \leq C_4{\sqrt{\eps}}. \eeqs
Consequently,
\beq\lab{eq:norm:1} \bigg  | \left(\frac{\ty - \tx}{\sqrt{\epsilon}\tau}\right)\bigg | - 1 \leq C_4{\sqrt{\eps}}. \eeq

We now have  \beq  \left\langle \frac{y - x}{|y - x|},  \frac{\ty - \tx}{\sqrt{\epsilon}\tau} \right\rangle & =& \left\langle \frac{y - x}{|y - x|},  \frac{y - x}{|y - x|} \right\rangle +  \left\langle \frac{y - x}{|y - x|},  \left(\frac{\ty - \tx}{\sqrt{\epsilon}\tau} - \frac{y - x}{|y - x|}\right)\right\rangle\\ & = & 1 + \left\langle \frac{y - x}{|y - x|}, \left( \frac{\ty - \tx}{\sqrt{\epsilon}\tau} - \frac{y - x}{|y - x|}\right)\right\rangle \\ & \geq &
1 - C_4 \sqrt{\eps}.\lab{eq:norm:2}\eeq
Since $\tx$ and $\ty$ belong to the range of $\Pi$, it follows from (\ref{eq:norm:1}) and (\ref{eq:norm:2}) that
\beqs |\Pi(x - y)| \geq (1- {C\sqrt{\eps}}) |x - y|.\eeqs

Next, suppose that  $|{x}- y| \geq \sqrt{\eps}\tau$,
Choose $\tx, \ty \in \bar S$ such that $|x - \tx| + |y - \ty| < 2c \eps\tau$.
Then,
\beqs \left\langle \frac{x - y}{|x-y|}, \frac{\tx - \ty}{|\tx - \ty|} \right\rangle & =  &  \left\langle \frac{x - y}{|x-y|}, \frac{x - y}{|\tx - \ty|} \right\rangle + \left(|\tx - \ty|^{-1}\right) \left\langle\frac{x - y}{|x-y|}, (\tx - x) - (\ty - y) \right\rangle\\
& \geq & 1 - C\sqrt{\eps},\eeqs and the claim follows since $\tx$ and $\ty$ belong to the range of $\Pi$.
\end{proof}
 By Claim~\ref{cl:g2sept10}, we see that \beq\lab{eq:g4sept10}\forall \,x \in \MM,\,\, Tan^0(x, \MM) \cap \ker(\Pi) = \{0\}.
\eeq
Moreover, by Claim~\ref{cl:g2sept10}, we see that if $x, y \in \MM$ and $\Pi(x)$ is close to $\Pi(y)$ then $x$ is close to $y$.  Therefore, to examine all $\Pi(x)$ in a neighborhood of $\Pi(y)$, it is enough to examine all $x$ in a neighborhood of $y$.
So by Definition~\ref{def:man},  it follows that $\Pi(\MM)$ is a submanifold of $\HH$.
Finally, in view of Claim~\ref{cl:g2sept10} and the fact that $\Pi$ is a contraction, we see that
\begin{eqnarray}\reach(\Pi(\MM)) & = & \sup_{x, y \in \MM} \frac{|\Pi(x) - \Pi(y)|^2}{2 \dist(\Pi(x), Tan(\Pi(y), \Pi(\MM)))}\\
& \geq & (1 - C\sqrt{\eps})  \sup_{x, y \in \MM} \frac{|x - y|^2}{2 \dist(x, Tan(y, \MM))}\\
& = & (1 - C\sqrt{\eps})\, \reach(\MM),\end{eqnarray}
the lemma follows.
\end{proof}

\section{Overview of the algorithm}\lab{sec:alg-overview}
%
Given a set $X:= \{x_1, \dots, x_s\}$ of points in $\R^n$, we give an overview of the algorithm that finds a nearly optimal interpolating manifold.
\begin{definition}
Let $\MM \in \G(d, V, \tau)$ be called an $\eps-$optimal interpolant if
 \beq\lab{eq:h:1} \sum_{i = 1}^s \dist(x_i, \MM)^2 \leq s{\eps} + \inf\limits_{\MM' \in \G(d,  V/C,  C\tau)}  \sum_{i=1}^s \dist(x_i, \MM')^2, \eeq where $C$ is some  constant depending only on $d$.
\end{definition}
 Given $d, \tau, V, \eps$ and $\de$, our goal is to output an implicit representation of a manifold $\MM$ and an estimated  error $\beps\geq 0$ such that \ben\item With probability greater than $1-\de$, $\MM$ is an $\eps-$optimal interpolant and
  \item
 \beqs \lab{eq:h:2}{s\beps} \leq  \sum\limits_{x \in X} \dist(x, \MM)^2 \leq {s\left(\frac{\eps}{2} +\beps\right) }.\eeqs\een

\begin{figure}\label{fig:bundle}
\begin{center}
\includegraphics[height=2.4in]{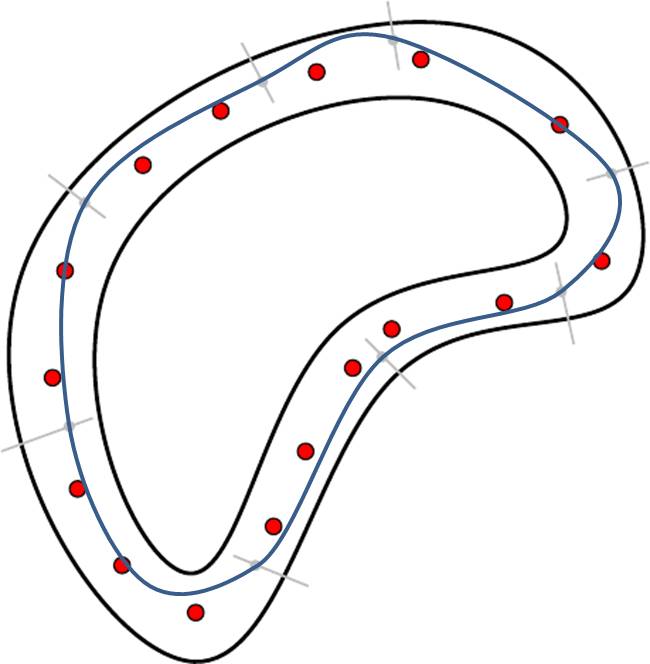}
\caption{A disc bundle $D^\norm \in \D^\norm$}
\end{center}
\end{figure}
Thus, we are required to perform an optimization over the set of manifolds  $\G = \G(d, \tau, V)$. This set $\G$ can be viewed as  a metric space $(G, \dhaus)$ by defining the distance between
two manifolds $\MM, \MM'$ in $\G$ to be the Hausdorff distance between $\MM$ and $\MM'$. The resulting metric space contains a large family of manifolds that are mutually non-homeomorphic. Our strategy for producing an approximately optimal manifold will be to execute the following steps. First identify a $O(\tau)-$net $S_{\G}$ of $(\G, \dhaus)$. Next, for each $\MM' \in S_\G$, construct a disc bundle $D'$ that approximates its normal bundle. The  fiber of $D'$ at a point $z \in \MM'$ is a $n-d-$dimensional disc of radius $O(\tau)$, that is roughly orthogonal to $Tan(z, \MM')$ (this is formalized in Definitions \ref{def:disc_bundle:1} and \ref{def:disc_bundle:2}). Suppose that $\MM$ is a manifold in $\G$ such that \beq\lab{eq:dist-MM'} \dhaus(\MM, \MM') < O(\tau).\eeq  As a consequence of  (\ref{eq:dist-MM'}) and the lower bounds on the reaches of $\MM$ and $\MM'$, it follows (as has been shown in Lemma~\ref{lem:17}) that $\MM$ must be the graph of a section of $D'$. In other words $\MM$ intersects each fiber of $D'$ in a unique point.
We use convex optimization to find good local sections, and patch them up to find a good global section. Thus, our algorithm involves two main phases:
\ben
\item Construct a set $\D^\norm$ of disc bundles over manifolds in $\G(d,  CV, \tau/C)$ is rich enough that every $\eps-$optimal interpolant is a section of some member of $\D^\norm$.
\item  Given $D^\norm \in \D^\norm$, use convex optimization to find a minimal $\hat{\eps}$ such that  $D^\norm$ has a section (\ie a small transverse perturbation of the base manifold of $D^{\norm}$) which is a $\hat{\eps}-$optimal interpolant. This is achieved by finding the right manifold in the vicinity of the base manifold of $D^\norm$ by finding good local sections (using  results from \cite{Feff_ext04, Feff_jets2}) and then patching these up using a gentle partition of unity supported on the base manifold of $D^\norm.$
\een

\section{Disc Bundles}
The following definition specifies the kind of bundles we will be interested in. The constants have been named so as to be consistent with their appearance in  (\ref{eq:ch49}) and Observation~\ref{obs:3}. Recall the parameter $r$ from Definition~\ref{def:man}.
\begin{definition}\lab{def:disc_bundle:1}
 Let  $D$ be an open subset of $\R^n$ and $\MM$ be a submanifold of $D$ that belongs to $\G(d, \tau, V)$ for some choice of parameters $d, \tau, V$. Let $\pi$ be a $\C^{4}$  map $\pi:D \ra \MM$ such that for any $z \in \MM$, $\pi(z) = z$ and $\pi^{-1}(z)$ is isometric to a Euclidean disc of dimension $n-d$, of some radius independent of $z$. We then say $D \stackrel{\pi}{\longrightarrow} \MM$ is a disc bundle. When $\MM$ is clear from context, we will simply refer to the bundle as $D$. We refer to $D_z := \pi^{-1}(z)$ as the fiber of $D$ at $z$. We call  $s:\MM \ra D$ a section of $D$ if for any $z \in \MM$, $s(z) \in D_z$ and for some $\htau, \hV > 0$,  $s(\MM) \in \G(d, \htau, \hV)$. Let $U$ be an open subset of $\MM$. We call a given $C^2-$map $s_{loc}:U \ra D$ a local section of $D$ if for any $z \in U$, $s(z) \in D_z$ and $\{(z, s_{loc}(z))|z \in U\}$ can  locally be expressed as the graph of a $\C^2-$function.
\end{definition}
\begin{definition}\lab{def:disc_bundle:2}

For reals $\htau, \hV >0$, let  $\D( d, \htau, \hV)$ denote the set of all disc bundles $D^\norm \stackrel{\pi}{\longrightarrow}  \MM$ with the following properties.

\ben
\item $D^\norm$ is a  disc bundle over the manifold $\MM \in \G(d, \htau, \hV)$.
\item Let  $z_0 \in \MM$. For $z_0 \in \MM$, let $D^\norm_{z_0} := \pi^{-1}(z_0)$ denote the fiber over $z_0$, and $\Pi_{z_0}$ denote the projection of $\R^n$ onto the affine span of $D^\norm_{z_0}$. Without loss of generality assume after rotation (if necessary) that $Tan(z_0, \MM) = \R^d\oplus\{0\}$ and $Nor_{z_0, \MM} = \{0\}\oplus\R^{n-d}$. Then, $D^\norm \cap B(z_0, \oc_{11}\htau)$ is a bundle over a graph $\{(z, \Psi(z))\}_{z \in \Omega_{z_0}}$ where the domain $\Omega_{z_0}$ is an open subset of $Tan(z_0, \MM)$.
\item Any $z \in B_n(z_0, \oc_{11})$ may be expressed uniquely in the form
 $(x, \Psi(x)) + v$ with $x \in B_d(z_0, \oc_{10}\htau), v \in \Pi_{(x, \Psi(x))} B_{n-d}(x, \frac{\oc_{10}\htau}{2}).$
  Moreover, $x$ and $v$ here are $\C^{{k-2}}-$smooth functions of $z \in B_n(x, \oc_{11}\htau)$, with derivatives up to order ${k-2}$ bounded by $C$ in absolute value.
\item Let $x\in B_d(z_0, \oc_{10}\htau)$, and let $v \in \Pi_{(x, \Psi(x))}\R^n.$ Then, we can express $v$ in the form \beq v = \Pi_{(x, \Psi(x))}v^{\#}\eeq  where $ v^\#\in \{0\}\oplus\R^{n-d}$ and $|v^\#| \leq 2 |v|$.
\een
\end{definition}

\begin{definition}
For any $D^\norm \ra \MM_\base \in \D(d, \htau, \hV)$, and $\a \in (0, 1),$ let $\a\D(d,\htau, \hV)$ denote a bundle over $\MM_\base$, whose every fiber is a scaling by $\a$ of the corresponding fiber of $D^\norm$.
\end{definition}

\section{A key lemma}\lab{sec:charlie}
Given a function with prescribed smoothness, the following key lemma allows us to  construct a bundle satisfying certain conditions, as well as assert that the base manifold has controlled reach. We decompose $\R^n$ as $\R^d \oplus \R^{n-d}$. When we write $(x, y) \in \R^n$, we mean $x \in \R^d$ and $y \in \R^{n-d}$.
\begin{lemma}\lab{lem:charlie}
Let the following conditions hold.
\ben \item Suppose $F:B_n(0, 1) \ra \R$ is   $\C^k-$smooth.
 \item
\beq \lab{eq:C0} \partial^\a_{x, y}F(x, y) \leq C_0 \lab{eq:ch1}\eeq for $(x, y) \in B_n(0, 1)$ and $|\a| \leq k$.
\item For $x \in \R^d$ and $y \in \R^{n-d}$ and $(x, y) \in B_n(0, 1)$, suppose also that
\beq  c_1[|y|^2 + \rho^2] \leq [F(x, y) + \rho^2] \leq C_1[|y|^2 + \rho^2], \lab{eq:ch2}\eeq
where \beq\lab{eq:ch3}0 < \rho < c\eeq where $c$ is a small enough constant determined by $C_0, c_1, C_1, k, n$.
\een
Then there exist constants $c_2, \dots, c_7$ and $C$ determined by $C_0, c_1, C_1, k, n$, such that the following hold.
\ben
\item For $z \in B_n(0, c_2),$ let $\iN(z)$ be the subspace of $\R^n$ spanned by the eigenvectors of the Hessian $\partial^2F(z)$ corresponding to the $(n-d)$ largest eigenvalues. Let $\Pi_{hi}(z):\R^n \ra \iN(z)$ be the orthogonal projection from $\R^n$ onto $N(z)$. Then $|\partial^\a\Pi_{hi}(z)| \leq C$ for $z \in B_n(0, c_2), |\a| \leq {k-2}.$ Thus, $\iN(z)$ depends $\C^{{k-2}}-$smoothly on $z$.
    \item There is a $\C^{{k-2}}-$smooth map \beq\lab{eq:c4c5} \Psi:B_d(0, c_4) \ra B_{n-d}(0, c_3),\eeq with the following properties
        \beq \lab{eq:chnew} |\Psi(0)| \leq C\rho; |\partial^\a\Psi| \leq C^{|\a|} \eeq on $B_d(0, c_4),$ for $1 \leq |\a|\leq {k-2}$. Then, the set of all $z = (x, y) \in B_d(0, c_4) \times B_{n-d}(0, c_3),$ such that \beqs \big\{z|\Pi_{hi}(z)\partial F(z) = 0\} = \{(x, \Psi(x))\big | x \in B_d(0, c_4)\big\}\eeqs  is a $\C^{{k-2}}-$smooth graph.
        \item \lab{item-3:key} We fix $\Psi$ as above.
        Any point $z \in B_n(0, c_7)$ can be expressed uniquely in the form $z = (x, \Psi(x)) + v$, with $x \in B_d(0, c_5), v \in \iN(x, \Psi(x))\cap B_{n}(0, c_6)$. Define
         \beq\lab{eq:c3c4c5}\Phi_d :   B_d(0, c_4) \times B_{n-d}(0, c_3) \ra B_d(0, c_5)\eeq and
 \beqs\Phi_{n-d} :  B_d(0, c_4) \times B_{n-d}(0, c_3) \ra B_{n}(0, c_6)\eeqs
           by  $z = (x, \Psi(x)) + v$. Then, $\Phi_d$ and $\Phi_{n-d}$  are $\C^{{k-2}}-$functions of $z$ and their derivatives of order up to ${k-2}$ are at most $C$ in absolute value.
            \een
            \end{lemma}

\begin{proof}
We first study the gradient and Hessian of $F$. Taking $(x, y) = (0, 0)$ in (\ref{eq:ch2}), we see that

\beq\lab{eq:ch4} c_1 \rho^2 \leq F(0, 0) \leq C_1 \rho^2. \eeq

A standard lemma in analysis asserts that non-negative $F$ satisfying (\ref{eq:ch1}) must also satisfy

\beqs \big | \nabla F(z) \big | \leq C\left(F(z)\right)^\frac{1}{2}. \eeqs

In particular, applying this result to the function $F + \rho^2$, we find that

\beq \lab{eq:ch5} \big | \nabla F(0, 0)\big | \leq C \rho. \eeq

Next, we apply Taylor's theorem : For $(|x|^2 + |y|^2)^\frac{1}{2} \leq \rho^{\frac{2}{3}}$,  for $z = (z_1, \dots, z_n) = (x, y),$ estimates (\ref{eq:ch1}) and (\ref{eq:ch4}) and Taylor's theorem yield
$$\big | F(x, y) + F(-x, -y) - \sum_{i, j = 1}^n \partial_{ij}^2 F(0,0)z_iz_j \big | \leq C\rho^2.$$
Hence, (\ref{eq:ch2}) implies that
$$c|y|^2 - C\rho^2 \leq \sum_{i,j = 1}^n \partial_{ij}^2 F(0, 0)z_i z_j \leq C(|y|^2 + \rho^2).$$
Therefore,
$$c|y|^2 - C\rho^{2/3} |z|^2 \leq \sum_{i,j = 1}^n \partial_{ij}^2 F(0, 0)z_iz_j \leq C\left(|y|^2 + \rho^{2/3} |z|^2\right)$$
for $|z| = \rho^{2/3},$ hence for all $z \in \RR^n.$ Thus, the Hessian matrix $\left(\partial_{ij}^2 F(0)\right)$ satisfies

\beq \lab{eq:ch6}
               \left(
\begin{array}{c|c}
-C\rho^{2/3} &0 \\ \hline
0 &cI
\end{array}\right) \preceq \left( \partial^2_{ij} F(0, 0)\right) \preceq  \left(
\begin{array}{c|c}
+C\rho^{2/3} &0 \\ \hline
0 &CI
\end{array}\right)\eeq

That is, the matrices
$$\left(\partial_{ij}^2 F(0, 0)- \left[-C\rho^{2/3} \delta_{ij} + c\delta_{ij} \one_{i,j > d}\right]\right).$$
and
$$\left(C\left[\rho^{2/3} \delta_{ij} + \delta_{ij} \one_{i,j > d}\right] - \partial_{ij}^2 F(0, 0)\right).$$
are positive definite, real and symmetric.
If $\left(A_{ij}\right)$ is positive definite, real and symmetric, then
$$\big|A_{ij}\big|^2 < A_{ii} A_{jj}$$ for $i \neq j$, since the two--by--two submatrix
$$\left(
  \begin{array}{cc}
    A_{ii} & A_{ij} \\
    A_{ji} & A_{jj} \\
  \end{array}
\right)$$
must also be positive definite and thus has a positive determinant.
It follows from (\ref{eq:ch6}) that
$$\big|\partial_{ii}^2 F(0, 0)\big| \leq C \rho^{2/3},$$ if $i\leq d$, and
$$\big|\partial^2_{jj} F(0, 0)\big| \leq C$$ for any $j$. Therefore, if $i \leq d$ and $j > d$, then
$$\big|\partial_{ij}^2 F(0, 0)\big|^2 \leq \big|\partial_{ii}^2 F(0, 0)\big| \cdot \big|\partial_{jj}^2 F(0, 0)\big| \leq C\rho^{2/3}.$$ Thus,
\beq \lab{eq:ch7} \big| \partial_{ij}^2F(0, 0)\big| \leq C\rho^{1/3}\eeq if $1 \leq i \leq d$ and $d + 1 \leq j \leq n$. Without loss of generality, we can rotate the last $n-d$ coordinate axes in $\R^n$, so that the matrix
$$\left(\partial^2_{ij}F(0, 0)\right)_{i,j = d+1, \dots, n}$$ is diagonal, say,
$$\left(\partial^2_{ij} F(0, 0)\right)_{i,j = d+1, \dots, n} = \left(
  \begin{array}{ccc}
    \lambda_{d+1} & \cdots & 0 \\
    \vdots & \ddots & \vdots \\
    0 & \cdots & \lambda_n \\
  \end{array}
\right).$$
For an $n \times n$ matrix $A = (a_{ij})$, let \beqs \|A\|_\infty := \sup_{(i, j) \in [n]\times[n]} |a_{ij}|.\eeqs
Then (\ref{eq:ch6}) and (\ref{eq:ch7}) show that
\beq\lab{eq:ch8} \left\| \left(\partial_{ij}^2 F(0, 0)\right)_{i,j = 1, \dots, n} - \left(
\begin{array}{c|ccc}
\mathbf{0}_{d\times d} &\mathbf{0}_{d\times 1} &\cdots &\mathbf{0}_{d\times 1}\\ \hline
\mathbf{0}_{1 \times d} &\la_{d+1} &\cdots &0\\
\vdots & \vdots & \ddots & \vdots\\
\mathbf{0}_{1 \times d} &0 &\cdots &\la_n
\end{array}\right)\right\|_{\infty} \leq C\rho^{1/3}\eeq and
\beq\lab{eq:ch9} c \leq \la_j \leq C\eeq for each $j = d+1, \dots, n.$ We can pick controlled constants so that (\ref{eq:ch8}), (\ref{eq:ch9}) and (\ref{eq:ch1}), (\ref{eq:ch3}) imply the following.
\begin{notation}\lab{not:10} For $\la_j$ satisfying (\ref{eq:ch9}), let $c^{\#}$ be a sufficiently small controlled constant.
Let $\Omega$ be the set of all real symmetric $n \times n$ matrices $A$ such that
\beq\lab{eq:ch10.1}\left\|A -
               \left(
\begin{array}{c|ccc}
\mathbf{0}_{d\times d} &\mathbf{0}_{d\times 1} &\cdots &\mathbf{0}_{d\times 1}\\ \hline
\mathbf{0}_{1 \times d} &\la_{d+1} &\cdots &0\\
\vdots & \vdots & \ddots & \vdots\\
\mathbf{0}_{1 \times d} &0 &\cdots &\la_n
\end{array}\right)\right\|_\infty < c^{\#}.\eeq
\end{notation}
Then, $\left(\partial_{ij}^2 F(z)\right)_{i,j = 1, \dots, n}$ for $|z| < \oc_4$ belongs to $\Omega$ by (\ref{eq:ch8}) and (\ref{eq:ch9}). Here $\mathbf{0}_{d\times d}$, $\mathbf{0}_{1\times d}$ and $\mathbf{0}_{d\times 1}$ denote all-zero  ${d\times d}, {1\times d}$ and ${d\times 1}$ matrices respectively.
\begin{definition}\lab{def:oc}
If $A \in \Omega$, let $\Pi_{hi}(A):\R^n \ra \R^n$ be the orthogonal projection from $\R^n$ to the span of the eigenspaces of $A$ that correspond to eigenvalues in $[\oc_2, \oC_3],$ and let $\Pi_{lo}:\R^n \ra \R^n$ be the orthogonal projection from $\R^n$ onto the span of the eigenspaces of $A$ that correspond to eigenvalues in $[-\oc_1, \oc_1].$
\end{definition}
Then, $A \mapsto \Pi_{hi}(A)$ and $A \mapsto \Pi_{lo}(A)$ are smooth maps from the compact set $\Omega$ into the space of all real symmetric $n \times n$ matrices. For a matrix $A$, let $|A|$ denote its spectral norm, \ie
\beqs |A| := \sup_{\|u\| = 1} \|Au\|. \eeqs
Then, in particular,
\beq\lab{eq:ch12} \big|\Pi_{hi}(A) - \Pi_{hi}(A')\big| + \big|\Pi_{lo}(A) - \Pi_{lo}(A')\big|\leq C \big|A - A'\big|. \eeq for $A, A' \in \Omega$, and
\beq \lab{eq:ch13} \big|\partial_A^\a \Pi_{hi}(A)\big| + \big|\partial_A^\a \Pi_{lo}(A)\big| \leq C\eeq for $A \in \Omega, |\a| \leq k.$
Let
\beq\lab{eq:ch14.1} \Pi_{hi}(z) = \Pi_{hi}\left(\partial^2 F(z)\right)\eeq and
\beq\lab{eq:ch14.2} \Pi_{lo}(z) = \Pi_{lo}\left(\partial^2 F(z)\right),\eeq for $z < \oc_4$, which make sense, thanks to the comment following (\ref{eq:ch10.1}). Also, we define projections
$\Pi_d:\R^n \ra \R^n$ and $\Pi_{n-d}: \R^n \ra \R^n$ by setting
\beq\lab{eq:ch15} \Pi_d: (z_1, \dots, z_n) \mapsto (z_1, \dots, z_d, 0, \dots, 0)\eeq and
\beq\lab{eq:ch16} \Pi_{n-d}: (z_1, \dots, z_n) \mapsto (0, \dots, 0, z_{d+1}, \dots, z_{n}).\eeq From $(\ref{eq:ch8})$ and $(\ref{eq:ch12})$ we see that
\beq \lab{eq:ch17} \big|\Pi_{hi}(0) - \Pi_{n-d}\big| \leq C\rho^{1/3}.\eeq Also, (\ref{eq:ch1}) and (\ref{eq:ch13}) together give
\beq \lab{eq:ch18} \big|\partial_z^\a \Pi_{hi}(z)\big| \leq C\eeq for $|z| < \oc_4, |\a| \leq {k-2}$. From (\ref{eq:ch17}), (\ref{eq:ch18}) and (\ref{eq:ch3}), we have
\beq \lab{eq:ch19} |\Pi_{hi}(z) - \Pi_{n-d}| \leq C \rho^{1/3}\eeq for $|z| \leq \rho^{1/3}$.
Note that $\Pi_{hi}(z)$ is the orthogonal projection from $\R^n$ onto the span of the eigenvectors of $\partial^2 F(z)$ with $(n-d)$ highest eigenvalues; this holds for $|z| < \oc_4$.
Now set
 \beq\lab{eq:ch21} \zeta(z) = \Pi_{n-d}\Pi_{hi}\partial F(z)\eeq for $|z| < \oc_4$. Thus
 \beq\lab{eq:ch22} \zeta(z) = (\zeta_{d+1}(z), \dots, \zeta_n(z))\in \R^{n-d},\eeq
 where \beq\lab{eq:ch23}\zeta_i(z) = \sum_{j=1}^n [\Pi_{hi}(z)]_{ij} \partial_{z_j}F(z)\eeq for $i = d+1, \dots, n, |z| < \oc_4$.
 Here, $[\Pi_{hi}(z)]_{ij}$ is the $ij$ entry of the matrix $\Pi_{hi}(z)$. From (\ref{eq:ch18}) and (\ref{eq:ch1}) we see that
 \beq\lab{eq:ch24}|\partial^\a \zeta(z)| \leq C\eeq for $|z| < \oc_4, |\a| \leq {k-2}$.
Also, since $\Pi_{n-d}$ and $\Pi_{hi}(z)$ are orthogonal projections from $\R^n$ to subspaces of $\R^n$, (\ref{eq:ch5}) and (\ref{eq:ch21}) yield
\beq\lab{eq:ch25} |\zeta(0)| \leq c\rho.\eeq
From (\ref{eq:ch23}), we have
\beq\lab{eq:ch26}\frac{\partial \zeta_i}{\partial z_\ell}(z) = \sum_{j=1}^n \frac{\partial}{\partial z_\ell}[\Pi_{hi}(z)]_{ij} \frac{\partial}{\partial z_j}F(z) + \sum_{j=1}^n [\Pi_{hi}(z)]_{ij} \frac{\partial^2 F(z)}{\partial z_\ell \partial z_j}\eeq for $|z| < \oc_4$ and $i = d+1, \dots, n,$ $\ell = 1, \dots, n.$ We take $z=0$ in (\ref{eq:ch26}). From (\ref{eq:ch5}) and (\ref{eq:ch18}), we have
$$\big|\frac{\partial}{\partial z_\ell}[\Pi_{hi}(z)]_{ij}\big| \leq C$$ and $$\big|\frac{\partial}{\partial z_j} F(z)\big| \leq C\rho$$ for $z=0$. Also, from (\ref{eq:ch17}) and (\ref{eq:ch8}), we see that
$$\big|[\Pi_{hi}(z)]_{ij} - \delta_{ij}\big| \leq C \rho^{\frac{1}{3}}$$ for $z=0, i = d+1, \dots, n, j = d+1, \dots, n;$
$$\big|[\Pi_{hi}(z)]_{ij}| \leq C\rho^{1/3}$$ for $z = 0,$ and  $i = d+1, \dots, n$ and $j =  1, \dots, d$;
and
$$\big|\frac{\partial^2 F}{\partial z_j \partial z_\ell}(z) - \de_{j\ell}\lambda_\ell \big| \leq C\rho^{\frac{1}{3}},$$ for $z=0$, $j = 1, \dots, n$, $\ell = d+1, \dots, n$.

In view of the above remarks, (\ref{eq:ch26}) shows that
\beq\lab{eq:ch27} \big|\frac{\partial\zeta_i}{\partial z_\ell} (0) - \la_\ell \delta_{i\ell} \big| \leq C \rho^{1/3}\eeq for $i, \ell = d+1, \dots, n$. Let $B_d(0, r), B_{n-d}(0, r)$ and $B_n(0, r)$ denote the open balls about $0$ with radius $r$ in $\R^d, \R^{n-d}$ and $\R^n$ respectively. Thanks to (\ref{eq:ch3}), (\ref{eq:ch9}), (\ref{eq:ch24}), (\ref{eq:ch25}), (\ref{eq:ch27}) and the implicit function theorem (see Section 3 of \cite{narasimhan}), there exist controlled constants $\oc_6 < \oc_5 < \frac{1}{2}\oc_4$ and a $\C^{{k-2}}-$map
\beq\lab{eq:ch28} \Psi: B_d(0, \oc_6) \ra B_{n-d}(0, \oc_5),\eeq
with the following properties:
\beq\lab{eq:ch29}|\partial^\a \Psi| \leq C\eeq on $B_d(0, \oc_6),$ for $|\a|\leq {k-2}$.
\beq\lab{eq:ch30} |\Psi(0)| \leq C\rho.\eeq
Let $z = (x, y) \in B_d(0, \oc_6) \times B_{n-d}(0, \oc_5).$ Then \beq\lab{eq:ch31}\zeta(z) = 0 \,\text{if and only if}\, y = \Psi(x).\eeq According to $(\ref{eq:ch17})$ and $(\ref{eq:ch18})$, the following holds for a small enough controlled constant $\oc_7$. Let $z \in B_n(0, \oc_7).$ Then $\Pi_{hi}(z)$ and $\Pi_{n-d}\Pi_{hi}(z)$ have the same nullspace. Therefore by (\ref{eq:ch21}), we have the following.
Let $z \in B_n(0, \oc_7)$. Then $\zeta(z) = 0$ if and only if $\Pi_{hi}(z) \partial F(z) = 0$. Consequently, after replacing $\oc_5$ and $\oc_6$ in (\ref{eq:ch28}), (\ref{eq:ch29}), (\ref{eq:ch30}), (\ref{eq:ch31}) by smaller controlled constants $\oc_9 < \oc_8 < \frac{1}{2} \oc_7$, we obtain the following results:
\beq\lab{eq:ch32} \Psi: B_d(0, \oc_9) \ra B_{n-d}(0, \oc_8)\eeq is a $\C^{{k-2}}-$smooth map;
\beq\lab{eq:ch33} |\partial^\a \Psi| \leq C \eeq on $B_d(0, \oc_9)$ for $|\a| \leq k-2$;
\beq\lab{eq:ch34}|\Psi(0)| \leq C \rho;\eeq
Let $$z = (x, y) \in B_d(0, \oc_9) \times B_{n-d}(0, \oc_8).$$ Then, \beq\lab{eq:ch35} \Pi_{hi}(z)\partial F(z) = 0\eeq if and only if $y = \Psi(x)$.
Thus we have understood the set $\{\Pi_{hi}(z) \partial F(z) = 0\}$ in the neighborhood of $0$ in $\R^n$. Next, we study the bundle over $\{\Pi_{hi}(z) \partial F(z) = 0\}$ whose fiber at $z$ is the image of $\Pi_{hi}(z)$.
For $x\in B_d(0, \oc_9)$ and $v = (0, \dots, 0, v_{d+1}, \dots, v_n) \in \{0\}\oplus \R^{n-d}$, we define
\beq\lab{eq:ch36}E(x, v) = (x, \Psi(x)) + [\Pi_{hi}(x, \Psi(x))]v \in \R^n. \eeq
From (\ref{eq:ch18}) and (\ref{eq:ch29}), we have
\beq\lab{eq:ch37}\big|\partial_{x, v}^\a E(x, v)\big| \leq C\eeq for $x \in B_d(0, \oc_9), v \in B_{n-d}(0, \oc_8), |\a| \leq {k-2}.$ Here and below, we abuse notation by failing to distinguish between $\R^d$ and $\R^d\oplus\{0\} \in \R^n$. Let $E(x, v) = (E_1(x, v), \dots, E_n(x, v)) \in \R^n.$ For $i = 1, \dots, d,$ (\ref{eq:ch36}) gives
\beq\lab{eq:ch38}E_i(x, v) = x_i + \sum_{i=1}^n [\Pi_{hi}(x, \Psi(x))]_{ij} v_j.\eeq
For $i = d+1, \dots, n,$ (\ref{eq:ch36}) gives
\beq\lab{eq:ch39}E_i(x, v) = \Psi_i(x) + \sum_{i=1}^n [\Pi_{hi}(x, \Psi(x))]_{ij}v_j,\eeq
where we write $\Psi(x) = (\Psi_{d+1}(x), \dots, \Psi_n(x)) \in \R^{n-d}.$ We study the first partials of $E_i(x, v)$ at $(x, v) = (0, 0).$ From (\ref{eq:ch38}), we find that
\beq\lab{eq:ch40} \frac{\partial E_i}{\partial x_j}(x, v) = \de_{ij}\eeq at $(x, v) = (0, 0)$, for $i,j = 1, \dots, d.$ Also, (\ref{eq:ch34}) shows that $|(0, \Psi(0))| \leq c\rho;$ hence (\ref{eq:ch19}) gives
\beq\lab{eq:ch41} \big|\Pi_{hi}(0, \Psi(0)) - \Pi_{n-d}\big| \leq C\rho^{1/3},\eeq for $i \in \{1, \dots, d\}$ and $j \in \{1, \dots, n\}$. Therefore, another application of (\ref{eq:ch38}) yields
\beq\lab{eq:ch42} \big|\frac{\partial E_i}{\partial v_j}(x, v)\big| \leq C\rho^{1/3}\eeq for $i \in [d], j \in \{d+1, \dots, n\}$ and $(x, v) = (0, 0)$. Similarly, from (\ref{eq:ch41}) we obtain
\beqs\big|[\Pi_{hi}(0, \Psi(0))]_{ij} - \de_{ij}\big| \leq C\rho^{1/3}\eeqs for $i = d+1, \dots, n$ and $j = d+1, \dots, n$. Therefore, from (\ref{eq:ch39}), we have
\beq\lab{eq:ch43}\big|\frac{\partial E_i}{\partial v_j} (x, v) - \de_{ij}\big| \leq C\rho^{1/3}\eeq for $i, j = d+1, \dots, n,$ $(x, v) = (0, 0).$ In view of (\ref{eq:ch37}), (\ref{eq:ch40}), (\ref{eq:ch42}), (\ref{eq:ch43}), the Jacobian matrix of the map
$(x_1, \dots, x_d, v_{d+1}, \dots, v_n) \mapsto E(x, v)$ at the origin is given by

\beq \lab{eq:ch44}
               \left(
\begin{array}{c|c}
I_d &O(\rho^{1/3}) \\\\ \hline\\
O(1) &I_{n-d} + O(\rho^{1/3})
\end{array}\right), \eeq
where $I_d$ and $I_{n-d}$ denote (respectively) the $d\times d$ and $(n-d)\times (n-d)$ identity matrices, $O(\rho^{1/3})$ denotes a matrix whose entries have absolute values at most $C \rho^{1/3}$; and $O(1)$ denotes a matrix whose entries have absolute values at most $C$.\\\\\noindent
A matrix of the form (\ref{eq:ch44}) is invertible, and its inverse matrix has norm at most $C$. (Here, we use (\ref{eq:ch3}).) Note also that that $|E(0, 0)| = |(0, \Psi(0))| \leq C\rho.$ Consequently, the inverse function theorem (see Section 3 of \cite{narasimhan}) and (\ref{eq:ch37}) imply the following.\\\\\noindent

There exist controlled constants $\oc_{10}$ and $\oc_{11}$ with the following properties:
\beq\lab{eq:ch45} \text{The map}\, E(x, v) \,\text{is one-to-one when restricted to\,}
B_d(0, \oc_{10}) \times B_{n-d}(0, \oc_{10}).\eeq
\beq\lab{eq:ch46}\text{The image of}\, E(x, r): B_d(0, \oc_{10}) \times B_{n-d}(0, \frac{\oc_{10}}{2}) \ra \R^n
\text{contains a ball \,}B_n(0, \oc_{11}).\eeq
\beq\lab{eq:ch47}\text{In view of} \,(\ref{eq:ch45}), (\ref{eq:ch46}), \, \text{the map}\eeq $$E^{-1}:B_n(0, \oc_{11}) \ra B_d(0, \oc_{10})\times B_{n-d}(0, \frac{\oc_{10}}{2})$$ is well-defined.
\beq\lab{eq:ch48}\,\text{The derivatives of}\,\, E^{-1}\,\text{of order}\, \leq {k-2}\, \, \text{have absolute value at most}\, C.\eeq
Moreover, we may pick $\oc_{10}$ in (\ref{eq:ch45}) small enough that the following holds.
\begin{obs}\lab{obs:ch49.1}
\beqn\lab{eq:ch49.1}\text{Let\,}\, x\in B_d(0, \oc_{10}), \,\text{and let}\, v \in \Pi_{hi}(x, \Psi(x))\R^n.\eeqn \beqn\lab{eq:ch49}\text{\,Then, we can express\,\,}\,v \text{\,in the form\,\,}\, v = \Pi_{hi}(x, \psi(x))v^{\#} \text{\, where\,} v^\# \in \{0\}\oplus\R^{n-d} \text{\,and\,\,}\, |v^\#| \leq 2 |v|.\eeqn
\end{obs}

Indeed, if $x \in B_d(0, \oc_{10})$ for small enough $\oc_{10}$, then by (\ref{eq:ch3}), (\ref{eq:ch33}), (\ref{eq:ch34}), we have $|(x, \Psi(x))| < c$ for small $c$; consequently, (\ref{eq:ch49}) follows from (\ref{eq:ch17}), (\ref{eq:ch18}). Thus (\ref{eq:ch45}), (\ref{eq:ch46}), (\ref{eq:ch47}), (\ref{eq:ch48}) and (\ref{eq:ch49}) hold for suitable controlled constants $\oc_{10}, \oc_{11}$.
 From (\ref{eq:ch46}), (\ref{eq:ch47}), (\ref{eq:ch49}), we learn the following.
 \begin{obs}\lab{obs:2}
 Let $x, \tix \in B_d(0, \oc_{10})$, and let $v, \tiv \in B_{n-d}(0, \frac{1}{2}\oc_{10}).$ Assume that $v \in \Pi_{hi}(x, \Psi(x))\R^n$ and $\tiv \in \Pi_{hi}(\tix, \Psi(\tix))\R^n$. If $(x, \Psi(x)) + v = (\tix, \Psi(\tix)) + \tiv$, then $x = \tix$ and $v = \tiv$.
 \end{obs}

 \begin{obs}\lab{obs:3}
 Any $z \in B_n(0, \oc_{11})$ may be expressed uniquely in the form
 $(x, \Psi(x)) + v$ with $x \in B_d(0, \oc_{10}), v \in \Pi_{hi}(x, \Psi(x))\R^n \cap B_{n-d}(0, \frac{\oc_{10}}{2}).$
 Moreover, $x$ and $v$ here are $\C^{{k-2}}-$smooth functions of $z \in B_n(0, \oc_{11})$, with derivatives up to order ${k-2}$ bounded by $C$ in absolute value.
 \end{obs}
\end{proof}
\section{Constructing a disc bundle possessing the desired characteristics}\lab{sec:disc-bundle}
\subsection{Approximate squared distance functions}
Suppose that $\MM \in \G(d,   V, \tau)$ is a submanifold of $\RR^n$.
Let \beq\bar \tau := \oc_{12} \tau.\eeq For $\tilde{\tau} > 0,$ let \beqs\MM_{\tilde \tau} := \{z|\inf\limits_{\bar{z} \in \MM} |z - \bar{z}| < \tilde{\tau}\}.\eeqs
Let $\tilde d$ be a suitable large constant depending only on $d$, and which is a monotonically increasing function of $d$. Let \beq\lab{eq:mono} \bar d := \min(n, \tilde d).\eeq
We use a basis for $\RR^n$ that is such that $\R^{\bar d}$ is the span of the first $\bar d$ basis vectors, and $\R^d$ is the span of the first $d$ basis vectors. We denote by $\Pi_{\bar d}$, the corresponding projection of $\R^n$ onto $\R^{\bar d}$.
\begin{definition}\lab{def:12}
Let $\asdf_\MM^{\bar \tau}$ denote the set of all  functions $\bar{F}:\MM_{\bar \tau}\ra \R$ such that  the following is true.
For every $z \in \MM$, there exists an isometry $\Theta_z$ of $\RR^n$ that fixes the origin, and maps $\R^d$ to a subspace parallel to the tangent plane at $z$ such that     $\hat{F}_z:B_{n}(0, 1) \ra \R$    given by
\beq\lab{eq:newF} \hat{F}_z(w) = \frac{\bar{F}(z+ \bar{\tau}\Theta_z( w))}{\bar{\tau}^2},\eeq satisfies the following.

\ben
\item[\underline{\texttt{ASDF-1}}] $\hat{F}_z$
satisfies the hypotheses of Lemma~\ref{lem:charlie} for a sufficiently small controlled constant $\rho$ which will be specified in Equation~\ref{eq:above1} in the proof of Lemma~\ref{lem:put}.  The value of $k$ equals $r+2$,  $r$ being the number in Definition~\ref{def:man}.
\item[\underline{\texttt{ASDF-2}}]   There is a function $F_z: \R^{\bar d} \ra \R$ such that for any  $w\in B_n(0, 1)$, \beq\lab{eq:diff1_12:51}\hat{F}_z(w) =  F_z\left(\Pi_{\bar d}(w)\right) + |w - \Pi_{\bar d}(w)|^2, \eeq where $\R^d \subseteq \R^{\bar d} \subseteq \R^n$.
\een
\end{definition}
 Let \beqs\lab{eq:diff1_6:44}  \Gamma_z = \{w\,|\,\Pi^z_{hi}(w)\partial \hat{F}_z(w) = 0\},\eeqs where $\Pi_{hi}$ is as in Lemma~\ref{lem:charlie} applied to the function $\hat{F}_z$.
\begin{lemma}\lab{lem:put} Let $\bar F$ be in $\asdf_\MM^{\bar \tau}$ and let $\Gamma_z$ and $\Theta_z$ be as in Definition~\ref{def:12}.  \ben
\item The  graph  $\Gamma_z$ is contained in $\R^{\bar d}$. \item Let $c_4$ and $c_5$ be the constants appearing in  (\ref{eq:c4c5}) in Lemma~\ref{lem:charlie}, once we fix $C_0$ in (\ref{eq:C0}) to be $10$, and the constants $c_1$ and $C_1$ (\ref{eq:ch2}) to $1/10$ and $10$ respectively. The "putative" submanifold \beqs \MM_{put} :=\left \{z\in \MM_{\min(c_4, c_5)\bar{\tau}}\big | \Pi_{hi}(z)\partial \bar{F}(z) = 0 \right\}, \eeqs  has a  reach greater than $c \tau$, where $c$ is a controlled constant depending only on $d$.\een Here $\Pi_{hi}(z)$  is the orthogonal projection onto the eigenspace corresponding to eigenvalues in the interval $[\oc_2, \oC_2]$ that is specified in Definition~\ref{def:oc}.
\end{lemma}
\begin{proof}
To see the first part of the lemma, note that  because of (\ref{eq:diff1_12:51}), for any $w \in B_{n}(0, 1)$, the span of the  eigenvectors corresponding to the  eigenvalues of the Hessian of  $F = \hat{F}_z$ that lie in $(\oc_2, \oC_3)$ contains the orthogonal complement of $\R^{\bar d}$ in $\R^n$ (henceforth referred to as  $\R^{n - {\bar d}}$). Further, if $w \not\in \R^{\bar d}$, there is a vector in $\R^{n - \bar{d}}$ that is not orthogonal to the gradient $ \partial \hat{F}_z(w)$. Therefore \beqs \Gamma_z \subseteq \R^{\bar d}.\eeqs

We proceed to the second part of the Lemma. We choose $\oc_{12}$ to be a small enough  monotonically decreasing function of $\bar d$ (by (\ref{eq:mono}) and the assumed monotonicity of $\tilde d$, $\oc_{12}$ is consequently a monotonically decreasing function of $d$) such that for every point $z \in \MM$, $F_z$ given by (\ref{eq:diff1_12:51}) satisfies the hypotheses of Lemma~\ref{lem:charlie} with $\rho < \frac{\tilde{c}{\bar \tau}}{ C^2}$ where $C$ is the constant in Equation~\ref{eq:chnew} and where $\tilde{c}$ is a sufficiently small controlled constant.
Suppose that there is a point $\hat z$ in $\MM_{put}$ such that $\dist(\hat z, \MM)$ is greater than $\frac{min(c_4, c_5)\bar \tau}{2}$, where $c_4$ and $c_5$ are the constants in (\ref{eq:c4c5}). Let $z$ be the unique  point on $\MM$ nearest to $\hat z$. We apply Lemma~\ref{lem:charlie} to $F_z$.
By Equation~\ref{eq:chnew} in Lemma~\ref{lem:charlie}, there is a point $\tilde z \in \MM_{put}$ such that \beq\lab{eq:above1}|z - \tilde z| < C \rho <  \frac{c_{lem} {\bar{\tau}}}{ C}.\eeq The constant $c_{lem}$ is controlled by $\tilde{c}$ and can be made as small as needed provided it is ultimately controlled by $d$ alone. We have an upper bound of  $C$ on the first-order derivatives of $\Psi$ in Equation~\ref{eq:chnew}, which is a function whose graph corresponds via $\Theta_z$ to  $\MM$ in a ${\frac{\bar \tau}{2}}-$neighborhood of $z$. Any unit vector $v \in Tan^0(z)$, is nearly orthogonal to $\tilde z - \hat z$ in that \beq\lab{eq:mnvlu_cnt}\big |\langle \tilde z - \hat z, v \rangle\big | < \frac{ 2 c_{lem}\big |\tilde z - \hat z\big| }{\min(c_4, c_5) C}.\eeq We can  choose $c_{lem}$ small enough that (\ref{eq:mnvlu_cnt})  contradicts the mean value theorem applied to $\Psi$ because of the
upper bound of  $C$ on $|\partial \Psi|$ from Equation~\ref{eq:chnew}.

This shows that for every $\hat z \in \MM_{put}$ its distance to $\MM$ satisfies \beq\lab{eq:hatz} \dist(\hat z, \MM) \leq \frac{\min(c_4, c_5)\bar \tau}{2}.\eeq
 Recall that \beqs \MM_{put} :=\left \{z\in \MM_{{\min(c_4, c_5)\bar{\tau}}}\big | \Pi_{hi}(z)\partial \bar{F}(z) = 0 \right\}. \eeqs Therefore, for every point $\hat{z}$ in $\MM_{put}$, there is a point $z\in \MM$ such that
 \beq B_n\left(\hat{z}, \frac{\min(c_4, c_5)\bar \tau}{2}\right) \subseteq \Theta_z\left(B_d(0, c_4{\bar \tau})\times B_{n-d}(0, c_5{\bar \tau}) \right). \eeq
 We have now shown that $\MM_{put}$ lies not only in $\MM_{{\min(c_4, c_5)\bar{\tau}}}$ but also in $\MM_{\frac{\min(c_4, c_5)\bar{\tau}}{2}}$.
 This fact, in conjunction with (\ref{eq:chnew}) and Proposition~\ref{thm:federer} implies that $\MM_{put}$ is a manifold with reach greater than $c\tau$.

\end{proof}

Let \beq \lab{eq:defbund}\bar{D}^\norm_{\bar{F}} \ra \MM_{put}\eeq be the bundle over  $\MM_{put}$ wherein the fiber at a point $ \hat{z} \in \MM_{put}$, consists of all points $z$ such that \ben\item$|\hat{z}-z| \leq \oc_{12}\tau$, and \item $z-w$ lies in the span of the top $n-d$ eigenvectors of the Hessian of $\bar{F}$ evaluated at $\hat{z}$. \een
\begin{obs}
By Lemma~\ref{lem:charlie},  $\MM$ is a $\C^{r}-$smooth section of ${\bar D}^\norm_{\bar F}$ and  the controlled constants $c_1, \dots, c_7$ and $C$ and depend only on $c_1, C_1, C_0, k$ and $n$ (these constants are identical to those in Lemma~\ref{lem:charlie}). By (\ref{eq:diff1_6:44}), we conclude that the dependence  $n$ can be replaced by a dependence on $\bar d$.
\end{obs}

%

\section{Constructing  cylinder packets}\lab{sec:cyl_pac}

We wish to construct a family of functions $\bar \F$ defined on open subsets of $B_n(0, 1)$ such that for every $\MM \in \G(d, V, \tau)$ such that $\MM \subseteq B_n(0, 1)$, there is some $\hat F \in \bar \F$ such that the domain of $\hat F$ contains $\MM_{\bar \tau}$ and  the restriction of $\hat F$ to $\MM_{\bar \tau}$ is contained in $\asdf_\MM^{\bar \tau}$.

Let $\R^d$ and $\R^{n-d}$ respectively denote the spans of the first $d$  vectors and the last $n-d$  vectors of the canonical basis of $\R^n$. Let $B_d$ and $B_{n-d}$ respectively denote the unit Euclidean balls in $\R^d$ and $\R^{n-d}$.
Let $\Pi_d$ be the map given by the orthogonal projection from $\R^n$ onto $\R^d$. Let $\cyl := {\bar \tau}(B_d \times B_{n-d})$, and $\cyl^{{2}} = {2\bar \tau}(B_d \times B_{n-d})$. Suppose that   for any $x \in 2{\bar \tau}B_d$ and $y \in 2{\bar \tau}B_{n-d}$, $\phi_{\cyl^2}: \R^d\oplus\R^{n-d} \ra \R$ is given by  \beqs \phi_{\cyl^2}(x, y) = |y|^2,\eeqs and for any $z\not\in {\cyl^2}$,
\beqs \phi_{\cyl^2}(z) = 0.\eeqs

Suppose for each $i \in [\bar N] := \{1, \dots,\bar N\} $, $x_i \in B_n(0, 1)$ and $o_i$ is a proper rigid body motion, \ie the composition of a proper rotation and translation of $\R^n$ and that $o_i(0) = x_i$.

For each $i \in [\bar N]$, let $\cyl_i := o_i(\cyl),$ and $\cyl^2_i := o_i(\cyl^2).$ Note that $x_i$ is the center of $\cyl_i$.

We say that a set of cylinders $C_p := \{\cyl_1^2, \dots, \cyl_{\bar N}^2\}$ (where each $\cyl_i^2$ is isometric to $\cyl^2$) is a cylinder packet if the following conditions hold true for each $i$.

Let $S_i:= \{\cyl_{i_1}^2, \dots, \cyl_{i_{|S_i|}}^2\}$ be the set of cylinders that intersect $\cyl_i^2$. Translate the origin to the center of $\cyl_i^2$ (\ie $x_i$) and perform a proper Euclidean transformation that puts the $d-$dimensional central cross-section of $\cyl_i^2$ in $\R^d$.

 There exist proper rotations $U_{i_1}, \dots, U_{i_{|S_i|}}$ respectively of the cylinders $ \cyl_{i_1}^2, \dots, \cyl_{i_{|S_i|}}^2$ in $S_i$ such that $U_{i_j}$ fixes the center $x_{i_j}$ of $\cyl_{i_j}^2$  and translations $Tr_{i_1}, \dots, Tr_{i_{|S_i|}}$ such that
 \ben
 \item For each $j \in [|S_i|]$, $ Tr_{i_j} U_{i_j} \cyl^2_{i_j}$ is a translation of $\cyl_i^2$ by a vector contained in $\R^d$.
\item $\big |\left( Id - U_{i_j}\right) v \big | < c_{12}{\bar \tau}|v - x_{i_j}|$,  for each $j$ in $\{1, \dots, |S_j|\}$
\item $|Tr_{i_j}(0)| < C \frac{{\bar \tau}^2}{\tau}$ for each $j$ in $\{1, \dots, |S_j|\}$.
\item $\bigcup_j ( Tr_{i_j} U_{i_j}  \cyl_j) \supseteq B_d(0, 3\bar \tau)$.
\een

We call $\{o_1, \dots, o_{\bar N}\}$ a packet if $\{o_1(\cyl), \dots, o_N(\cyl)\}$ is a cylinder packet.
\section{Constructing an exhaustive family of disc bundles}

We now show how to construct a set $\bar{D}$ of disc bundles rich enough that any manifold $ \MM \in \G(d, \tau, V)$ corresponds to a section of at least one disc bundle in $\bar{D}$. The constituent disc bundles in $\bar D$ will be obtained from cylinder packets.

Define \beq\lab{eq:bmp} \bmp:\R^d \ra [0, 1]\eeq to be a bump function that has the following properties
for any fixed $k$ for a controlled constant $C$.
\ben
\item  For all $\a$ such that $0 < |\a| \leq k$,
 for all $x \in  \{0\}\cup\{x|\,|x| \geq 1\}$  \beqs \partial^\a \bmp(x)= 0,\eeqs and
 for all $x \in  \{x|\,|x| \geq 1\}$  \beqs \bmp(x)= 0.\eeqs
\item for all $x,$ \beqs \big| \partial^\a \bmp (x)\big | < C,\eeqs and
for $|x| < \frac{1}{4}$, \beqs\bmp(x) = 1.\eeqs
\een

\begin{definition}
Given a  Packet $\bar o:= \{o_1, \dots, o_{\bar N}\}$, define $ F^{\bar o}:\bigcup_i \cyl_i \ra \R$ by
\beq \lab{eq:Fhatbaro}{F}^{\bar o}(z) = \frac{\sum\limits_{\cyl^2_i \ni z} \phi_{\cyl^2}(\frac{o_i^{-1}(z)}{2\bar \tau})\bmp\left(\frac{\Pi_d(o_i^{-1}(z))}{2\bar \tau}\right)}{\sum\limits_{\cyl^2_i \ni z} \bmp\left(\frac{\Pi_d(o_i^{-1}(z))}{2\bar \tau}\right)}.\eeq
\end{definition}

\begin{definition}
Let $A_1$ and $A_2$ be two $d-$dimensional affine subspaces of $\R^n$ for some $n \geq 1$, that respectively contain points $x_1$ and $x_2$. We define $\sphericalangle(A_1, A_2)$, the "angle between $A_1$ and $A_2$", by
\beqs \sphericalangle(A_1, A_2) := \sup_{x_1 + v_1 \in A_1\setminus x_1}\left( \inf_{x_2 + v_2 \in A_2\setminus x_2 } \arccos\left(\frac{\langle v_1, v_2 \rangle}{\|v_1\|\|v_2\|}\right)\right).\eeqs
\end{definition}

Let $\MM$ belong to  $\G(d, V, \tau)$. Let $Y := \{y_1, \dots, y_{\bar N}\}$ be a maximal subset of $\MM$ with the property that no two distinct points are at a distance of less than $\frac{\bar \tau}{2}$ from each other. We construct an {\it ideal} cylinder packet $\{\cyl_1^2, \dots, \cyl_{\bar N}^2\}$ by fixing the center of $\cyl_i^2$ to be $y_i$, and fixing their orientations by the condition that for each cylinder $\cyl^2_i$, the $d-$dimensional central cross-section is a tangent disc to the manifold at $y_i$. Given an ideal cylinder packet, an {\it admissible} cylinder packet corresponding to $\MM$ is obtained by perturbing the the center of each cylinder by less than $c_{12}{\bar \tau}$ and applying arbitrary unitary transformations to these cylinders whose difference with the identity has a  norm less than $C\frac{{\bar \tau}^2}{\tau}.$

\begin{lemma}\lab{lem:17}
Let $\MM$ belong to  $\G(d, V, \tau)$ and let $\{\cyl_1, \dots, \cyl_{\bar N}\}$ be an admissible packet corresponding to $\MM$.

Then,
\beqs {F}^{\bar o} \in \asdf_\MM^{\bar \tau}.\eeqs
\end{lemma}

\begin{proof}

Recall that
$\asdf_\MM^{\bar \tau}$ denotes the set of all $\bar{F}:\MM_{\bar \tau}\ra \R$ (where $\bar{\tau} = \oc_{12} \tau$ and $\MM_{\bar \tau}$ is a $\bar \tau-$neighborhood of $\MM$) for which the following is true:
\bit
 \item For every $z \in \MM$, there exists an isometry $\Theta$ of $\HH$ that fixes the origin, and maps $\R^d$ to a subspace parallel to the tangent plane at $z$ satisfying the conditions below.\\
Let $\hat{F}_z:B_{n}(0, 1) \ra\R$   be given by
\beqs \hat{F}_z(w) = \frac{\bar{F}(z+ \bar{\tau}\Theta( w))}{\bar{\tau}^2}. \eeqs Then, $\hat{F}_z$

\ben \item
satisfies the hypotheses of Lemma~\ref{lem:charlie} with $k = r + 2$.
 \item  For any  $w\in B_n$, \beq\lab{eq:diff1_oct_12:50}\hat{F}_z(w) =  F_z\left(\Pi_{\bar d}(w)\right) + |w - \Pi_{\bar d}(w)|^2, \eeq where $\R^n \supseteq \R^{\bar d} \supseteq \R^d$, and $\Pi_{\bar d}$ is the projection of $\R^n$ onto $\R^{\bar d}$.
\een
\eit
For any fixed $z \in \MM$, it suffices to check that there exists a proper isometry $\Theta$ of $\HH$ such that :  \ben \item[(A)] The hypotheses of Lemma~\ref{lem:charlie} are satisfied by \beq\lab{eq:fzobar} \hat{F}_z^{\bar o}(w) :=  \frac{F^{\bar{o}}(z+ \bar{\tau}\Theta(w))}{\bar{\tau}^2},\eeq
and \item[(B)] \beqs \hat{F}_z^{\bar o}(w) =  \hat{F}_z^{\bar o}\left(\Pi_{\bar d}(w)\right) + |w - \Pi_{\bar d}(w)|^2, \eeqs where $\R^n \supseteq \R^{\bar d} \supseteq \R^d$, and $\Pi_{\bar d}$ is the projection of $\R^n$ onto $\R^{\bar d}$.
\een

We begin by checking the condition (A).
It is clear that $\hat{F}_z^{\bar o}:B_n(0, 1) \ra \R$ is $\C^k-$smooth.

Thus, to check  condition (A), it suffices to establish the following claim.
 \begin{claim}\lab{cl:4.1,4.2}
There is a constant $C_0$ depending only on $d$ and $k$ such that
\ben
\item[\texttt{C4.1}] $\partial^\a_{x, y}\hat{F}_z^{\bar o}(x, y) \leq C_0$ for $(x, y) \in B_n(0, 1)$ and $1 \leq |\a| \leq k$.
\item[\texttt{C4.2}] For $(x, y) \in B_n(0, 1)$,
$$c_1[|y|^2 + \rho^2] \leq [\hat{F}_z^{\bar o}(x, y) + \rho^2] \leq C_1[|y|^2 + \rho^2],$$
where, by making $c_{12}$  and $\oc_{12}$ sufficiently small we can ensure that $\rho> 0$ is less than any  constant determined by $C_0, c_1, C_1, k, {d}$.
\een\end{claim}
\begin{proof}
That the first part of the  claim, \ie (\texttt{C4.1}) is true follows from the chain rule and the definition of $\hat{F}_z^{\bar o}(x, y)$ after rescaling by $\bar \tau$. We proceed to show (\texttt{C4.2}). For any $i \in [\bar N]$ and any vector $v$ in $\R^d$,
For $\rho$ taken to be the value from Lemma~\ref{lem:charlie}, we see that
 for a sufficiently small value  of $\oc_{12} = \frac{\bar \tau}{\tau}$ (controlled by $d$ alone), and a sufficiently small controlled constant as the value of $c_{12}$, (\ref{eq:159}) and (\ref{eq:160})  follow because  $\MM$ is a manifold of reach greater or equal to $\tau$, and consequently Proposition~\ref{thm:federer} holds true.
 \beq \lab{eq:159}|x_i - \Pi_\MM x_i| < \frac{\rho}{100}.\eeq
\beq \lab{eq:160} \sphericalangle\left(o_i(\R^d), Tan(\Pi_\MM(x_i), \MM)\right) \leq  \frac{\rho}{100}.\eeq
Making use of Proposition~\ref{thm:federer} and Claim~\ref{cl:g1sept}, we see that
 for any $x_i, x_j$ such that $|x_i - x_j| < 3 \bar \tau$,
\beq \lab{eq:161}\sphericalangle\left(Tan(\Pi_\MM(x_i), \MM), Tan(\Pi_\MM(x_j), \MM)\right) \leq  \frac{3 \rho}{100}.\eeq
 The inequalities (\ref{eq:159}), (\ref{eq:160}) and (\ref{eq:161}) imply ({\texttt{C4.2}}), completing the proof of the claim.
\end{proof}
We proceed to check condition (B). This holds because for every point $z$ in $\MM$, the number of $i$ such that the cylinder $\cyl_i$ has a non-empty intersection with a ball of radius $2\sqrt{2}(\bar \tau)$ centered at $z$ is bounded above by a controlled constant (\ie a quantity that depends only on $d$). This, in turn, is because $\MM$ has a reach of $\tau$ and no two distinct $y_i, y_j$ are at a distance less than $\frac{\bar \tau}{2}$ from each other. Therefore, we can choose $\Theta$ so that $\Theta(\Pi_{\bar d}(w))$ contains the linear span of the $d-$dimensional cross-sections of all the cylinders containing $z$. This, together with the fact that $\HH$ is a  Hilbert space, is sufficient to yield condition $(B)$.
The Lemma now follows.
\end{proof}


\begin{definition}
Let  $\bar \F$ be set of all functions $F^{\bar o}$ obtained as $\{\cyl^2_i\}_{i \in [\bar{N}]}$  ranges over all  cylinder packets centered on points of a lattice whose spacing is a controlled constant multiplied by $\tau$ and the orientations are chosen arbitrarily from a net of the Grassmannian manifold $Gr_d^n$ (with the usual Riemannian metric) of scale that is a sufficiently small controlled constant.
 \end{definition}
By Lemma~\ref{lem:17}   $\bar \F$  has the following property:\\
\begin{corollary}\lab{cor:asdf_rich}For every $\MM \in \G$ that is a $\C^r-$submanifold, there is some $\hat F \in \bar \F$ that is an approximate-squared-distance-function for $\MM$, \ie the restriction of $\hat F$ to $\MM_{\bar \tau}$ is contained in $\asdf_\MM^{\bar \tau}$.
\end{corollary}

%

\section{Finding good local sections}\lab{sec:loc_sec}
\begin{definition}\lab{def:interpolant}
Let $(x_1, y_1), \dots, (x_N, y_N)$ be ordered tuples belonging to $B_d \times B_{n-d}$, and let $r \in \N$. Recall that by definition~\ref{def:man}, the value of $r$ is $2$. However, in the interest of clarity, we will use the symbol $r$ to denote the number of derivatives.
We say that that a function \beqs f: B_d \ra B_{n-d}\eeqs is an $\eps-$optimal interpolant if the $\C^r-$norm of $f$ (see Definition~\ref{def:crnorm})) satisfies \beqs \|f\|_{\C^r} \leq c,\eeqs and
 \beq\lab{eq:h:func} \sum_{i = 1}^N |f(x_i) - y_i|^2 \leq C N{\eps} + \inf\limits_{\{\check{f}:\|\check{f}\|_{\C^r} \leq \, C^{-1}c\}}  \sum_{i=1}^N |\check{f}(x_i) - y_i|^2, \eeq where $c$ and $C > 1$ are some  constants depending only on $d$.
\end{definition}
\begin{figure}\label{fig:cylinder}
\begin{center}
\includegraphics[height=2.4in]{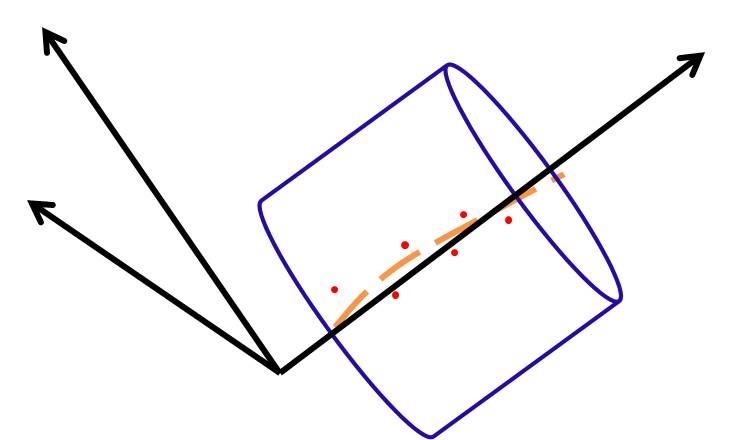}
\caption{Optimizing over local sections.}
\end{center}
\end{figure}
\subsection{Basic convex sets}
We will denote the codimension $n-d$ by $\bar n$.
It will be convenient to introduce the following  notation. For some $i \in \N$, an "$i-$Whitney field" is a family
$\vec{P} = \{P^x\}_{x\in E}$ of $i$ dimensional vectors of real-valued polynomials $P_x$
 indexed by the points $x$ in a finite set
$E \subseteq  \R^d$.
We say that
$\vec{P} = (P_x)_{x\in E}$ is a Whitney field "on $E$", and we write $\wh^{\bar{n}}_r(E)$ for the
vector space of all $\bar{n}-$Whitney fields on $E$ of degree at most $r$.
 \begin{definition}
 Let $\C^r(\R^d)$ denote the space of all real functions on $\R^d$ that are $r-$times continuously differentiable and
 $$\sup_{|\a| \leq r} \sup_{ x \in \R^d } {|\partial^\a f\big|_x|} < \infty.$$

  For a closed subset $U \in \R^d$ such that $U$ is the closure of its interior $U^o$, we define the $\C^{r}-$norm of a function $f:U \ra \R$ by
 \beq \lab{eq:crnorm} \|f\|_{\C^{r}(U))} := \sup_{|\a| \leq r} \sup_{ x \in U^o} |\partial^\a f\big|_x|.\eeq

When $U$ is clear from context, we will abbreviate $\|f\|_{\C^{r}(U)}$ to $\|f\|_{\C^{r}}$.
 \end{definition}

\begin{definition}\lab{def:crnorm}
 We define $\C^{r}(B_d, B_{\bar{n}})$ to consist of all  $f:B_d \ra B_{\bar{n}}$ such that $f(x) = (f^{1}(x), \dots, f^{\bar n}(x))$ and for each $i\in \bn$, $f_i:B_d \ra \R$ belongs to $\C^r(B_d)$.
  We define the $\C^{r}-$norm of $f(x):=(f^{1}(x), \dots, f^{\bn}(x))$ by
 $$\|f\|_{\C^{r}(B_d, B_{\bn})} := \sup_{|\a| \leq r}\sup_{v \in B_{\bn}} \sup_{ x \in B_d } |\partial^\a (\langle f, v\rangle)\big|_x|.$$

 Suppose $F \in \C^r(B_d)$, and $x \in B_d$, we denote by $J_x(F)$ the polynomial that is the $r^{th}$ order Taylor approximation to $F$ at $x$, and call it the ``jet of $F$ at $x$".
\end{definition}
If
$\vec{P} = \{P_x\}_{x \in E}$ is an $\bar{n}-$Whitney field,  and $F \in \C^r(B_d, B_{\bn})$, then we say
that ``$F$ agrees with $\vec{P}$\,", or ``$F$ is an extending function for
$\vec{P}$\,", provided  $J_x(F) = P_x$
for each
$x \in E$. If $E^+ \supset E$, and $(P^+_x)_{x\in E^+}$ is an $\bar{n}-$Whitney field on $E^+$, we say that $\vec{P}^+$ ``agrees with $\vec{P}$ on $E$" if for all $x \in E$, $P_x = P^+_x$.
We define a $C^r-$norm on $\bar{n}-$Whitney fields as follows. If
$\vec{P} \in \wh^{\bar{n}}_r(E)$,
we define
\beq\lab{Cr_norm} \|\vec{P}\|_{\C^r(E)} = \inf_{F} \|F\|_{\C^r(B_d, B_{\bar{n}})},\eeq where the infimum is taken over all $F \in \C^r(B_d, B_{\bar{n}})$ such that $F$ agrees with $\vec{P}$.

We are interested in the set of all $f \in \C^r(B_d, B_{\bar{n}})$ such that $\|f\|_{\C^r(B_d, B_{\bar{n}})} \leq 1$.
By results of Fefferman (see page 19, \cite{Feff_jets2}) we have the following.

\begin{theorem}\lab{thm:feff1} Given $\eps>0$, a positive integer $r$ and a finite set $E \subset \R^d$, it is possible to  construct in time and space bounded by $\exp(C/\eps)|E|$ (where $C$ is controlled by $d$ and $r$), a set $E^+$ and a convex set $K$ having the following properties.
\bit
\item  Here $K$ is the intersection of $\bar{m} \leq \exp(C/\eps)|E|$ sets $\{x|(\alpha_i(x))^2 \leq \beta_i\}$, where $\a_i(x)$ is a real valued linear function such that $\a(0) = 0$ and  $\beta_i > 0$. Thus
     \beqs K:= \{x|\forall i\in [\bar{m}], \,(\alpha_i(x))^2 \leq \beta_i\}\subset \wh^{1}_r(E^+).\eeqs

\item If $\vec{P} \in \wh^{1}_r(E^+)$  such that $\|\vec{P}\|_{\C^r(E)} \leq 1-\eps$, then
there exists a Whitney field $\vec{P}^+ \in K$, that agrees with $\vec{P}$ on $E$.
\item Conversely, if there exists a Whitney field $\vec{P}^+ \in K$ that agrees with $\vec{P}$ on $E$, then  $\|\vec{P}\|_{\C^r(E)} \leq 1+\eps.$
\eit
\end{theorem}

For our purposes, it would suffice to set the above $\eps$ to any controlled constant. To be specific, we set $\eps$ to $2$.
By Theorem 1 of  \cite{Feff_ext04} we know  the following.
\begin{theorem}\lab{thm:feff2} There exists a linear map $T$ from $\C^{r}(E)$ to  $\C^{r}(\R^d)$ and a controlled constant $C$ such that ${T}f\big|_{E} = f$ and $\|{T}f\|_{\C^{r}( \R^d)} \leq C \|f\|_{\C^{r}(E)}$.
\end{theorem}

\begin{definition}\lab{def:barK}
For $\{\alpha_i\}$ as in Theorem~\ref{thm:feff1}, let
$\bar{K} \subset \bigoplus_{i=1}^{\bn} \wh_r^1(E^+)$ be the set of all  $(x_1, \dots, x_{\bn})\in \bigoplus_{i=1}^{\bn} \wh_r^1(E^+)$ (where each $x_i \in \wh_r^1(E^+)$) such that  for each $i\in [\bar{m}]$ \beqs \sum_{j = 1}^{\bn}(\alpha_i( x_j))^2 \leq \beta_i.\eeqs
\end{definition}

 Thus, $\bar{K}$ is an intersection of $\bar{m}$ convex sets, one for each linear constraint $\alpha_i$. We identify $\bigoplus_{i=1}^{\bn} \wh_r^1(E^+)$ with $\wh_r^{\bar{n}}(E^+)$ via the natural isomorphism. Then, from Theorem ~\ref{thm:feff1} and Theorem~\ref{thm:feff2} we obtain the following.
\begin{corollary}\lab{cor:feff3}
There is a controlled constant $C$ depending on $r$ and $d$ such that
\bit
\item If $\vec{P}$ is a $\bar{n}-$Whitney field  on $E$ such that $\|\vec{P}\|_{\C^r(E, \R^{\bn})} \leq C^{-1}$,   then
there exists a $\bar{n}-$Whitney field $\vec{P}^+ \in \bar{K}$, that agrees with $\vec{P}$ on $E$.
\item Conversely, if there exists a $\bar{n}-$Whitney field $\vec{P}^+ \in \bar{K}$ that agrees with $\vec{P}$ on $E$, then  $\|\vec{P}\|_{\C^r(E, \R^{\bn})} \leq C.$
\eit
\end{corollary}

\subsection{Preprocessing}\lab{ssec:sketch}
Let $\beps >0$ be an error parameter.
\begin{notation}
For $n \in \N$, we  denote the set $\{1, \dots, n\}$ by $[n]$. Let $\{x_1, \dots, x_N\} \subseteq \R^d$.
\end{notation}
Suppose $x_1, \dots, x_N$ is a set of data points in $\R^{\bar d}$ and $y_1, \dots, y_N$ are corresponding values in $\R^{\bn}$.
The following procedure constructs a function $\pare:[N] \ra [N]$ such that $\{x_{\pare(i)}\}_{i \in [N]}$ is an $\beps-$net of $\{x_1, \dots, x_N\}$.
For $i = 1$ to $N$, we sequentially define sets $S_i$, and construct $\pare$.

Let $S_1 := \{1\}$ and $\pare(1) := 1$. For any $i > 1$,
\ben \item if $\{j: j \in S_{i-1} \, \text{and}\, |x_j - x_i| < \beps\} \neq \emptyset$, set $\pare(i)$ to be an arbitrary element of $\{j: j \in S_{i-1} \, \text{and}\, |x_j - x_i| < \beps\}$, and set $S_i := S_{i-1}$,
\item and otherwise  set $\pare(i) := i$ and set $S_i := S_{i-1}\cup\{i\}$.
    \een
Finally, set $S := S_N$, $\hN = |S|$ and for each $i$, let $$\chil(i) := \{j: \pare(j) = i\}.$$
For $i \in S$, let $\mu_i := N^{-1}\big|\chil(i)\big|$, and let \beq\bar{y}_i := \left(\frac{1}{{|\chil(i)|}}\right)\sum_{j \in \chil(i)} {y_j}.\eeq

It is clear from the construction that for each $i\in [N]$, $|x_{\pare(i)} - x_i| \leq \beps$.
The construction of $S$ ensures that the distance between any two points in $S$ is at least $\bar{\eps}$.  The motivation for sketching the data in this manner was that now, the extension problem  involving $E = \{x_i|i \in S\}$ that we will have to deal with will be better conditioned in a sense explained in the following subsection.

\subsection{Convex program}\lab{ssec:convex_program}
Let the indices  in $[N]$ be permuted so that $S = [\hN]$. For any $f$ such that $\|f\|_{\C^2} \leq C^{-1}c$, and $|x-y| < \beps$, we have $|f(x) - f(y)| < \beps$, (and so the grouping and averaging described in the previous section do not affect the quality of our solution), therefore we see that in order to find a $\beps-$optimal interpolant, it suffices to minimize the objective function
\beqs \obj := \sum_{i=1}^{\hN} \mu_i |\bar{y}_i - P_{x_i}(x_i)|^2,\eeqs over all $\vec{P} \in \bar{K} \subseteq \wh_r^{\bar{n}}(E^+)$, to within an additive error of $\beps$, and to find the corresponding point in $\bar{K}$. We note that $\obj$ is a convex function over $\bar{K}$.
\begin{lemma}\lab{lem:conditioning} Suppose that the distance between any two points in $E$ is at least $\beps$. Suppose $\vec{P} \in \wh_r^{1}(E^+)$ has the property that for each $x \in E$,
 every  coefficient of $P_x$  is bounded above by $c'\beps^2$.
Then, if $c'$ is less than some controlled constant depending on  $d$, \beqs \|\vec{P}\|_{\C^2(E)} \leq 1. \eeqs
\end{lemma}
\begin{proof} Let
\beqs f(x) = \sum_{z \in E} \bmp\left(\frac{10(x-z)}{\beps}\right)P_z(x).\eeqs By the properties of $\bmp$ listed above, we see that  $f$  agrees with $\vec{P}$ and that $\|f\|_{\C^2(\R^d)} \leq 1$ if $c'$ is bounded above by a sufficiently small controlled constant.
\end{proof}

Let $z_{opt} \in \bar{K}$ be any point such that \beqs\obj(z_{opt}) = \inf_{z'\in \bar{K}}\obj(z').\eeqs
\begin{obs}\lab{obs:conditioning}
By Lemma~\ref{lem:conditioning} we see that the set $K$ contains a Euclidean ball of radius $c'\beps^2 $ centered at the origin, where $c'$ is a controlled constant depending on  $d$.

It follows that  $\bar{K}$ contains a Euclidean ball of the same radius $c'\beps^2$ centered at the origin. Due to the fact that the the magnitudes of the first $m$ derivatives at any point in $E^+$ are bounded by $C$, every point in $\bar K$ is at a Euclidean distance of at most $C\hN$ from the origin. We can bound $\hN$ from above as follows: $$\hN \leq \frac{C}{\beps^d}.$$
\end{obs}

Thanks to Observation~\ref{obs:conditioning} and facts from Computer Science, we will see in a few paragraphs that the relevant optimization problems are tractable.
\subsection{Complexity}\lab{ssec:convex_complexity}

Since we have an explicit description of $\bar{K}$ as in intersection of cylinders, we can construct a ``separation oracle", which, when fed with $z$, does the following. \bit\item If $z \in \bar{K}$ then the separation oracle  outputs ``Yes." \item If $z \not\in \bar{K}$ then the separation oracle outputs ``No" and in addition outputs a real affine function $a: \wh_r^{\bar{n}}(E^+) \ra \R$ such that $a(z) < 0$ and $\forall z' \in \bar{K}$ $a(z') > 0$. \eit
To implement this separation oracle for $\bar{K}$, we  need to do the following. Suppose we are presented with a point $x = (x_1, \dots, x_{\bn}) \in \wh_r^{\bar{n}}(E^+)$, where each $x_j \in \wh_r^1(E^+)$.
\ben
\item If, for each $i \in [\bar m]$,
\beqs \sum_{j = 1}^{\bn}(\alpha_i( x_j))^2 \leq \beta_i\eeqs holds, then declare that $x \in \bar K$.
\item Else, let there be some $i_0 \in [\bar m]$ such that
\beqs \sum_{j = 1}^{\bn}(\alpha_{i_0}( x_j))^2 \leq \beta_{i_0}.\eeqs
Output the following separating half-space :
\beqs \{(y_1, \dots, y_{\bn}): \sum_{j = 1}^{\bn} \alpha_{i_0}(x_j) \alpha_{i_0}(y_j - x_j) \leq 0\}.\eeqs
\een

The complexity $A_0$ of answering the above query is the complexity of evaluating $\alpha_i(x_j)$ for each $i \in [\bar m]$ and each $j \in [\bn]$. Thus \beq A_0 \leq \bn {\bar m}(dim(K)) \leq C n \hN^2. \eeq

\begin{claim}
For some $a \in \bar K$, \beqs\lab{eq:valzopt0} B(a, 2^{-L}) \subseteq \{z \in \bar{K} |\obj(z) - \obj(z_{opt}) < \beps\} \subseteq B(0, 2^{L}),\eeqs   where $L$ can be chosen so that
$L \leq C(1 + |\log(\beps)|)$.
\end{claim}
\begin{proof}
By Observation~\ref{obs:conditioning}, we see that the diameter of $\bar K$ is at most $C \beps ^{-d}$  and $\bar K$ contains a ball $B_L$ of radius $2^{-L}$. Let the convex hull of $B_L$ and the point $z_{opt}$ be $K_h$. Then,  $$\{z \in {K_h} |\obj(z) - \obj(z_{opt}) < \beps\}  \subseteq \{z \in \bar{K} |\obj(z) - \obj(z_{opt}) < \beps\}$$ because $\bar{K}$ is convex.
Let the set of all $\vec{P} \in \wh_r^{\bar{n}}(E^+)$ at which \beqs \obj := \sum_{i=1}^{\hN} \mu_i |\bar{y}_i - P_{x_i}(x_i)|^2 = 0\eeqs be the affine subspace $H$. Let $f:\wh_r^{\bar{n}}(E^+) \ra \R$ given by $$f(x) = \dist(x, z_{opt}) :=  |x- z_{opt}|,$$ where $|\cdot|$ denotes the Euclidean norm. We  see that the magnitude of the gradient of  $\obj$ is bounded above by $C\hN$ at $z_{opt}$, and the Hessian of $\obj$ is bounded above by the Identity. Therefore,
$$\{z \in {K_h} |\obj(z) - \obj(z_{opt}) < \beps\}  \supseteq \{z \in {K_h} \big | 2 C\hN (f(z)) < \beps\}.$$

We note that $$\{z \in {K_h} \big | 2 C\hN (f(z)) < \beps\} = K_h \cap B\left(z_{opt}, \frac{\beps}{2C\hN}\right),$$ where the right hand side denotes the intersection of $K_h$ with a Euclidean ball of radius $\frac{\beps}{2C\hN}$ and center $z_{opt}$.
By the definition of $K_h$,
$K_h \cap B\left(z_{opt}, \frac{\beps}{2C\hN}\right)$ contains a ball of radius $2^{-2L}$. This proves the claim.
\end{proof}
Given a separation oracle for $\bar{K} \in \R^{\bar{n}(\dim(K))}$ and the guarantee that for some $a \in \bar K$, \beq\lab{eq:valzopt} B(a, 2^{-L}) \subseteq \{z \in \bar{K} |\obj(z) - \obj(z_{opt}) < \beps\} \subseteq B(0, 2^{L}),\eeq   if $\eps > \beps + \obj(z_{opt})$, Vaidya's algorithm (see  \cite{vaidya}) finds a point in $\bar{K} \cap \{z|\obj(z) < \eps\}$ using \beqs O(\dim(\bar{K})A_0L' + \dim(\bar{K})^{3.38}L')\eeqs arithmetic steps, where $L' \leq C(L + |\log(\beps))|)$. Here $A_0$ is the number of arithmetic operations required to answer a query to the separation oracle.

 Let $\eps_{va}$ denote the smallest real number such that
\ben
\item
\beqs \eps_{va} > \beps.
\eeqs
\item For any $\eps > \eps_{va}$, Vaidya's algorithm  finds a point in $\bar{K} \cap \{z|\obj(z) < \eps\}$ using \beqs O(\dim(\bar{K})A_0L' + \dim(\bar{K})^{3.38}L')\eeqs arithmetic steps, where $L' \leq C(1 + |\log(\beps))|)$.
\een

A consequence of (\ref{eq:valzopt}) is that $\eps_{va} \in [2^{-L},  2^{L+1}].$ It is therefore clear that $\eps_{va}$ can be computed to within an additive error of $\beps$ using binary search and $C(L + |\ln \beps|)$ calls to Vaidya's algorithm.

The total number of arithmetic operations is therefore $O(\dim(\bar{K})A_0 L^2 + \dim(\bar{K})^{3.38}L^2)$ where $L \leq C(1 + |\log(\beps)|)$.

\section{Patching local sections together}
\begin{figure}\label{fig:patching}
\begin{center}
\includegraphics[height=2.4in]{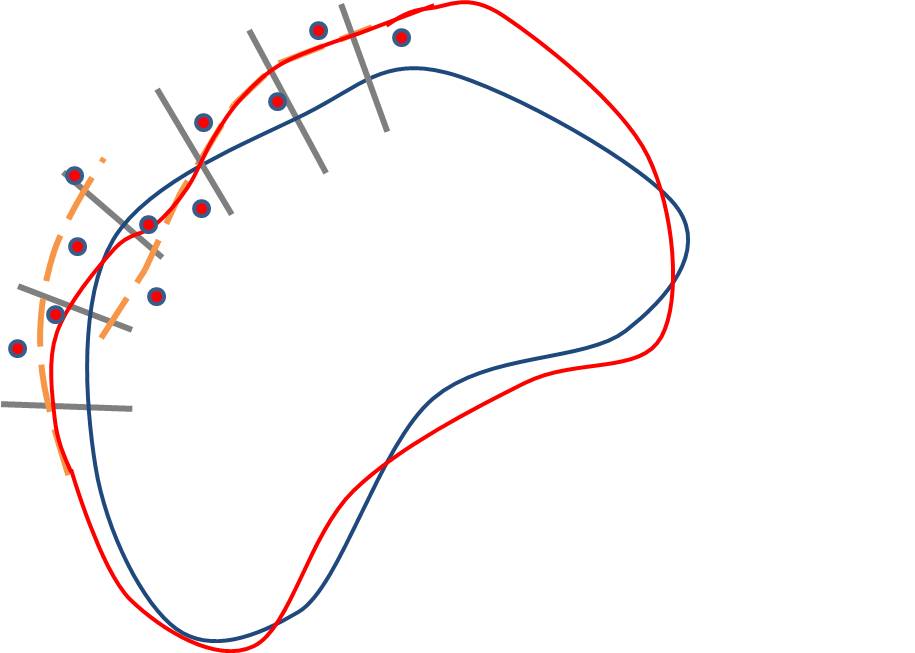}
\caption{Patching local sections together: base manifold in blue, final manifold in red}
\end{center}
\end{figure}

For any $i \in [\bar N]$, recall the cylinders $\cyl_i$ and Euclidean motions $o_i$ from Section~\ref{sec:cyl_pac}.

Let $\base(\cyl_i) := o_i(\cyl\cap \R^d)$ and $\stalk(\cyl_i) := o_i(\cyl \cap \R^{n-d})$. Let $\check{f}_i : B_d \ra B_{n-d}$ be an arbitrary $C^2$ function such that \beq\lab{eq:bound.01}\|\check{f}_i\|_{\C^2} \leq \frac{2\bar{\tau}}{\tau}.\eeq

Let $f_i:\base(\cyl) \ra \stalk(\cyl)$ be given by
\beq\lab{eq:bound.011} f_i(x) = \bar{\tau}\check{f}_i\left(\frac{x}{\bar{\tau}}\right).\eeq

Now, fix an $i \in [\bar N]$.  Without loss of generality, we will drop the subscript $i$ (having fixed this $i$), and assume that $o_i := id$, by changing the frame of reference using a proper rigid body motion.  Recall that  $\hat{F}^{\bar o}$ was defined by (\ref{eq:fzobar}), \ie
\beqs \hat{F}^{\bar o}(w) :=  \frac{F^{\bar{o}}(\bar{\tau} w)}{\bar{\tau}^2},\eeqs (now $0$ and $o_i = id$ play the role that $z$ and $\Theta$ played in (\ref{eq:fzobar})). Let $\iN(z)$ be the linear subspace spanned by the top $n-d$ eigenvectors of the Hessian of $\hat{F}^{\bar o}$ at a variable point $z$.  Let the intersection of
\beqs  B_d(0, 1)\times B_{n-d}(0, 1)\eeqs with \beqs\{\tilde{z}\big |\langle \nabla \hat{F}^{\bar o}\big|_{\tilde{z}}, v \rangle = 0 \,\text{for all} \,v \in \Pi_{hi}(\tilde{z})(\R^n)\}\eeqs be locally expressed as the graph of a function $g_i$, where \beq\lab{eq:defgi} g_i:B_d(0, 1) \ra\R^{n-d}.\eeq
For this fixed $i$, we drop the subscript and let
$g:B_d(0, 1) \ra \R^{n-d}$ be given by \beq\lab{eq:def_g_frm_gi}g:=g_i.\eeq
 As in (\ref{eq:diff1_6:44}), we see that  \beqs \Gamma = \{w\,|\,\Pi_{hi}(w)\partial F^{\bar o}(w) = 0\}\eeqs  lies in $\R^{\bar d},$ and the manifold $\MM_{put}$ obtained by patching up all such manifolds for $i \in [\bar N]$
is, as a consequence of  Proposition~\ref{thm:federer} and Lemma~\ref{lem:charlie}  a submanifold, whose reach is at least $c \tau$.
 Let \beqs \bar{D}^\norm_{\bar{F}^{\bar o}} \ra \MM_{put}\eeqs be the bundle over  $\MM_{put}$ defined by (\ref{eq:defbund}).

Let $s_i$ be the local section of $\bar{D}^\norm := \bar{D}^\norm_{{F^{\bar{o}}}}$ defined by \beq\lab{eq:fixU}\{z + s_i(z)| z \in U_i\} := o_i\left(\left\{x +  f_i(x)\right\}_{x \in \base(\cyl)}\right),\eeq where $U := U_i \subseteq \MM_{put}$ is an open set fixed by (\ref{eq:fixU}). The choice of $\frac{\bar{\tau}}{\tau}$ in (\ref{eq:bound.01}) is small enough to ensure that there is a unique open set $U$ and a unique $s_i$ such that (\ref{eq:fixU}) holds (by Observations \ref{obs:ch49.1}, \ref{obs:2} and \ref{obs:3}). We define $U_j$ for any $j \in [\bar N]$ analogously.
Next, we construct a partition of unity on $\MM_{put}$.
For each $j \in [\bar N]$, let $\tilde{\theta}_j: \MM_{put} \ra [0, 1]$ be an element of a partition of unity defined as follows. For $x \in \cyl_j$,
\beqs \tilde{\theta}_j(x) :=  \left\{
                               \begin{array}{ll}
                                 \bmp\left(\frac{\Pi_d(o_j^{-1} x)}{\bar \tau}\right), & \hbox{if $x \in \cyl_j$;} \\
                                 0 & \hbox{otherwise.}
                               \end{array}
                             \right.\eeqs
where $\bmp$ is defined by (\ref{eq:bmp}).
Let \beqs \theta_j(z) := \frac{\tilde{\theta}_j(z)}{\sum_{j' \in [\bar N]} \tilde{\theta}_{j'}(z)}. \eeqs

We use the local sections $\{s_j|j\in [\bar N]\}$, defined separately for each $j$ by (\ref{eq:fixU}) and the partition of unity $\{\theta_i\}_{i \in \bar N}$, to obtain a global section  $s$ of $D^\norm_{\bar o}$ defined as follows for $x \in U_i$.
\beq \lab{eq:defsx} s(x) := \sum_{j\in [\bar N]} \theta_j(x) s_j(x). \eeq
We also define $f:V_i\ra B_{n-d}$ by
\beq\lab{eq:def-f}\{z + s(z)| z \in {U}_i\} := \left\{x +  \bar \tau f(x/\bar\tau)\right\}_{x \in V_i}.\eeq The above equation fixes an open set $V_i$ in $\R^d$.
The graph of $s$, \ie \beq\lab{eq:defmfin} \left\{(x + s(x))\big |x \in
\MM_{put}\right\} =: \MM_{fin}\eeq is the output manifold. We see that (\ref{eq:defmfin}) defines a manifold $\MM_{fin}$, by checking this locally.
We will obtain a lower bound on the reach of $\MM_{fin}$ in Section~\ref{sec:reach}.

\section{The reach of the output manifold}\lab{sec:reach}
Recall that  $\hat{F}^{\bar o}$ was defined by (\ref{eq:fzobar}), \ie
\beqs \hat{F}^{\bar o}(w) :=  \frac{F^{\bar{o}}(\bar{\tau} w)}{\bar{\tau}^2},\eeqs (now $0$ and $o_i = id$ play the role that $z$ and $\Theta$ played in (\ref{eq:fzobar})).
We place ourselves in the context of  Observation~\ref{obs:3}. By construction, ${F}^{\bar o}:B_n \ra \R$ satisfies the conditions of Lemma~\ref{lem:charlie}, therefore there exists a  map
\beqs {{\Phi}}: B_n(0, \oc_{11}) \ra B_d(0, \oc_{10})\times B_{n-d}\left(0, \frac{\oc_{10}}{2}\right),\eeqs
satisfying the following condition.
\beq \lab{eq:zxv} {\Phi}(z) = (x, \Pi_{n-d} v), \eeq
where $$z = x + g(x) + v,$$ and $$v \in \iN(x + g(x)).$$
Also, $x$ and $v$ are $C^{r}-$smooth functions of $z \in B_n(0, \bar{c}_{11})$.
with derivatives of order up to $r$ bounded above by $C$.
Let  \beq \check{\Phi}: B_n(0, \oc_{11}\bar{\tau}) \ra B_d(0, \oc_{10}\bar{\tau})\times B_{n-d}\left(0, \frac{\oc_{10}\bar{\tau}}{2}\right)\eeq be given by
$$\check{\Phi}(x) = \bar{\tau}\Phi(x/\bar{\tau}).$$

Let $D_{g}$ be the disc bundle over the graph of $g$, whose fiber at $x + g(x)$ is the disc  \beqs B_{n}\left(x + g(x), \frac{\oc_{10}}{2}\right) \bigcap \iN(x + g(x)).\eeqs
By Lemma~\ref{lem:zoom_out} below,  we can ensure, by setting $\oc_{12} \leq \bar c$ for a sufficiently small controlled constant $\bar c$, that the derivatives of $\Phi - id$ of order less or equal to  $r=k-2$ are bounded above by a prescribed controlled constant $c$.

\begin{lemma}\lab{lem:zoom_out}
For any controlled constant $c$, there is a controlled constant $\bar c$ such that if $\oc_{12} \leq \bar c$, then for each $i \in [\bar N]$, and each $|\a| \leq 2$ the functions $\Phi$ and $g$, respectively defined in (\ref{eq:zxv}) and (\ref{eq:def_g_frm_gi}) satisfy
\beqs |\partial^\alpha (\Phi - id)| \leq c.\eeqs
\beqs |\partial^\alpha g| \leq c.\eeqs
\end{lemma}

\begin{proof}[Proof of Lemma~\ref{lem:zoom_out}]
We would like to apply Lemma~\ref{lem:charlie} here, but its conclusion would not directly help us, since it would give a bound of the form
$$|\partial^\alpha \Phi| \leq C,$$ where $C$ is some controlled constant. To remedy this, we are going to use a simple scaling argument. We first provide an outline of the argument. We change scale by "zooming out", then apply Lemma~\ref{lem:charlie}, and thus obtain a bound of the the desired form
\beqs |\partial^\alpha (\Phi - id)| \leq c.\eeqs
We replace each cylinder $\cyl_j = o_j(\cyl)$ by  $\check{\cyl}_j := o_j( \bar{\tau}(B_d \times (\check{C} B_{n-d})))$.
 Since the guarantees provided by Lemma~\ref{lem:charlie} have an unspecified dependence on $\bar d$ (which appears in (\ref{eq:diff1_oct_12:50})), we require an upper bound on the "effective dimension"  that depends only on $d$ and is independent of  $\check{C}$. If we were only to "zoom out", this unspecified dependence on $\bar d$ renders the bound useless. To mitigate this, we need  to modify the cylinders that are far away from the point of interest. More precisely, we consider a point $x \in \check{cyl}_i$ and replace each $\cyl_j$ that does not contribute to $\Phi(x)$ by  $\check{\cyl}_j$,  a suitable translation of $$\bar{\tau}(B_d \times (\check{C} B_{n-d})).$$ This ensures that the dimension of
\beqs \{\sum_j \la_j v_j|\, \la_j \in \R,  v_j \in \check{o}_j(\R^d)\}\eeqs is bounded above by a controlled constant depending only on $d$. We then apply Lemma~\ref{lem:charlie} to the function
$ \check{F}^{\check o}(w)$ defined in (\ref{eq:fzocheck}).
This concludes the outline; we now proceed with the details.\\\\
Recall that we have fixed our attention to $\check{cyl}_i$.
Let \beqs \check{\cyl} := \bar{\tau}(B_d \times (\check{C} B_{n-d})) = \check{\cyl_i}, \eeqs where $\check{C}$ is an appropriate (large) controlled constant, whose value will be specified later.

Let \beqs \check{\cyl^2} := 2\bar{\tau}(B_d \times (\check{C} B_{n-d})) = \check{\cyl^2_i}. \eeqs
Given a  Packet $\bar o:= \{o_1, \dots, o_{\bar N}\}$, define a collection of cylinders \beqs\{\check{\cyl}_j|j \in [\check{N}]\} \eeqs in the following manner.
Let \beqs\check{S} := \left\{ j\in [\bar N]\big| |o_j(0)| < 6\bar{\tau}\right\}.\eeqs
 Let \beqs\check{T} := \left\{j\in [\bar N]\big| |\Pi_d(o_j(0))| < \check{C}\bar{\tau} \,\text{and}\,|o_j(0)| < 4\sqrt{2}\hat{C}\bar{\tau}\right\},\eeqs and assume without loss of generality that $\check{T} = [\check{N}]$ for some integer $[\check{N}]$. Here  $ 4\sqrt{2}\hat{C}$ is a constant  chosen to ensure that for any $j \in [\bar N]\setminus[\check{N}]$, $\check{\cyl^2}_j \cap \check{\cyl^2} = \emptyset$.
For $v \in \R^n$, let $Tr_v:\R^n \ra \R^n$ denote the map that takes $x$ to $x + v$. For any $j \in  \check{T} \setminus \check{S} $, let \beqs v_{j}:= \Pi_d o_j(0).\eeqs
Next, for any $j \in \check{T}$, let
\beqs\check{o}_j := \left\{\begin{array}{ll}
                                                              o_j, & \hbox{if $ \check{S}$;} \\
                                                          Tr_{v_{j}}   , & \hbox{if $j \in \check{S}\setminus \check{T} $;.}
                                                            \end{array}
                                                          \right.
\eeqs
For each $j \in \check{T}$, let $\check{\cyl}_j := \check{o}_j(\check{\cyl}).$
Define $F^{\check o}:\bigcup_{j \in \check{T}} \check{\cyl}_j \ra \R$ by

\beq \lab{eq:Fchecko}{F}^{\check o}(z) = \frac{\sum\limits_{\check{\cyl^2}_j \ni z} \big|\Pi_{n-d}(\check{o}_j^{-1}(z))\big|^2\bmp\left(\frac{\Pi_d(\check{o}_j^{-1}(z))}{2\bar \tau}\right)}{\sum\limits_{\check{\cyl^2}_j \ni z} \bmp\left(\frac{\Pi_d(\check{o}_j^{-1}(z))}{2\bar \tau}\right)}.\eeq

Taking $\oc_{12}$ to be a sufficiently small controlled constant depending on $\check{C}$, we see that
\beq
\lab{eq:fzocheck} \check{F}^{\check o}(w) :=  \frac{F^{\check{o}}( \check{C}\bar{\tau} w)}{\check{C}^2\bar{\tau}^2},\eeq restricted to $B_n$, satisfies the requirements of Lemma~\ref{lem:charlie}.
Choosing $\check{C}$ to be sufficiently large, for each $|\a| \in [2, k]$, the function $\Phi$  defined  in  (\ref{eq:zxv}) satisfies
\beq\lab{eq:phi2andmore} |\partial^\alpha \Phi| \leq c,\eeq and

for each $|\a| \in [0, k-2]$, the function $g$  defined in  (\ref{eq:zxv}) satisfies
\beq \lab{eq:g2andmore} |\partial^\alpha g| \leq c.\eeq
Observe that we can choose  $j \in [\bar N]\setminus [\check N]$ such that $|\check{o}_j(0)| < 10\tau$, and for this $j$,   $\check{\cyl}_j \cap \check{\cyl} = \emptyset$ and so
 \beq\lab{eq:thisafterphi2} \partial \Phi\big|_{({\bar\tau}^{-1})\check{o}_j(0)} = id.\eeq
The Lemma follows from Taylor's Theorem, in conjunction with (\ref{eq:phi2andmore}), (\ref{eq:g2andmore}) and  (\ref{eq:thisafterphi2}).
\begin{obs}\lab{obs:phi-inv} By choosing $\check{C} \geq 2/\oc_{11}$ we find that the domains of both $\Phi$ and $\Phi^{-1}$ may be extended to contain the cylinder $\left(\frac{3}{2}\right) B_d \times B_{n-d}$, while satisfying (\ref{eq:zxv}).
\end{obs}
\end{proof}

Since $|\partial^\a(\Phi - Id)(x)| \leq c$ for $|\a| \leq r$ and $x \in \left(\frac{3}{2}\right) B_d \times B_{n-d}$,
we have  $|\partial^\a(\Phi^{-1} - Id)(w)| \leq c$ for $|\a| \leq r$ and  $w \in  B_d \times B_{n-d}$.
For the remainder of this section, we will assume a scale where $\bar \tau = 1$.

For $u \in U_i$, we have the following equality which we restate from (\ref{eq:defsx}) for convenience.
\beqs s(u) = \sum_{j \in [\bar N]} \theta_j(u)s_j(u). \eeqs
Let $\Pi_{pseud}$ (for "pseudonormal bundle") be the map from a point $x$ in $\cyl$ to the basepoint belonging to $\MM_{put}$ of the corresponding fiber.
The following relation exists between $\Pi_{pseud}$ and $\Phi$:
$$ \Pi_{pseud} = \Phi^{-1}\Pi_d\Phi.$$
We define the $C^{k-2}$ norm of a local section $s_j$ over $U \subseteq U_j \cap U_i$ by
$$\|s_j\|_{C^{k-2}(U)} := \|s_j \circ \Phi^{-1}\|_{C^{k-2}(\Pi_d(U))}.$$

Suppose for a specific $x$ and $t$,
\beqs x + f_j(x) = t + s_j(t),\eeqs where $t$ belongs to $U_j\cap U_i$.
Applying $\Pi_{pseud}$ to both sides,
\beqs \Pi_{pseud}(x + f_j(x)) = t. \eeqs
Let \beqs \Pi_{pseud}(x + f_j(x)) =: \phi_j(x).\eeqs
Substituting back, we have
\beq\lab{eq:sj1}x + f_j(x) =  \phi_j(x) + s_j(\phi_j(x)).\eeq

By definition~\ref{def:interpolant}, we have the bound  $\|f_j\|_{C^{k-2}(\phi_j^{-1}(U_i \cap U_j))}\leq c$. We have $$\Pi_{pseud}(x + f_j(x)) = (\Pi_{pseud}- \Pi_d)(x + f_j(x)) + x,$$ which gives the bound  $$\|\phi_j - Id\|_{C^{k-2}(\phi_j^{-1}(U_i \cap U_j))}\leq c.$$
Therefore, from (\ref{eq:sj1}),
\beq \|s_j\circ\phi_j\|_{C^{k-2}(\phi_j^{-1}(U_i \cap U_j))} \leq c.\eeq

Also, \beq \|\phi_j^{-1}\circ \Phi^{-1} - Id\|_{C^{k-2}(\Pi_d(U_i\cap U_j))} \leq c.\eeq
From the preceding two equations, it follows that

\beq \|s_j\|_{C^{k-2}(U_i\cap U_j)} \leq c.\eeq

The cutoff functions $\theta_j$ satisfy
\beq \|\theta_j\|_{C^{k-2}(U_i\cap U_j)} \leq C.\eeq
Therefore,
by (\ref{eq:defsx}), \beq \|s\|_{C^{k-2}(U_i\cap U_j)} \leq C c,\eeq
which we rewrite as  \beq \lab{eq:s2}\|s\|_{C^{k-2}(U_i\cap U_j)} \leq c_1.\eeq

We will now show that $$\|f\|_{C^{k-2}(V_i)} \leq c.$$

By (\ref{eq:def-f}) in view of $\bar\tau=1$, for $u \in U_i$, there is an $x \in V_i$ such that
$$u + s(u) = x +  {f}(x).$$ 

This gives us \beqs \Pi_d(u + s(u)) = x.\eeqs
Substituting back, we have
$$\Pi_d(u + s(u)) +  {f}(\Pi_d(u + s(u))) = u + s(u).$$

Let $$\psi(u) := \Pi_d(u + s(u)).$$
This gives us \beq\lab{eq:psi1}f(\psi(u)) = (u - \psi(u)) + s(u).\eeq

By (\ref{eq:s2}) and the fact that $|\partial^\a(\Phi - Id)(x)| \leq c$ for $|\a| \leq r$, we see that \beq\lab{eq:psi}\|\psi - Id\|_{C^{k-2}(U_i)} \leq c.\eeq
By (\ref{eq:psi1}),(\ref{eq:psi}) and (\ref{eq:s2}), we have $\|f\circ \psi\|_{C^{k-2}(U_i)} \leq c.$

By (\ref{eq:psi}), we have
$$\|\psi^{-1} - Id\|_{C^{k-2}(V_i)} \leq c.$$

Therefore
\beq\lab{eq:boundalphaf}\|f\|_{C^{k-2}(V_i)} \leq c.\eeq

For any point $u \in \MM_{put}$, there is by Lemma~\ref{lem:put} for some $j \in [\bar N]$, a $U_j$ such that $\MM_{put}\cap B(u, 1/10) \subseteq U_j$ (recall that $\bar \tau = 1$). Therefore, suppose $a, b$ are two points on $\MM_{fin}$ such that $|a - b| < 1/20$, then $|\Pi_{pseud}(a) - \Pi_{pseud}(b)| < 1/10$, and so both $\Pi_{pseud}(a)$ and $\Pi_{pseud}(b)$ belong to $U_j$ for some $j$. Without loss of generality, let this $j$ be $i$. This implies that $a, b$ are points on the graph of $f$ over  $V_i$. Then, by (\ref{eq:boundalphaf}) and Proposition~\ref{thm:federer},
$\MM_{fin}$ is a manifold whose reach is at least $c \tau$.

\section{The mean-squared distance of the output manifold from a random data point}

Let $\MM_{opt}$ be an approximately optimal manifold in that
\beqs \reach(\MM_{opt}) > C\tau,\eeqs and
\beqs \vol(\MM_{opt}) < V/C,\eeqs and
\beqs \E_\PP \dist(x, \MM_{opt})^2 \leq  \inf\limits_{\MM \in \G(d, C\tau, cV)}\E_\PP \dist(x, \MM)^2 + \eps.\eeqs


Suppose that $\bar{o}$ is the packet from the previous section and that the corresponding function $F^{\bar o}$ belongs to $\asdf(\MM_{opt})$.
We need to show that the $\MM_{fin}$ constructed using $\bar{o}$ serves the purpose it was designed for; namely that the following Lemma holds.
\begin{lemma}\lab{lem:quad_loss}
\beqs \E_{x \dashv \PP} \dist(x, \MM_{fin})^2 \leq C_0\left( \E_{x \dashv \PP} \dist(x, \MM_{opt})^2 + \eps\right).\eeqs
\end{lemma}
\begin{proof}
Let us examine the manifold $\MM_{fin}$. Recall that $\MM_{fin}$ was constructed from a collection of  local sections $\{s_i\}_{i \in \bar N}$, one for each $i$ such that $o_i \in \bar o$. These
 local sections were obtained from functions $f_i:\base(\cyl_i) \ra \stalk(\cyl_i)$. The $s_i$ were patched together using a partition of unity supported on $\MM_{put}$.

Let $\PP_{in}$ be the measure obtained by restricting  $\PP$ to $\cup_{i \in [\bar N]}\cyl_i$. Let $\PP_{out}$ be the measure obtained by restricting $\PP$ to $\left(\cup_{i \in [\bar N]} \cyl_i\right)^c$. Thus,
\beqs \PP = \PP_{out} + \PP_{in}.\eeqs
For any $\MM \in \G$,
\beq \lab{eq:quad_loss:1} \E_\PP \dist(x, \MM)^2 & = & \E_{\PP_{out}} \dist(x, \MM)^2 + \E_{\PP_{in}} \dist(x, \MM)^2.\eeq

We will separately analyze the two terms on the right when $\MM$ is $\MM_{fin}$. We begin with $\E_{\PP_{out}} \dist(x, \MM_{fin})^2$.
We make two observations:
\ben \item
By (\ref{eq:bound.01}), the function $\check{f}_i$, satisfies
\beqs\|\check{f}_i\|_{L^\infty} \leq \frac{\bar{\tau}}{\tau}.\eeqs \item By Lemma~\ref{lem:zoom_out}, the fibers of the disc bundle $D^{\norm}$ over $\MM_{put} \cap \cyl_i$ are nearly orthogonal to $\base(\cyl_i)$. \een

Therefore, no point outside the union of the $\cyl_i$ is at a distance less than $\bar{\tau}(1 - \frac{2\bar{\tau}}{\tau})$ to $\MM_{fin}$.

Since $F^{\bar o}$ belongs to $\asdf(\MM_{opt})$, we see that no point outside the union of the $\cyl_i$ is at a distance less than $\bar{\tau}(1 - Cc_{12})$ to $\MM_{opt}$. Here $C$ is a controlled constant.

For any given controlled constant $c$, by choosing $\bar{c}_{11}$ (\ie $\frac{\bar \tau}{\tau}$) and $c_{12}$ appropriately, we can arrange for
\beq\lab{eq:quad_loss:2} \E_{\PP_{out}}[\dist(x, \MM_{fin})^2] \leq (1 + c) \E_{\PP_{out}}[\dist(x, \MM_{opt})^2]\eeq
to hold.

Consider terms involving $\PP_{in}$ now. We assume without loss of generality that $\PP$ possesses a density, since we can always find an arbitrarily small perturbation of $\PP$ (in the $\ell^2-$Wasserstein metric) that is supported in a ball and also possesses a density.
Let \beqs\Pi_{put}:\cup_{i \in \bar N} \cyl_i \ra \MM_{put}\eeqs be the projection which maps a point in $\cup_{i \in \bar N} \cyl_i$ to the unique nearest point on $\MM_{put}$.
Let $\mu_{put}$ denote the $d-$dimensional volume measure on $\MM_{put}$.

Let
$\{\PP_{in}^z\}_{z \in \MM_{put}}$
denote the natural measure induced on the fiber of the normal disc bundle of radius $2\bar \tau$ over $z$.

Then,
\beq\lab{eq:abov_eq1} \E_{\PP_{in}} [\dist(x, \MM_{fin})^2] = \int\limits_{\MM_{put}} \E_{\PP_{in}^z} [\dist(x, \MM_{fin})^2]d\mu_{put}(z).\eeq

Using the partition of unity $\{\theta_j\}_{j \in [\bar N]}$ supported in $\MM_{put}$, we split the right hand side of (\ref{eq:abov_eq1}).

\beq\lab{eq:abov_eq2} \int\limits_{\MM_{put}} \E_{\PP_{in}^z} [\dist(x, \MM_{fin})^2]d\mu_{put}(z) = \sum_{i \in {\bar N}} \int\limits_{\MM_{put}} \theta_i(z) \E_{\PP_{in}^z} [\dist(x, \MM_{fin})^2]d\mu_{put}(z). \eeq

For $x \in \cyl_i$, let $\N_x$ denote the unique fiber of $D^\norm$ that $x$ belongs to.
Observe that $\MM_{fin} \cap \N_x$ consists of a single point. Define $\tilde{\dist}(x, \MM_{fin})$ to be the distance of $x$ to this point, \ie \beqs \tilde{\dist}(x, \MM_{fin}) := \dist(x, \MM_{fin} \cap \N_x).\eeqs

We proceed to examine the right hand side in (\ref{eq:abov_eq2}).

By (\ref{eq:abov3})
\beqs \sum_i \int\limits_{\MM_{put}} \theta_i(z) \E_{\PP_{in}^z} [\dist(x, \MM_{fin})^2]d\mu_{put}(z) \leq \sum_i \int\limits_{\MM_{put}} \theta_i(z) \E_{\PP_{in}^z} [\tilde{\dist}(x, \MM_{fin})^2]d\mu_{put}(z).\eeqs

For each $i \in [\bar N]$, let $\MM^i_{fin}$ denote  manifold with boundary corresponding to the graph of $f_i$, \ie let
\beq\MM^i_{fin} := \left\{x +  f_i(x)\right\}_{x \in \base(\cyl)}.\eeq
Since the quadratic function is convex, the average squared "distance" (where "distance" refers to $\tilde{\dist}$) to $\MM_{fin}$ is less or equal to the average of the squared "distances" to the local sections in the following sense.
\beqs \sum_i \int\limits_{\MM_{put}} \theta_i(z) \E_{\PP_{in}^z} [\tilde{\dist}(x, \MM_{fin})^2]d\mu_{put}(z) \leq \sum_i \int\limits_{\MM_{put}} \theta_i(z) \E_{\PP_{in}^z} [\tilde{\dist}(x, \MM^i_{fin})^2]d\mu_{put}(z). \eeqs

Next, we will look at the summands of the right hand side.
Lemma~\ref{lem:zoom_out} tells us that $\N_x$ is almost orthogonal to $o_i(\R^d)$. By Lemma~\ref{lem:zoom_out}, and the fact that each $f_i$ satisfies (\ref{eq:boundalphaf}), we see that \beq\lab{eq:abov3} \dist(x, \MM^i_{fin}) \leq \tilde{\dist}(x, \MM^i_{fin}) \leq (1 + c_0)\dist(x, \MM^i_{fin}).\eeq


Therefore,
\beqs \sum_i \int\limits_{\MM_{put}} \theta_i(z) \E_{\PP_{in}^z} [\tilde{\dist}(x, \MM^i_{fin})^2]d\mu_{put}(z)\leq (1 + c_0)\sum_i \int\limits_{\MM_{put}} \theta_i(z) \E_{\PP_{in}^z} [\dist(x, \MM^i_{fin})^2]d\mu_{put}(z).\eeqs

We now fix  $i \in [\bar N]$.
Let $\PP^i$ be the measure which is obtained, by the translation via $o_i^{-1}$ of the restriction of $\PP$ to $\cyl_i$. In particular, $\PP^i$ is supported on $\cyl$.

Let $\mu^i_{base}$ be the push-forward of $\PP^i$ onto $\base(\cyl)$ under $\Pi_d$.
For any $x \in \cyl_i$, let $v(x) \in \MM^i_{fin}$ be the unique point such that $x - v(x)$ lies in $o_i(\R^{n-d})$. In particular, $$v(x) = \Pi_d x + f_i(\Pi_d x).$$ By Lemma~\ref{lem:zoom_out}, we see that
\beqs \int\limits_{\MM_{put}} \theta_i(z) \E_{\PP_{in}^z} [\tilde{\dist}(x, \MM^i_{fin})^2]d\mu_{put}(z) \leq C_0 \E_{{\PP}^i} |x - v(x)|^2. \eeqs
 Recall that $\MM^i_{fin}$ is the graph of a function $f_i:\base(\cyl) \ra \stalk(\cyl)$. In Section~\ref{sec:loc_sec}, we have shown how to construct $f_i$ so that it satisfies (\ref{eq:bound.01}) and (\ref{eq:optfi}), where $\hat{\eps} =  \frac{c\eps}{\bar{N}}$, for some sufficiently small controlled constant $c$.

\beq\lab{eq:optfi} \E_{\PP^i} |f_i(\Pi_d x) - \Pi_{n-d} x|^2 \leq \hat{\eps} + \inf\limits_{f:\|f\|_{\C^r} \leq c \bar{\tau}^{-2}} \E_{\PP^i} |f(\Pi_d x) - \Pi_{n-d} x|^2. \eeq

Let $f_i^{opt}:\base(\cyl) \ra \stalk(\cyl)$ denote the function (which exists because of the bound on the reach of $\MM_{opt}$) with the property that \beqs \MM_{opt} \cap \cyl_i = o_i\left(\{x, f_i^{opt}(x)\}_{x \in \base(\cyl)}\right).\eeqs

By (\ref{eq:optfi}), we see that
\beq\lab{eq:optfifopt} \E_{\PP^i} |f_i(\Pi_d x) - \Pi_{n-d} x|^2 \leq \hat{\eps} +  \E_{\PP^i} |f_i^{opt}(\Pi_d x) - \Pi_{n-d} x|^2. \eeq

Lemma~\ref{lem:zoom_out} and the fact that each $f_i$ satisfies (\ref{eq:boundalphaf}), and (\ref{eq:optfi}) show that
\beq\lab{eq:quad_loss:3} \E_{\PP_{in}}[\dist(x, \MM_{fin})^2] \leq C_0 \E_{\PP_{in}}[\dist(x, \MM_{opt})^2] + C_0 \hat{\eps}.\eeq
The proof follows from (\ref{eq:quad_loss:1}), (\ref{eq:quad_loss:2}) and (\ref{eq:quad_loss:3}).
\end{proof}

\section{Number of arithmetic operations}
After the dimension reduction of Section~\ref{sec:dim_red}, the ambient dimension is reduced to \beqs n:= O\left(\frac{N_p \ln^4 \left(\frac{N_p}{\eps}\right) + \log \de^{-1}}{\eps^2}\right),\eeqs
where \beqs N_p := V\left(\tau^{-d} + (\eps\tau)^{\frac{-d}{2}}\right). \eeqs The number of times that local sections are computed is bounded above by the product of the maximum number of cylinders in  a cylinder packet, (\ie $\bar N$, which is less or equal to $\frac{CV}{\tau^d}$) and the total number of cylinder packets contained inside $B_n \cap (c_{12}\tau)^{-1}\Z_n.$ The latter number is bounded above by $(c_{12}\tau)^{-n \bar N}$. Each optimization for computing a local section requires only a polynomial number of computations as discussed in Subsection~\ref{ssec:convex_complexity}. Therefore, the total number of arithmetic operations required is bounded above by
\beqs \exp\left(C\left(\frac{V}{\tau^d}\right)n\ln \tau^{-1}\right). \eeqs

\section{Conclusion}

We developed an algorithm for testing if data drawn from a distribution supported in a separable Hilbert space has an expected squared distance of $O(\eps)$ to a submanifold (of the unit ball) of dimension $d$, volume at most $V$ and reach at least $\tau$. The number of data points required is of the order of
\beqs n:= \frac{N_p \ln^4 \left(\frac{N_p}{\eps}\right) + \ln \de^{-1}}{\eps^2} \eeqs
where \beqs N_p := V\left(\frac{1}{\tau^d} + \frac{1}{\tau^{d/2}\eps^{d/2}}\right),\eeqs
and the number of arithmetic operations and calls to the black-box that evaluates inner products in the ambient Hilbert space is
\beqs \exp\left(C\left(\frac{V}{\tau^d}\right)n \ln \tau^{-1}\right). \eeqs


%
%

\section{Acknowledgements}
 We are grateful to Alexander Rakhlin and Ambuj Tiwari for valuable feedback and for directing us to material related to Dudley's entropy integral.
We thank Stephen Boyd, Thomas Duchamp, Sanjeev Kulkarni, Marina Miela, Thomas Richardson, Werner Stuetzle and Martin Wainwright for valuable discussions.

 \bibliographystyle{ACM} \bibliography{kpca_colt}

\appendix
\section{Proof of Lemma~\ref{lem:fat_to_gen}}
\begin{definition}[Rademacher Complexity]
Given a class $\F$ of functions $f:X \ra \R$ a measure $\mu$ supported on $X$, and a natural number $n \in \N$, and an $n-$tuple of points $(x_1, \dots x_n)$, where each $x_i \in X$ we define the empirical Rademacher complexity $R_n(\F, x)$
as follows.
Let $\sigma = (\sigma_1, \dots, \sigma_n)$ be a vector of $n$ independent Rademacher (\ie unbiased $\{-1, 1\}-$valued) random variables.
Then,
$$R_n(\F, x) := \E_{\sigma}\frac{1}{n}\left[\sup_{f\in \F}\left(\sum_{i=1}^n \sigma_i f(x_i)\right)\right].$$
\end{definition}

\begin{proof}
We will use Rademacher complexities to bound the sample complexity from above. We know (see Theorem $3.2$, \cite{BousquetBoucheronLugosi})
that for all $\de >0$,
\beq \lab{eq:fat_to_gen1}\p\left[\sup_{f \in \F} \bigg |\E_\mu f - \E_{\mu_s} f\bigg | \leq 2 R_s (\F, x) + \sqrt{\frac{2\log (2/\de)}{s}}\right] \geq 1- \de. \eeq
Using a ``chaining argument"
the following Claim is proved below.
\begin{claim}\lab{cl:chaining} \beq \lab{eq:fat_to_gen2} R_s(\F, x) \leq \eps + 12  \int_{\frac{\eps}{4}}^\infty \sqrt{\frac{\ln N(\eta, \F, \LL_2(\mu_s))}{s}} d\eta. \eeq\end{claim}
When $\eps$ is taken to equal $0$, the above is known as Dudley's entropy integral \cite{Dudley}.

A result of Rudelson and Vershynin (Theorem 6.1, page 35 \cite{Rudelson_Vershynin}) tells us that the integral in (\ref{eq:fat_to_gen2}) can be bounded from above using an integral involving the square root of the fat-shattering dimension (or in their terminology, combinatorial dimension.) The precise relation that they prove is
\beq\lab{eq:fat_to_gen3} \int_\eps^\infty \sqrt{\ln N(\eta, \F, \LL_2(\mu_s))} d\eta \leq C  \int_\eps^\infty \sqrt{\fat_{c\eta}(\F)} d\eta, \eeq for universal constants $c$ and $C$.

From Equations (\ref{eq:fat_to_gen1}), (\ref{eq:fat_to_gen2}) and (\ref{eq:fat_to_gen3}), we see that
if $$s \geq \frac{C}{\eps^2}\left(\left(\int_{{c\eps}}^\infty \sqrt{\fat_\gamma(\F) }d\gamma\right)^2 + \log 1/\de\right) ,$$
then,
$$\p\left[\sup_{f \in \F} \bigg| \E_{\mu_s}  f(x_i) - \E_\mu  f\bigg| \geq \eps\right] \leq 1 - \de.$$
\end{proof}
\section{Proof of Claim~\ref{cl:chaining}}
We begin by stating the finite class lemma of Massart (\cite{Massart}, Lemma 5.2).
\begin{lemma}
Let $X$ be a finite subset of $B(0, r) \subseteq \R^n$ and let $\sigma_1, \dots, \sigma_n$ be $i.i.d$ unbiased $\{-1, 1\}$ random variables. Then, we have
\beqs \E_\sigma\left[\sup_{x \in X} \frac{1}{n}\sum_{i=1}^n \sigma_i x_i\right] \leq \frac{r \sqrt{2 \ln |X|}}{n}.\eeqs
\end{lemma}
We now move on to prove Claim~\ref{cl:chaining}. This claim is closely related to Dudley's integral formula, but appears to have been stated for the first time by Sridharan-Srebro \cite{SridharanSrebro}.  We have furnished a proof following Sridharan-Srebro \cite{SridharanSrebro}.
For a function class $\FF \subseteq \RR^\mathcal{X}$ and points $x_1, \dots, x_s \in \mathcal{X}$
\beq  R_s(\F, x) \leq  \eps + 12 \int_{\frac{\eps}{4}}^\infty \sqrt{\frac{\ln N(\eta, \F, \LL_2(\mu_s))}{s}} d\eta. \eeq

\begin{proof} Without loss of generality, we assume that $0 \in \FF$; if not, we choose some function $f \in \FF$ and translate $\FF$ by $-f$.
Let $M = \sup_{f \in \FF} \|f\|_{L_2(P_n)}$, which we assume is finite. For $i\geq 1$, choose $\alpha_i = M2^{-i}$ and let $T_i$ be a  $\alpha_i$-net of $\FF$ with respect to the metric derived from $L_2(\mu_s)$.  Here $\mu_s$ is the probability measure that is uniformly distributed on the $s$ points $x_1, \dots, x_s$. For each $f \in \FF$, and $i$, pick an  $\hat{f}_i \in T_i$  such that $f_i$ is an $\alpha_i-$approximation of $f$, \ie $\|f - f_i\|_{L_2(\mu_s)} \leq \alpha_i$. We use chaining to write
\beq f = f - \hat{f}_N + \sum_{j=1}^N(\hat{f}_j - \hat{f}_{j-1}),\eeq
where $\hat{f}_0 = 0$. Now, choose $N$ such that $\frac{eps}{2} \leq M2^{-N} < \eps$,
\begin{eqnarray} \hat{R}_s(\FF) &  = & \E\left[\sup_{f\in \FF}\frac{1}{s}\sum_{i=1}^s \sigma_i \left(f(x_i) - \hat{f}_N(x_i) + \sum_{j=1}^N(\hat{f}_j(x_i) - \hat{f}_{j-1}(x_i))\right)\right]\\
& \leq & \E\left[\sup_{f\in\FF} \frac{1}{s}\sum_{i=1}^s \sigma_i(f(x_i) - \hat{f}_N(x_i))\right] + \E\left[\sup_{f\in \FF}\frac{1}{s}\sum_{i=1}^s \sigma_i(\hat{f}_j(x_i) - \hat{f}_{j-1}(x_i))\right]\\
& \leq &  \E\left[\sup_{f\in\FF} \langle \sigma,  f - \hat{f}_N\rangle_{L_2(\mu_s)})\right] +  \sum_{j=1}^N\E\left[\sup_{f\in \FF}\frac{1}{s}\sum_{i=1}^s \sigma_i(\hat{f}_j(x_i) - \hat{f}_{j-1}(x_i))\right].\lab{eq:mas1}
\end{eqnarray}
We use Cauchy-Schwartz on the first term to give
\begin{eqnarray}
 \E\left[\sup_{f\in\FF} \langle \sigma,  f - \hat{f}_N\rangle_{L_2(\mu_s)})\right]  & \leq &
 \E\left[\sup_{f\in\FF} \|\sigma\|_{L_2(\mu_s)} \| f - \hat{f}_N\|_{L_2(\mu_s)})\right] \\
& \leq & \eps.\lab{eq:mas2}
\end{eqnarray}

 Note that
\begin{eqnarray}
\|\hat{f}_j - \hat{f}_{j-1}\|_{L_2(\mu_s)} \leq \|\hat{f}_j - f - (\hat{f}_{j-1} - f)\|_{L_2(\mu_s)}& \leq & \alpha_j + \alpha_{j-1}\\
& \leq & 3 \alpha_j.
\end{eqnarray}

We use Massart's Lemma to bound the second term,
\begin{eqnarray}
  \E\left[\sup_{f\in \FF}\frac{1}{s}\sum_{i=1}^s \sigma_i(\hat{f}_j(x_i) - \hat{f}_{j-1}(x_i))\right] & = &
 \E\left[\sup_{f\in\FF} \langle \sigma,  (\hat{f}_j- \hat{f}_{j-1})\rangle_{L_2(\mu_s)}\right]\\ & \leq &
\frac{3 \alpha_j \sqrt{2\ln(|T_j|\cdot |T_{j-1}|)}}{s}\\
& \leq & \frac{6\alpha_j\sqrt{\ln(|T_j|)}}{s}.\lab{eq:mas3}
\end{eqnarray}

Now, from equations (\ref{eq:mas1}),  (\ref{eq:mas2}) and  (\ref{eq:mas3}),

\begin{eqnarray}
 \hat{R}_s(\FF) & \leq & \eps + 6 \sum_{j=1}^N\alpha_j \sqrt{\frac{\ln N(\alpha_j, \FF, L_2(\mu))}{s}}\\
& \leq & \eps + 12\sum_{j=1}^N(\alpha_j - \alpha_{j+1})\sqrt{\frac{\ln N(\alpha_j, \FF, L_2(\mu_s))}{s}}\\
& \leq & \eps +  12\int_{\alpha_{N+1}}^{\alpha_0} \sqrt{\frac{\ln N(\a, \FF, L_2(\mu_s))}{s}}d\alpha\\
& \leq & \eps + 12 \int_{\frac{\eps}{4}}^\infty \sqrt{\frac{\ln N(\a, \FF, L_2(\mu_s))}{s}}d\alpha.
\end{eqnarray}

\end{proof}

\end{document}

%% file: macros.tex
\newtheorem{theorem}{Theorem}
\newtheorem{proposition}{Proposition}
\newtheorem{lemma}[theorem]{Lemma}
\newtheorem{corollary}[theorem]{Corollary}

\newtheorem{definition}{Definition}
\newtheorem{notation}{Notation}
\newtheorem{obs}{Observation}

\newtheorem{claim}{Claim}

\newcommand{\E}{\ensuremath{\mathbb E}}
\newcommand{\R}{\ensuremath{\mathbb R}}

\newcommand{\F}{\ensuremath{\mathcal F}}
\newcommand{\FF}{\ensuremath{\mathcal F}}

\newcommand{\reach}{\mathrm{reach}}

\newcommand{\lab}{\label} \newcommand{\M}{\ensuremath{\mathcal M}} \newcommand{\ra}{\ensuremath{\rightarrow}}  \def\a{{\mathbf{\alpha}}} \def\de{{\mathbf{\delta}}}   
  \def\beq{\begin{eqnarray}} \def\eeq{\end{eqnarray}} \def\ben{\begin{enumerate}}
\def\een{\end{enumerate}}
 \def\bit{\begin{itemize}}
\def\eit{\end{itemize}}
 \def\beqs{\begin{eqnarray*}} \def\eeqs{\end{eqnarray*}} \def\bel{\begin{lemma}} \def\eel{\end{lemma}}
\newcommand{\N}{\mathbb{N}} \newcommand{\Z}{\mathbb{Z}}  \newcommand{\C}{\mathcal{C}} 
 \newcommand{\A}{\mathbb{A}}   
   \newcommand{\EE}{\mathbb{E}} \newcommand{\p}{\mathbb{P}}
\newcommand{\PP}{\mathcal P}  \newcommand{\HH}{\mathcal H}  \newcommand{\one}{\mathrm{1}}
\newcommand{\LL}{\mathcal L} \newcommand{\MM}{\mathcal M}\newcommand{\NN}{\mathcal N} \newcommand{\la}{\lambda}  
 \def\a{{\mathbf{\alpha}}}  \def\eps{{\epsilon}} \def\A{{\mathcal{A}}} \def\ie{i.\,e.\,} 
\def\vol{\mathrm{vol}}\newcommand{\tc}{{\tilde{c}}}

\renewcommand{\a}{\alpha}

\newcommand{\tix}{\tilde{x}}

\newcommand{\tiv}{\tilde{v}}

\newcommand{\tZ}{\hat{Z}}

\newcommand{\La}{\Lambda}

\newcommand{\cc}{\mathbf{c}}
\newcommand{\RR}{\mathbb{R}}
\newcommand{\dist}{\mathbf{d}}
\newcommand{\dhaus}{\mathbf{d}_{\mathtt{haus}}}
\newcommand{\G}{\mathcal{G}}
\newcommand{\TG}{\mathcal{\tilde{G}}}
\newcommand{\fat}{\mathrm{fat}}
\newcommand{\cuf}{\textsf{CurvFit}}
\renewcommand{\H}{\mathbb{H}}
\newcommand{\cscs}{{\textsf{Coreset Conditions}\,}}
\newcommand{\csone}{{\textsf{Case 1}\,}}
\newcommand{\cstwo}{{\textsf{Case 2}\,}}

\newcommand{\cyl}{{\texttt{cyl}}}
\newcommand{\tran}{{\texttt{tran}}}
\newcommand{\norm}{{\texttt{norm}}}
\newcommand{\D}{\mathcal{\bar{D}}}
\newcommand{\base}{\texttt{base}}
\newcommand{\stalk}{\texttt{stalk}}
\newcommand{\se}{\mathrm{s}}
\newcommand{\htau}{\hat{\tau}}
\newcommand{\str}{\texttt{strt}}
\newcommand{\pare}{p} \newcommand{\chil}{h}

\newcommand{\hV}{\hat V}
\newcommand{\ty}{\tilde{y}}

\newcommand{\tx}{\tilde{x}}
\newcommand{\hN}{\hat{N}}
\newcommand{\oc}{\overline{c}}

\newcommand{\oC}{\overline{C}}

\newcommand{\iN}{\mathit{N}}
\newcommand{\hess}{\,\texttt{Hess}\,}

\newcommand{\beps}{\bar{\eps}}

\newcommand{\wh}{\texttt{Wh}}
\newcommand{\bn}{\bar{n}}
\newcommand{\bN}{\bar{N}}
\newcommand{\obj}{\zeta}
\newcommand{\bmp}{\theta}
\newcommand{\hak}{\hat{k}}
\newcommand{\asdf}{{asdf}}

\newcommand{\beqn}{\begin{equation}}
\newcommand{\eeqn}{\end{equation}}